\documentclass{article}
\usepackage{mathrsfs} 
\usepackage{graphicx}
\usepackage{comment}
\usepackage{amsmath}
\usepackage{amsthm}
\usepackage{amssymb}
\usepackage[colorlinks=true,linkcolor=blue,citecolor=green,urlcolor=orange]{hyperref}
\usepackage[margin=3.5cm]{geometry}
\usepackage{stmaryrd}
\usepackage{bussproofs}
\usepackage{mathtools}
\usepackage[all]{xy}
\usepackage{todonotes}
\usepackage{forest}
\usepackage{enumitem}
\newtheorem{theorem}{Theorem}[section]
\newtheorem{proposition}[theorem]{Proposition}

\newtheorem{lemma}[theorem]{Lemma}
\theoremstyle{definition}
\newtheorem{definition}[theorem]{Definition}
\theoremstyle{remark}
\newtheorem{remark}{Remark}[section]
\newtheorem{example}{Example}[section]
\newtheorem{convention}{Convention}[section]
\newtheorem*{convention*}{Convention}

\usepackage{fancyhdr}
\providecommand{\keywords}[1]
{
  \small	
  \textbf{\textit{Keywords: }} #1
}

\usepackage{tikz}
\newcommand{\pl}{\mathtt{pl}}

\newcommand{\Pl}{\mathsf{Pl}}
\newcommand{\Prop}{\mathtt{Prop}}
\newcommand{\Sf}{\mathtt{Sf}}
\newcommand{\LIT}{\mathsf{Lit}}
\newcommand{\Prob}{\mathsf{Pr}}
\newcommand{\Bell}{\mathsf{Bel}}

\newcommand{\bel}{\mathtt{bel}}
\newcommand{\Bel}{\mathtt{Bel}}
\newcommand{\belmod}{\mathsf{B}}
\newcommand{\mainL}{{\cal{L}}}
\renewcommand{\P}{\mathcal{P}}
\newcommand{\BD}{\mathsf{BD}}
\newcommand{\ETL}{\mathsf{ETL}}
\newcommand{\N}{\mathsf{N}}
\newcommand{\DS}{\mathsf{DS}}
\newcommand{\LBD}{\mathscr{L}_\mathsf{BD}}
\newcommand{\LTBD}{\mathcal{L}_\mathsf{BD}}
\newcommand{\Langluk}{\mathscr{L}_{\Luk^2}}
\newcommand{\Langluka}{\mathscr{L}_{\Luk}}
\newcommand{\LanglukN}{\mathscr{L}_{\N\Luk}}
\newcommand{\LatBD}{\mathcal{L}_\BD}
\newcommand{\LbBD}{\LBD^*}
\newcommand{\pr}{\mathtt{p}}
\newcommand{\mass}{\mathtt{m}}

\newcommand{\f}{\varphi}
\newcommand{\p}{\psi}
\newcommand{\fDNF}{\mathsf{fDNF}}
\newcommand{\iDNF}{\mathsf{iDNF}}
\newcommand{\DNF}{\mathsf{DNF}}
\newcommand{\CNF}{\mathsf{CNF}}
\newcommand{\NNF}{\mathsf{NNF}}
\newcommand{\Luk}{{\mathchoice{\mbox{\rm\L}}{\mbox{\rm\L}}{\mbox{\rm\scriptsize\L}}{\mbox{\rm\tiny\L}}}}

\newcommand{\LsquaretopDelta}{\Luk^2_{(1,0)}(\rightarrow,\triangle)}
\newcommand{\triangletop}{{\triangle^{\top}}}
\begin{document}
\title{Reasoning with belief functions over Belnap--Dunn logic}
\author{Marta B\'{\i}lkov\'{a} \and Sabine Frittella \and Daniil Kozhemiachenko \and Ondrej Majer \and Sajad Nazari\thanks{The research of Marta B\'ilkov\'a was supported by the grant 22-01137S of the Czech Science Foundation. The research of Sabine Frittella, Daniil Kozhemiachenko, and Sajad Nazari was funded by the grant ANR JCJC 2019, project PRELAP (ANR-19-CE48-0006). This research is part of the MOSAIC project financed by the European Union's Marie Sk\l{}odowska-Curie grant No.~101007627.}}
\date{}
\maketitle
\begin{abstract}
We design an expansion of Belnap--Dunn logic with belief and plausibility functions that allows non-trivial reasoning with contradictory and incomplete probabilistic information. We also formalise reasoning with non-standard probabilities and belief functions in two ways. First, using a~calculus of linear inequalities, akin to the one presented in~\cite{FaginHalpernMegiddo1990}. Second, as a~two-layered modal logic wherein reasoning with evidence (the outer layer) utilises paraconsistent expansions of \L{}ukasiewicz logic. The second approach is inspired by~\cite{BaldiCintulaNoguera2020}. We prove completeness for both kinds of calculi and show their equivalence by establishing faithful translations in both directions.

\end{abstract}
\tableofcontents
 \section{Introduction}
Every day we have to make decisions based on various pieces of information. The information at our disposal might be unequivocal (e.g., one sees the rain outside, therefore one knows it is raining), but it might also be incomplete or contradictory. Indeed, we \emph{have no information} whether (to employ the most overused example) the decimal representation of $\pi$ contains two thousand $9$'s in a~row. On the other hand, we have \emph{both evidence for and against} the efficacy of mirror therapy in phantom pain treatment.

In the logical context, the logics that can non-trivially reason with contradictory statements are called \emph{paraconsistent} and the ones that allow incomplete information by rejecting the law of excluded middle go under the moniker \emph{paracomplete}. For our purpose, we require a~logic that is both paraconsistent and paracomplete. Ideally, this logic should explicitly differentiate between all four types of information an agent can have regarding a~statement $\phi$: that $\phi$ is only told to be true; that $\phi$ is only told to be false; that one was not told whether $\phi$ is true or false; and that one was told both that $\phi$ is true and that it is false. We choose First-Degree Entailment, or Belnap--Dunn logic ($\BD$), as our base logic because it was introduced specifically for reasoning about incomplete or inconsistent pieces of information, and it is nowadays an established formalism for that usage. In~\cite{Belnap2019}, $\BD$ was formulated as a~four-valued logic\footnote{Note, however, that in the earlier work, e.g., in~\cite{AndersonBelnap1962}, $\BD$ was presented as the first-degree fragment of the relevant logic $\mathbf{E}$, whence the other name.} with truth table semantics where each value from $\{t,b,n,f\}$ represents the information a~computer (cf.~\cite{AndersonBelnapDunn1992} for the interpretation) or a reasoning agent might have regarding a~statement.
\begin{itemize}
\item $t$ stands for ‘just told True’.
\item $f$ stands for ‘just told False’.
\item $b$ (or ‘both’) stands for ‘told both True and False’.
\item $n$ (or ‘neither’) stands for ‘told neither True nor False’.
\end{itemize}

However, the information one has is often not only just incomplete or inconsistent, but also bears a~degree of uncertainty. This is why one needs a~probability theory that accounts for contradictory and incomplete information.

There are several approaches to the paraconsistent theories of probability and uncertainty. In~\cite{BesnardLang1994}, the reasoning with possibility and necessity functions is formalised using da Costa's logic $\mathsf{C}_1$ from~\cite{DaCosta1974}. In~\cite{Bueno-SolerCarnielli2016}, a probability theory based on the logic of formal inconsistency ($\mathsf{LFI}$) which is an expansion of $\BD$ with an implication $\rightarrow$ and a \emph{consistency operator} $\circ$ is developed.

To the best of our knowledge, the earliest formalisation of probability theory in terms of $\BD$ was provided in~\cite{Mares1997}. Another formalisation is given in~\cite{Dunn2010}. In this paper, we, however, will use the probability axioms as they were given in~\cite{KleinMajerRad2021}. This is for several reasons.

First, the conditional statements do not correspond to event descriptions (and thus, the presence of an implication not reducible to $\neg$, $\wedge$, and $\vee$ is not required), whence it suffices to use $\BD$ for this purpose. This is why, $\mathsf{LFI}$ is too expressive for our purposes. Moreover, the law of excluded middle is valid in $\mathsf{C}_1$ which means that we cannot reason about incomplete information. On the other hand, $\BD$ is both paraconsistent and paracomplete.

Second, the probability of $\phi\wedge\chi$ in~\cite{Dunn2010} can be computed directly from the probabilities of $\phi$ and $\chi$. Furthermore, the probability is interpreted not as a ‘real probability’ of an event but rather as an agent's degree of certainty in the event. While this is a reasonable assumption in the classical case, one can argue (cf.~\cite{Dubois2008} for further details) that if the available information is contradictory or incomplete, the agent's certainty is, in fact, not compositional. Finally, the probabilistic axioms in~\cite{Mares1997} and~\cite{KleinMajerRad2021} are very close, with their sole distinction being that Mares postulates (axiom $\Pr1$) that the probability of the whole sample set is equal to $1$ (and, accordingly, the probability of the empty set is $0$).

Note, however, that there are no $\BD$-valid formulas, nor the formulas that are always false. Thus, $\Pr1$ does not have an immediate analogue in the language of $\BD$. This is why, we assume the probability measures defined in terms of $\BD$ to be \emph{non-normalised} by default. This is also related to the idea first proposed in~\cite{Smets1988,Smets1992} where the positive mass of the empty set was used to account for the contradictory evidence. On the other hand, the normalisation \emph{is justified} once we add the truth and falsity constants in $\BD$ (cf.~Definition~\ref{def:BDBD*}).



Many generalisations of  classical probability theory, such as inner and outer measures~\cite{Halmos1950}, belief and plausibility functions~\cite{Shafer1976}, upper and lower probabilities~\cite{Dempster1967-UpperLowerProbability}, and possibility and necessity measures~\cite{Zadeh1978,DuboisLangPrade1994} have been developed to account for the fact that an agent is not necessarily capable of assigning probabilities to all events. In fact, one may even reasonably argue that this is even less the case when one wishes to reason with contradictory evidence in a~non-trivial way. Thus, a~need for a~more general uncertainty measure arises in the same way it does in the classical probability theory.

Belief functions and possibility measures (as well as their duals --- plausibility functions and necessity measures) have a significant advantage over other generalisations of probability measures mentioned in the paragraph above since they can be directly obtained from a mass function on the sample space. In this paper, we will be using belief functions for two reasons. First, some scenarios involving contradictory evidence (cf.~Examples~\ref{example:doctorsclassical} and~\ref{example:doctorsBD}) cannot be formalised using possibility measures but can be formalised in terms of belief functions. Second, all probability measures are particular cases of belief functions which does not hold regarding the possibility measures. Thus, we can use belief functions when the agent's uncertainty is, in fact, a probability assignment.


This paper has a two-fold objective. First, we generalise belief functions over $\BD$ logic, and, algebraically, over De Morgan algebras (recall from~\cite{Font1997} that $\BD$ is the logic of De Morgan algebras). This part of our work is inspired by~\cite{Zhou2013} which provides treatment of belief functions on distributive lattices. Our goal is to expand that approach to incorporate De Morgan negation~$\neg$. Second, we  provide the calculi that formalise the reasoning with \emph{both} non-standard probabilities and belief functions. We also discuss reasoning with plausibility functions. While in the case when the background logic is classical the notions of belief and plausibility are dual and they might be interpreted as a lower and upper bound for the ‘true’ probability of the  event in question, the framework of $\BD$ logics allows for a wider range of interpretations.

This second part of our goal can be reached in two ways. The first one is by defining a~calculus that allows for the reasoning with the statements concerning probabilities or beliefs directly. This is the way it is done in~\cite{FaginHalpernMegiddo1990}. The calculi there contain three types of axioms and rules: the ones that govern the arithmetical part, i.e., the reasoning about inequalities; the ones that axiomatise probabilities; and the rules and axioms of the logic wherein the reasoning itself occurs, i.e., classical propositional logic. The proposed calculus has the advantage of being quite intuitive and easy to use, however, its axiomatisation is infinite. To address this issue, one can undertake the second approach and utilise a~two-layered modal logic in a~similar manner to~\cite{BaldiCintulaNoguera2020}. A~calculus will then consist of three parts: the rules and axioms of the logic of events or ‘inner logic’; the ‘outer logic’ that formalises reasoning with evidence; and finally, the modalities that transform events into probabilistic evidence. While these two approaches may seem different at the first glance, it is shown in~\cite{BaldiCintulaNoguera2020} that they are actually equivalent for the classical probabilities when the outer-layer logic is taken to be \L{}ukasiewicz logic. Our goal is to provide both these perspectives on reasoning with non-standard probabilities and belief functions. We will also show that they are equivalent in the same way that the two formalisations of reasoning with classical probabilities are.

\paragraph{Structure of the paper} In Section~\ref{sec:preliminaries}, we introduce necessary preliminaries on lattices, Belnap--Dunn logic, belief functions, and their mass functions. In Section~\ref{sec:uncertainty}, we present non-standard probabilities, then we discuss Dempster--Shafer theory of evidence and its application on De Morgan algebras, finally we introduce $\DS$ models on which we interpret belief and plausibility on $\BD$ logic. In Section~\ref{sec:logics}, we devise two types of two-layered logics that formalise reasoning with non-standard probabilities and belief functions. First, we construct two-layered modal logic based upon $\BD$ and paraconsistent expansions of the \L{}ukasiewicz logic where modalities are interpreted as probabilities or belief functions. In the second approach, we use linear inequalities on the outer layer. We prove completeness of both systems and establish faithful translations between them. Appendices~\ref{app:belief:plausibility:mass} and~\ref{app:proof:sec3} present the proofs of Section~\ref{sssec:belief:plausibility:mass} and~\ref{sec:uncertainty}.
\section{Preliminaries}
\label{sec:preliminaries}
\subsection{Lattices and logics}
Let us introduce several terminological and notational conventions concerning lattices and logics.
\begin{convention}[Lattices]
A \emph{bounded lattice} is a~tuple $\mainL=\langle L,\vee,\wedge,\top,\bot\rangle$, such that $\langle L,\vee,\wedge\rangle$ is a~lattice and $x\vee\top=\top$ and $x\wedge\bot=\bot$\footnote{Obviously, every finite lattice has the least and greatest elements, but we reserve the term ‘bounded lattice’ for the case when \emph{the lattice signature} contains $\top$ and $\bot$.} hold for every $x\in\mathcal{L}$. For bounded lattices, we define $\bigvee\varnothing\coloneqq\bot$ and $\bigwedge\varnothing\coloneqq\top$. For finite unbounded lattices $\mathcal{L}$, we define $\bigvee \varnothing\coloneqq\bigwedge\limits_{l\in\mathcal{L}}l$ and $\bigwedge\varnothing\coloneqq\bigvee\limits_{l\in\mathcal{L}}l$. A lattice $\mainL$ is:
\begin{itemize}
\item \emph{distributive} if $(x\vee y)\wedge z=(x\wedge z)\vee(y\wedge z)$ holds for all $x,y,z\in\mainL$;
\item \emph{complemented} if $\mathcal{L}$ is bounded and every $x\in\mainL$ has a~\emph{complement}: i.e., for every $x\in\mainL$, there exists $x'\in\mainL$, such that $x\wedge x'=\bot$ and $x\vee x'=\top$;
\item \emph{(bounded) De Morgan algebra} if it is (bounded) distributive, and equipped with an additional unary operation $\neg$ such that $\neg\neg x=x$ and $\neg(x\wedge y)=\neg x\vee\neg y$ for any $x,y\in\mathcal{L}$;
\item \emph{Boolean algebra} if it is a De Morgan algebra s.t.\ $\sim$ is its negation and ${\sim}a$ is the only complement of $a$.
\end{itemize}
Throughout the article, we denote proper De Morgan negations with $\neg$ and Boolean negations with ${\sim}$.
\end{convention}

Notice that, in bounded De Morgan algebras, it holds that $\neg\top=\bot$ and $\neg\bot=\top$. The law of excluded middle, $\neg x \vee x =\top$, and the principle of explosion, $x\wedge\neg x=\bot$, however, do not always hold.


\begin{convention}[Logics and Lindenbaum algebras]
A \emph{logic} is a~tuple $\mathsf{L}=\langle\mathscr{L},\vdash\rangle$ with $\mathscr{L}$ being a~language over $\{\circ_1,\ldots,\circ_n\}$ and $\vdash\subseteq\mathcal{P}(\mathscr{L})\times\mathscr{L}$ is \ structural, reflexive and transitive relation s.t.\ $\Gamma\vdash\phi$ entails $\Gamma'\vdash\phi$ for $\Gamma\subseteq\Gamma'$.

A~\emph{Lindenbaum algebra of $\mathsf{L}$} ($\mathcal{L}_\mathsf{L}$) is a~tuple $\langle\mathscr{L}/_{\dashv\vdash},\bullet_1,\ldots,\bullet_n\rangle$ where for each $i\in\{1,\ldots,n\}$ and each $\phi,\phi'\in\mathscr{L}$, it holds that $[\phi\circ_i\phi']=[\phi]\bullet_i[\phi']$ with $[\phi]$ being the equivalence class of $\phi$ under~$\dashv\vdash$.
\end{convention}
\subsection{Belnap-Dunn logic}\label{ssec:FDE}
In this section, we are presenting Belnap--Dunn logic ($\BD$) --- a~propositional logic over $\{\neg,\wedge,\vee\}$ and its conservative extension with constants $\top$ and $\bot$ --- $\BD^*$. 
\begin{definition}[$\BD$ and $\BD^*$: languages and calculi]\label{def:BDBD*}
We fix a~finite set $\Prop$ of propositional variables and define complex formulas via the following grammar:
\[\LbBD\ni\phi\coloneqq p\in\Prop\mid\top\mid\bot\mid\neg\phi\mid\phi\wedge\phi\mid\phi\vee\phi.\]
We will further use $\LBD$ to designate the constant-free fragment of $\LbBD$. We also define $\LIT=\Prop\cup\{\neg p:p\in\Prop\}$ and denote
\begin{align*}
\mathsf{Var}(\phi)&=\{p\in\Prop:p\text{ occurs in }\phi\}, 
&\LIT(\phi)&=\{l\in\LIT:l\text{ occurs in }\phi\}.
\end{align*}
$\BD$ can be axiomatised using the following axioms from~\cite{Prenosil2018PhD}:
\begin{align*}
\phi\wedge\chi\vdash\phi&&\phi\wedge\chi\vdash\chi&&\phi\vdash\phi\vee\chi&&
\chi\vdash\phi\vee\chi\\
\phi\vee\psi\vdash\neg\neg\phi\vee\psi&&\phi,\chi\vdash\phi\wedge\chi&&\neg\neg\phi\vee\psi\vdash\phi\vee\psi&&\phi\vee\phi\vdash\phi
\end{align*}
\begin{align*}
\phi\vee(\chi\vee\psi)\vdash(\phi\vee\chi)\vee\psi&&\phi\wedge(\chi\vee\psi)\vdash(\phi\wedge\chi)\vee(\phi\wedge\psi)&&(\phi\wedge\chi)\vee(\phi\wedge\psi)\vdash\phi\vee(\chi\wedge\psi)
\end{align*}
\begin{align*}
\neg(\phi\wedge\chi)\vee\psi\vdash(\neg\phi\vee\neg\chi)\vee\psi&&(\neg\phi\vee\neg\chi) \vee\psi \vdash\neg(\phi\wedge\chi)\vee\psi\\
\neg(\phi\vee\chi)\vee\psi\vdash(\neg\phi\wedge\neg\chi) \vee\psi&&(\neg\phi\wedge\neg\chi) \vee\psi \vdash\neg(\phi\vee\psi)\vee\psi
\end{align*}
$\BD^*$ can be axiomatised by adding the following axioms:
\begin{align*}
\varnothing\vdash\top&&\neg\top\vee\phi\vdash\phi&&\varnothing\vdash\neg\bot&&\bot\vee\phi\vdash\phi
\end{align*}

We will say that $\varphi$ and $\psi$ are \emph{equivalent}, denoted $\varphi\dashv\vdash_\BD\psi$, iff both $\varphi\vdash\psi$ and $\psi\vdash\varphi$ are derivable.
\end{definition}

There are several ways to provide semantics for $\BD$ (cf., e.g.~\cite{OmoriWansing2017}).
In addition to the already mentioned truth table semantics, one can treat $\BD$ as the logic of De Morgan algebras. Indeed, it is clear from~\cite[Proposition~2.5]{Font1997} that $\LTBD=\langle\LBD/_{\dashv\vdash_\BD},\wedge,\vee,\neg\rangle$ is the Lindenbaum algebra of $\BD$ and that $\LTBD$ is actually a~De Morgan algebra. Hence, its $\neg$-less reduct $\LTBD^+=\langle\LBD/_{\dashv\vdash_\BD},\wedge,\vee\rangle$ is a~distributive lattice. The Lindenbaum algebra of $\BD^*$ is the bounded De Morgan algebra $\LTBD=\langle\LBD/_{\dashv\vdash_{\BD^*}},\wedge,\vee,\neg\rangle$ and is denoted  $\LTBD^*$.

In this paper, we will opt for frame semantics with two valuations\footnote{The semantics that we present below is a straightforward generalisation of Dunn's ‘relational semantics’ from~\cite{Dunn1976} to the case of multiple states in the frame.}: $v^+$ and $v^-$ which are intuitively interpreted as support of truth and support of falsity. This approach will allow us to treat $\BD$ probabilistically and is also in line with its original motivation (one can think of each state as a~source that gives us information).

In this context, $w\vDash^+\phi$ can be interpreted as ‘source $w$ states that $\phi$ is true’. Note that this does not exclude the possibility of $w$ telling that $\phi$ is false as well. Neither not stating that $\phi$ is false implies that $w$ says that $\phi$ is true. One can think of $w$ as being a~database that may (or may not) have information about $\phi$. The database may for instance contain both ‘Tom's birthday is on February, 29th’ and ‘Tom's birthday is on March, 1st’ and also no information at all about whether Tom likes apple pies. If we add constants, $\bot$ represents absurdity or incoherence, a~piece of information that the agent rejects without considering it. It is important to note that a~contradiction is not an absurdity or incoherence: it is perfectly possible that a~source provides inconsistent data. Dually, $\top$ is a~piece of trivial information: the one that is accepted to be true without questions and does not provide any information. Alternatively, one can simply state that $\bot$ is the De Morgan negation of $\top$. Again, an instance of a~classical tautology, say $p\vee\neg p$, is not trivial in this framework for it is possible that a~source says nothing on $p$, nor on its negation.
\begin{definition}[$\BD$ and $\BD^*$: frame semantics]\label{def:BDframesemantics}
A Belnap--Dunn model is a~tuple $\mathfrak{M}=\langle W,v^+,v^-\rangle$ with $W\neq\varnothing$ and $v^+,v^-:\Prop\rightarrow {\mathcal{P}}(W)$. We define notions of $w\vDash^+\phi$ and $w\vDash^-\phi$ for $w\in W$ and $\phi\in\LbBD$ (positive and negative support of $\phi$ at $w$) as follows.
\begin{align*}
w\vDash^+\top&\quad w\nvDash^-\top&
w\nvDash^+\bot&\quad w\vDash^-\bot\\
w\vDash^+p&\text{ iff }w\in v^+(p)&w\vDash^-p&\text{ iff }w\in v^-(p)\\
w\vDash^+\neg\phi&\text{ iff }w\vDash^-\phi&w\vDash^-\neg\phi&\text{ iff }w\vDash^+\phi\\
w\vDash^+\phi\wedge\phi'&\text{ iff }w\vDash^+\phi\text{ and }w\vDash^+\phi'&w\vDash^-\phi\wedge\phi'&\text{ iff }w\vDash^-\phi\text{ or }w\vDash^-\phi'\\
w\vDash^+\phi\vee\phi'&\text{ iff }w\vDash^+\phi\text{ or }w\vDash^+\phi'&w\vDash^-\phi\vee\phi'&\text{ iff }w\vDash^-\phi\text{ and }w\vDash^-\phi'
\end{align*}

We denote the positive and negative interpretations of a~formula as follows:
\begin{align*}
|\phi|^+\coloneqq\{w\in W\mid w\vDash^+\phi\}&&|\phi|^-\coloneqq\{w\in W\mid w\vDash^-\phi\}.
\end{align*}

We say that a~sequent $\phi\vdash\chi$ is \emph{satisfied on $\mathfrak{M}=\langle W,v^+,v^-\rangle$} (denoted, $\mathfrak{M}\models[\phi\vdash\chi]$) iff for any $w\in W$, it holds that:
\begin{itemize}
\item if $w\vDash^+\phi$, then $w\vDash^+\chi$ as well;
\item if $w\vDash^-\chi$, then $w\vDash^-\phi$ as well.
\end{itemize}

A sequent $\phi\vdash\chi$ is \emph{valid} iff it is satisfied on every model. In this case, we will say that $\phi$~\emph{entails} $\chi$. For the sake of readability, we avoid subscripts, but we will always specify which logic ($\BD$ or $\BD^*$) we are considering.
\end{definition}
\begin{convention}
In the remainder of the paper, we will not distinguish between a~formula and its equivalence class in the Lindenbaum algebra. Therefore, we will write $|\phi| ^+$ both for the positive interpretation of the formula and for the set of states that satisfy all the formulas in the equivalence class of $\phi$. We will always specify whether $\phi$ refers to the formula or to its equivalence class.
\end{convention}

In what follows, we present a~special version of disjunctive normal forms, $X$-full DNFs (with $X$ being a~set of literals). Intuitively, an $X$-full DNF of $\phi$ lists all possible clauses over $X$ that entail $\phi$ in $\BD$. This gives a~straightforward connection to frame semantics since each state validates some finite set of literals. Furthermore, $X$-full DNFs are unique up to permutations of literals and clauses which enables their use as canonical counterparts of a~given formula. We will utilise both these traits in the completeness proofs of our calculi (cf.~Theorems~\ref{theorem:wFDEcompleteness} and~\ref{theorem:bFDEcompleteness}).

\begin{definition}[Clauses and normal forms]\label{def:normalforms}
A \emph{conjunctive (resp., disjunctive) clause} is a~conjunction (resp., disjunction) of literals and (or) constants. We define the following normal forms of the formulas in languages $\LBD$ and $\LbBD$.
\begin{itemize}
\item $\phi$ is in \emph{negation normal form} ($\NNF$) iff it does not contain any of the following subformulas: $\neg\neg\psi$, $\neg(\psi\vee\psi')$, $\neg(\psi\wedge\psi')$, $\neg\top$, $\neg\bot$;
\item $\phi$ is in \emph{disjunctive normal form} ($\DNF$) iff it is a~disjunction of conjunctive clauses;
\item $\phi$ is in \emph{conjunctive normal form} ($\CNF$) iff it is a~conjunction of disjunctive clauses.
\end{itemize}


\end{definition}



\begin{definition}[$X$-full disjunctive normal form]
\label{def:XfDNF:BD}
Let $\phi\in\LBD$ and let further $X\supseteq\LIT(\phi)$ be finite and $\neg$-closed, i.e. $$\forall p\in\mathsf{Var}:p\in X\Leftrightarrow\neg p\in X$$

A \emph{$\bigwedge$-$X$-clause} $\mathsf{cl}$ is a~non-empty subset of $X$. An \emph{$X$-full disjunctive normal form of $\phi$} ($\fDNF_X(\phi)$ or $\fDNF(\phi)$ if there is no risk of confusion) is defined as the disjunction of all $\bigwedge$-$X$-clauses entailing $\phi$ over $\mathsf{BD}$.
\[\fDNF_X(\phi)\coloneqq\bigvee\limits_{\mathsf{cl}\ \vdash_{\mathsf{BD}}\ \phi}\mathsf{cl}\]

To define $X$-full disjunctive normal form for formulas $\phi\in\LbBD$, we need  \emph{$\bigwedge$-$X$-clause} to be a~non-empty subset of $X$, $\bot$, or $\top$.
The $X$-full disjunctive normal form of $\phi$ is denoted $\fDNF^*_X(\phi)$ or $\fDNF^*(\phi)$ 
\end{definition}

The next example shows how to transform a~formula into its fDNF.
\begin{example}\label{ex:XfDNFtransformation}
Let $X=\{p,\neg p,q,\neg q\}$. Consider $p\wedge q$. Clearly, it is already in DNF. We now need to add the remaining clauses:
\begin{align*}
\fDNF_X(p\wedge q)&=(p\wedge q)\vee(p\wedge\neg p\wedge q)\vee(p\wedge q\wedge\neg q)\vee(p\wedge\neg p\wedge q\wedge\neg q)
\end{align*}
Observe that $\bigwedge X$ is always present in the $\fDNF_X$.
\end{example}
\begin{definition}[Irredundant disjunctive normal forms: $\iDNF$]\label{def:iDNF}
A conjunctive clause is \emph{irredundant} if it contains each literal at most once. A~formula $\varphi$ is in \emph{irredundant disjunctive normal} form if it is a~disjunction of irredundant conjunctive clauses, and moreover, if $\phi=\bigvee\limits_{i\in I}\phi_i$, then, for every $i,j\in I$ such that $i\neq j$, $\LIT(\phi_i)\not\subseteq\LIT(\phi_j)$.
\end{definition}
Intuitively, no clause of an $\iDNF$ implies another one. For example, $(p\wedge q)\vee(p\wedge\neg q)$ is in $\iDNF$ but $(p \wedge q) \vee p$ is not.
\subsection{Monotone functions on posets and their M\"obius transforms}
\label{sssec:belief:plausibility:mass}
In this article, we discuss the interpretation of belief and plausibility functions over De Morgan algebras. These notions were first introduced on Boolean algebras and generalised to distributive lattices. In this section, we show that the expected relations between belief, plausibility and mass functions hold on De Morgan algebras.

First, we briefly present the link between functions on posets and their M\"{o}bius transforms and introduce the notion of mass function. Then, we define (general) belief functions and (general) plausibility functions, and present some useful properties for the remainder of the paper. The proofs of this section are in Appendix~\ref{app:belief:plausibility:mass}. These results are folklore, however, in order to help the reader we provide a~sketch of proofs when we cannot provide a~reference to a~detailed proof.
%


\paragraph{M\"{o}bius transform}
It is well-known (cf.~\cite[Proposition 3.7.1]{Stanley2011}) that if $f$ is an arbitrary real-valued function on a~poset ${\cal P}=\langle P,\leq\rangle$, then there exists a~unique function $g$ on ${\cal P}$, called the \emph{Möbius transform} of $f$ such that: 
\begin{equation}\label{eq:mobiustransform}
 f(x)=\sum\limits_{y\leq x}g(y)\quad\text{iff}\quad g(x)=\sum\limits_{y\leq x}\mu(y,x)\cdot f(y)
\end{equation}
where $\mu:{\cal P}\times\mathcal{P}\to\mathbb{R}$ is the \emph{Möbius function} defined recursively as follows: 
\begin{align}
\mu(y,x)&=
\begin{cases}
1 & \text{ if }y=x,\\
-\sum\limits_{y\leq t<x}\mu(y,t)&\text{ if }y<x,\\
0 & \text{ if }y>x.
\end{cases}
\label{eq:mobiusfunction}
\end{align}




\paragraph{Belief functions and their mass functions} Let us recall the definitions of belief functions, plausibility functions and mass functions. We slightly generalize the definitions that were initially proposed in the context of Boolean algebras in order to encompass the case of De Morgan algebras. We need to do this because existing definitions and results consider belief functions on bounded lattices. In the remainder of the article, however, we study belief functions within the framework of $\BD$ which is usually considered without constants $\bot$ and $\top$. Therefore, its associated Lindenbaum algebras are unbounded De Morgan algebras.

\begin{definition}[(General) belief functions]\label{def:generalbelieffunction}
Let $\mathcal{L}$ be a~lattice. A~function $\bel:\mathcal{L}\rightarrow[0,1]$ is called a~\emph{general belief function} if the following conditions hold:
\begin{itemize}
\item $\bel$ is monotone with respect to $\mathcal{L}$, that is, for 
$x,y\in\mainL$, if $x\leq_{\mathcal{L}} y$, then $\mathtt{bel}(x)\leq\mathtt{bel}(y)$,
\item 
for every $k \geq 1$ and every $a_1,\ldots,a_k\in\mathcal{L}$, it holds that 
\begin{equation}
\label{eq:bel:k:inequality}
\mathtt{bel}\left(\bigvee\limits_{1\leq i\leq k}a_i\right)\geq\sum\limits_{\scriptsize{\begin{matrix}J\subseteq\{1,\ldots,k\}\\J\neq\varnothing\end{matrix}}}(-1)^{|J|+1}\cdot\mathtt{bel}\left(\bigwedge\limits_{j\in J}a_j\right).
\end{equation}
\end{itemize}
A general belief  function $\mathtt{bel}$ on a~bounded lattice $\mathcal{L}$ is called \emph{belief function} if $\mathtt{bel}(\bot)=0$ and $\mathtt{bel}(\top)=1$.
\end{definition}
\begin{definition}[(General) plausibility functions]\label{def:generalplausibilityfunction}
Let $\mathcal{L}$ be a lattice. $\pl:\mathcal{L}\rightarrow[0,1]$ is called a~\emph{general plausibility function} if the following conditions hold:
\begin{itemize}
\item $\pl$ is monotone with respect to $\mathcal{L}$, 
\item for every $k \geq 1$ and every $a_1,\ldots,a_k\in\mathcal{L}$, it holds that
\begin{equation}
\label{eq:pl:k:inequality}
\mathtt{pl}\left(\bigwedge\limits_{1\leq i\leq k}a_i\right)\leq\sum\limits_{\scriptsize{\begin{matrix}J\subseteq\{1,\ldots,k\}\\J\neq\varnothing\end{matrix}}}(-1)^{|J|+1}\cdot\mathtt{pl}\left(\bigvee\limits_{j\in J}a_j\right).
\end{equation}
\end{itemize}
Let $\mathcal{L}$ be a bounded lattice and $\mathtt{pl}$ a general plausibility function on $\mathcal{L}$. $\mathtt{pl}$ is called \emph{plausibility function} if $\mathtt{pl}(\bot)=0$ and $\mathtt{pl}(\top)=1$.
\end{definition}

\begin{definition}[(General) mass function]
\label{def:general:mass:function}
Let $S\neq\varnothing$. A~\emph{general mass function} on $S$ is a~function $\mass:S\rightarrow[0,1]$ such that $\sum\limits_{x\in S}\mass(x)\leq 1$. A~\emph{mass function} on $S$ is a~function $\mass:S\rightarrow[0,1]$ such that $\sum\limits_{x\in S}\mass(x)=1$.
\end{definition}


\begin{theorem}\label{theo:totallymonotone:charactrisation}
Let $\mathcal{L}$ be a~finite lattice and  $\bel:\mainL\rightarrow [0,1]$ be a~monotone function. Let further, $\mass_\bel : \mainL\rightarrow [0,1]$ be the M\"obius transform of $\bel$. Then, $\bel$ is a general belief function iff $\mass_\bel$ is a general mass function. 

If $\mathcal{L}$ is a~finite bounded lattice, then, $\bel$ is a  belief function iff $\mass_\bel$ is a mass function.

In addition, if $\bel$ is a (general) belief function, we have for every $x\in\mathcal{L}$, 
\begin{equation}
\bel(x)=\sum\limits_{y \leq x}\mass_\bel(y).
\end{equation}
We call $\mass_\bel$ the (general) mass function associated to $\bel$.
\end{theorem}
This theorem follows from \cite[Theorem~2.8]{Zhou2013} that states that $f$ is weakly totally monotone iff $g$ is non-negative. It is immediate to prove that $g$ is indeed a (general) mass function.

\begin{lemma}[Mass function associated to a (general) plausibility function] 
\label{lem:pl:associated:mass}
Let $\mathcal{L}$ be a De Morgan algebra, and $\pl  :  \mathcal{L} \rightarrow [0,1]$  a general plausibility function. Then, the function $\bel_\pl  :  \mathcal{L} \rightarrow [0,1]$ such that $\bel_\pl(x)=1-\pl(\neg x)$ is a general belief function. We denote $\mass_\pl$ the mass function associated to $\bel_\pl$ and we call $\mass_\pl$ the \emph{mass function associated to $\pl$}. Then
\begin{equation}
\pl(x)=1-\sum_{y\leq\neg x} \mass_{{\pl}}(y).
\end{equation}
Moreover, if $\pl$ is a plausibility function, then $\bel_\pl$ is a belief function.
\end{lemma}

\begin{lemma}
\label{lem:bel:pl:1-bel}
Let $\mathcal{L}$ be a~De Morgan algebra and $\bel : \mathcal{L} \rightarrow [0,1]$  a~general belief function. Then, the function $\pl : \mathcal{L} \rightarrow [0,1]$ such that $\pl(x)=1-\bel(\neg x)$ is a~general plausibility function. If $\bel$ is a~belief function, then $\pl$ is a~plausibility function.
\end{lemma}

The following lemma will be useful for the proof of theorem~\ref{th:complBelAxioms}.
\begin{lemma}
\label{lem:Zhou:3.7}
Let $\mathcal{L}$ be a~finite distributive lattice, and $\mathcal{B}_\mathcal{L}$  the Boolean algebra generated by $\mathcal{L}$. Any (general) belief function $\bel$ on $\mathcal{L}$ can be extended to a~belief function $\bel'$ on $\mathcal{B}_\mathcal{L}$ in the sense that, for any $x\in \mathcal{L}$, $\bel'(x) = \bel(x)$.
\end{lemma}
\begin{proof}
If $\bel$ is a~belief function, then we use \cite[Lemma 3.7]{Zhou2013}. Assume that $\bel$ is a~general belief function on a~finite distributive lattice $\mathcal{L}=\langle L, \vee,\wedge\rangle$. We consider $\bel^*$ the extension of $\bel$ to the distributive lattice $\mathcal{L}^*$ obtained by adding a~top and a~bottom element to $\mathcal{L}$. We define  $\bel^*(\bot)=0$ and $\bel^*(\top)=1$. This new lattice is again a~finite distributive lattice. By applying \cite[Lemma 3.7]{Zhou2013} to $\bel^*$, we obtain a~belief function $\bel'$ on $\mathcal{B}_{\mathcal{L}^*}$ such that $\bel'(x) = \bel(x)$, for every $x\in \mathcal{L}$.
\end{proof}

\section{Representations of uncertainty}\label{sec:uncertainty}
We are now ready to deal with the generalisation of uncertainty measures in the case of $\BD$. The remainder of the section is structured as follows. In Section~\ref{ssec:probabilities}, we define probabilistic $\BD$ models and recall definitions of classical and non-standard probabilities introduced within the framework of classical and Belnap Dunn logics. In Section~\ref{ssec:evidential reasoning}, we discuss the interpretation of belief and mass functions in the context of evidence-based reasoning and present Dempster's combination rule. We show a~well-known example in which it gives counterintuitive results within the framework of classical logic, and we discuss the added value of reasoning with belief functions within the framework of Belnap Dunn logic. In Section~\ref{ssec:two:dimension:belief}, we introduce $\DS$ models to interpret belief and plausibility on formulas of $\BD$ logic. Finally, we present different interpretations of belief and plausibility that lead to different generalisations of the classical definition.
Appendix~\ref{app:proof:sec3} contains the proofs of Section~\ref{sec:uncertainty}.
\subsection{Non-standard probabilities}
\label{ssec:probabilities}

Probability is the most traditional measure of uncertainty. It is usually introduced as a~measure on a~Boolean algebra, but it can also be defined as a~function on formulas of classical logic satisfying the following axioms:
\begin{itemize}
\item $\pr(\top) = 1$ (normalisation);
\item if $\varphi \vdash_{CL} \psi$  then $\pr(\varphi) \leq \mathtt{p}(\psi)$ (monotonicity);
\item $\pr(\f\lor \p) = \pr(\f) + \pr(\p)$ for $\f\land \p\equiv\bot$ (additivity).
\end{itemize}
This definition is equivalent to introducing probability on the Lindenbaum algebra of the classical propositional logic using Kolmogorov's axioms. 

There are various attempts in the literature to define probabilities on structures more general than Boolean algebras. The main purpose of introducing probability measure over $\BD$ in~\cite{KleinMajerRad2021} was to enrich the framework of Belnap--Dunn logic designed to be able to capture incomplete and/or inconsistent information with an uncertainty measure. The framework is based on the notion of a~probabilistic $\BD$ model, which is a~standard $\BD$ model equipped with a~(classical) probability measure on the set of states. 

\begin{definition}[Probabilistic $\BD$ models]
A probabilistic Belnap--Dunn model is a~tuple $\mathfrak{M}=\langle W,\mu, v^+,v^-\rangle$, such that $\langle W, v^+,v^-\rangle$ is a~$\BD$ model and $\mu: \P(W) \to [0, 1]$ is a~classical probability measure. 
\end{definition}

Probabilistic models allow for lifting the (classical) probability measure on a~set of states to probability on formulas of $\BD$ logic via their extensions: $\pr^+_\mu(\f) = \mu(|\f|^+)$, $\pr^-_\mu(\f) = \mu(|\f|^-)$. As $\pr^+_\mu$ and $\pr^-_\mu$ are related: $ \pr^-_\mu(\f) = \mu(|\f|^-) = \mu(|\neg\f|^+) = \pr_\mu(\neg\f)$, it is sufficient to work only with $\pr^+_\mu$, whence, the index can be omitted. It is shown in~\cite{KleinMajerRad2021} that the function $ \pr_\mu$ satisfies properties (i)--(iii) below. Moreover, for each function $\pr$ on the formulas of $\BD$ logic satisfying (i)--(iii) there is a~probabilistic model $\langle W,\mu, v^+,v^-\rangle$ such that $\pr(\f) = \mu(|\f|) $. This allows us to take (i)--(iii) to be an axiomatisation of probability functions over Belnap--Dunn logic, which are in \cite{KleinMajerRad2021} called non-standard probabilities.
\begin{definition} [Non-standard probability]
\label{DEF:NSprob}
A map $\pr:\LBD \rightarrow \mathbb{R}$ is a~\emph{non-standard probability} if it satisfies the following conditions:
\begin{enumerate}[label=(\roman*)]
    \item $ 0\leq \pr(x) \leq 1$ (normalisation)\label{ax:NSprob:0-1};
    \item if $\varphi \vdash_{\BD} \psi$, then $\pr(\varphi) \leq \mathtt{p}(\psi)$ (monotonicity);
    \item $\pr(\f\lor \p) = \pr(\f) + \pr(\p) - \pr(\f\wedge \p)$ \text{ (inclusion/exclusion)}. \label{ax:NSprob:import-export}
\end{enumerate}
\end{definition}

These axioms are weaker than Kolomogorovian ones and the resulting framework  behaves non-classically: probabilities of $\f$ and its negation do not sum up to $1$ any more and the probability of $\phi\wedge\neg\phi$ might be greater than $0$: probabilistic information might be incomplete and inconsistent analogously to the background system of $\BD$ logic. It also  provides us with a~continuous reading of the standard Belnap--Dunn square (Figure~\ref{fig:BDsquare}): we can see each point in the  square  as an ordered couple representing positive and negative probabilistic support assigned to a~particular proposition (Figure~\ref{fig:NSProb}). Some parts of the square suggest a~natural intuitive interpretation. The vertical line corresponds to the ‘classical’ case in the sense that the positive and negative probabilities of a~proposition sum up to $1$. The left triangle represents the area of incomplete information, while the right triangle the area of inconsistent information. The horizontal line encodes the situation when there is an equal amount of positive and negative support of the proposition. 
\begin{figure}[h]
\begin{minipage}[h]{0.45\linewidth}
\begin{center}
\begin{tikzpicture}[>=stealth,relative]
\node (U1) at (0,-1.2) {$f$};
\node (U2) at (-1.2,0) {$n$};
\node (U3) at (1.2,0) {$b$};
\node (U4) at (0,1.2) {$t$};
	
\path[-,draw] (U1) to (U2);
\path[-,draw] (U1) to (U3);
\path[-,draw] (U2) to (U4);
\path[-,draw] (U3) to (U4);
\path[->,draw] (U1) to (U4);
\path[->,draw] (U2) to (U3);
\end{tikzpicture}
\caption{Belnap--Dunn square}
\label{fig:BDsquare}
\end{center}
\end{minipage}
\hfill
\begin{minipage}[h]{0.45\linewidth}
\begin{center}
\begin{tikzpicture}[-,>=stealth,shorten >=0.5pt,auto,node distance=1.2cm,thin,
	main node/.style={circle,draw,font=\sffamily\normalsize},]
\node[main node][label=left:{$(0,0)$}] (1) {};
\node[main node][label={$(1,0)$}] (2) [above right of=1] {};
\node[main node][label=below:{$(0,1)$}] (3) [below right of=1] {};
\node[main node][label=right:{$(1,1)$}] (4) [above right of=3] {};
	
\path[every node/.style={font=\sffamily\small}]
(1) edge (2)
edge (3)
(2) edge (4)
(3) edge (4);
\path[dotted]
(2) edge (3)
(1) edge (4);
\end{tikzpicture}
\end{center}
\caption{Continuous version of Belnap--Dunn square}
\label{fig:NSProb}
\end{minipage}
\end{figure}

In the article, we do not differentiate between probabilities on $\BD$ formulas over a~given set of atomic formulas $\Prop$ and probabilities on the associated Lindenbaum algebra, that is, on the free De Morgan algebra generated by $\Prop$. The following Lemma ensures that these two notions indeed coincide.
\begin{lemma}\label{lem:correspondece:lattice:language}
 There is a~ one-one correspondence between the functions on $\LBD$ satisfying   the properties (i)--(iii) of Definition~\ref{DEF:NSprob} and the functions on the Lindenbaum algebra $\LatBD$ with the same properties. 
\end{lemma}

The following theorem~\cite[Theorem 4]{KleinMajerRad2021} shows that the axioms of non-standard probability are complete with respect to probabilistic $\BD$ models.

\begin{theorem}[Completeness of non-standard probabilities]\label{th:completeness_probablities} Let $\Prop$ be a~finite set of variables, and $\pr$  a~function satisfying the axioms in Definition~\ref{DEF:NSprob}. There is a~probabilistic  model $\mathfrak{M}=\langle W,\mu, v^+,v^-\rangle$, such that $\pr = \pr_\mu$ in the sense that $\pr(\f) = \mu(|\f|^+)$.
\end{theorem}
The construction of a~canonical model used in the proof of the previous theorem uses the fact we previously mentioned  that each formula of $\BD$ logic can be uniquely represented in \emph{irredundant disjunctive normal form} (iDNF). There is a~straightforward correspondence between $\BD$ formulas in iDNF and sets of sets of literals:\footnote{This correspondence is not one-to-one, as some of the sets correspond to a~formula in DNF, but not in iDNF.
}
$$\f = \bigvee\limits_{i} \bigwedge\limits_{j} l^i_j \to \{\{l^1_1, \dots , l^1_{n_1}\}, \dots, \{l^m_1, \dots ,  l^m_{n_m}\}\},\ l^i_j \in\  \LIT$$

In other words, each formula corresponds to a~disjunction of a~family of sets of literals interpreted conjunctively. 

\begin{definition}[Canonical model]\label{DEF:CanModel}
The canonical $\BD$ model is a~tuple $\mathscr{M}_c = \langle\P(\LIT), v^+_c, v^-_c\rangle$, where the valuations $v^+_c, v^-_c: \Prop \to \P(S_{c}) = \P(\P(\LIT))$ are defined as  $v^+_c(p) = \{s\mid p\in s\}$, $v^-_c(p) = \{s\mid \neg p\in s\}$.
\end{definition}
\begin{proof}[Proof of Theorem~\ref{th:completeness_probablities}]
Consider the canonical model $\mathscr{M}_c = \langle \P(\LIT),  v^+_c, v^-_c\rangle$ introduced above. Notice that, for every $p\in\Prop$, $v^+_c(p)$ and $v^-_c(p)$ are uppersets\footnote{Indeed, for every $s,s'\in \P(\LIT)$, if $s\models^+ p$ and $s \subseteq s'$, then $s'\models^+ p$.} in the poset $\langle\P(\LIT),\subseteq\rangle$. The positive extension of a~formula $\phi$ in iDNF, $\phi = \bigvee_{i=1}^n \gamma_i $ for some conjunctions of literals $\gamma_i$ is the set 
$|\phi|^+ = \{s\mid s\models\phi\} = \{s\mid s\supseteq \LIT(\gamma_i)\text{ for some } i, 1\leq i \leq n \}$. 
Thus, extensions of formulas are uppersets in the poset $\langle\P(\LIT),\subseteq\rangle$, the sets of literals $\gamma_i$ generate the upperset $|\varphi|^+$ and in fact, they are the minimal set of its generators.
This correspondence is one-to-one: each extension is an upperset, and each upperset (other than $\varnothing$ and $\P(\LIT)$)\footnote{Notice that $\varnothing$ and $\P(\LIT)$ are the extensions of $\bot$ and $\top$.} is a~positive extension (of the formula given in iDNF using the finite antichain of the generators of the upperset).  Moreover, the mapping $\f\mapsto|\f|^+$ is an isomorphism of both structures understood as distributive lattices ($|\f\lor \p|^+ = |\f|^+\cup |\p|^+$ and $|\f\land \p|^+ = |\f|^+\cap |\p|^+$), hence a~non-standard probability function $\pr$ on formulas defines a~non-standard probability function $\pr'$ on uppersets (other than $\varnothing$ and $\P(\LIT)$) as $\pr'(|\f|^+) = \pr(\f)$. We can extend $\pr'$ to $\varnothing$ and $\P(\LIT)$ in the obvious way: $\pr'(\varnothing)=0$ and $\pr'(\P(\LIT))=1$ and we obtain what is in \cite{Zhou2013} called probability function on a~distributive lattice.
Then we use Lemma 3.5 in \cite{Zhou2013}, which says that a~probability function on a~distributive lattice $\mathcal{L}$  can be uniquely extended to a~probability function on the Boolean algebra $B_\mathcal{L}$ generated by the lattice $\mathcal{L}$.

Let us note that in \cite{KleinMajerRad2021} an alternative proof is provided. It uses the fact that the required probability measure $\mu$ on the canonical model is generated by its values on singletons $\{s\}, s\in \P(\LIT)$. As the sets of literals are ordered by inclusion with the maximal element being the set corresponding to the conjunction of all the literals, we can define the measure on singletons inductively. We start with the conjunctive clause $\gamma_{max} = \bigwedge\limits_{l \in \LIT}l, \LIT = \{1, \dots, n\}$ and assign $\mu(\{{l \mid l \in \LIT}\}) = \pr(\gamma_{max})$. In the induction step for $s = \{  l_1, \dots ,l_k \}, k \leq n$, we define  $ \mu(\{s\}) = \pr\left(\bigwedge\limits_{i = 1\dots k}l_i\right) - \sum\limits_{s \subset s'} \mu(\{s'\})$.
\end{proof}
\subsection[Evidential reasoning]{Evidential reasoning via mass functions on algebras}
\label{ssec:evidential reasoning}

The classical treatment of probability has two distinctive traits. First, the probability is assumed to be ‘compositional’, in the sense the probability measure of any given event is uniquely determined by the probabilities of elementary events. Second, and related to the first, is that probabilities of all events are assumed to be known (or at least, knowable). Formally, these assumptions lead to sample spaces being Boolean algebras.

It may be reasonably argued that these assumptions are too optimistic and do not correspond to the situations one encounters in practice. Indeed, given the probability assignments of some elementary events $a_1$, \ldots, $a_m$, one may not be able to infer probabilities of their combinations if said assignments were obtained by different methods (i.e., the data were heterogeneous). On a~related note,  it is not necessarily the case that the probability of all elementary events is known even if one somehow obtained an assignment for a~complex event composed of those.

Taking that into account, one can generalize the classical approach to probability in two (compatible) ways. First, given a~Boolean algebra $\mathcal{B}$, one can define the probability assignment on its proper subalgebra $\mathcal{B}'$. The values of the events in $\mathcal{B}\setminus\mathcal{B}'$ can be then approximated via more general uncertainty measures, e.g., belief and plausibility functions or inner and outer measures. The other approach is to represent the sample space not as a~Boolean algebra but in the form of another, more general structure.

\paragraph{Belief and plausibility functions}
Belief functions were introduced in~\cite{Shafer1976} as a~generalisation of probabilities for the case where the exact compositional uncertainty measures are not given to the entire sample space of events. Originally, they were defined on Boolean algebras, however, later work~\cite{Barthelemy2000,Grabisch2009,Zhou2013,FritellaManoorkarPalmigianoTzimoulisWijnberg2020} saw them further expanded on arbitrary and distributive lattices. 
In this section, we will use a~combination of the two approaches given above and consider belief (and their dual counterparts, plausibility functions) on De Morgan algebras of which Boolean algebras are a~particular case and which themselves are a~special case of distributive lattices equipped with negation.\\
In the standard approach both belief and plausibility use in fact the same information represented by the mass function, but deal with it in a~different way. While we can see belief as the amount of information which directly supports the statement in question, plausibility represents the amount of information which does not contradict the statement. As Halpern~\cite[P.38]{Halpern2017} says: ‘$\pl_m(U)$ can be thought of as the sum of the probabilities of the evidence that is compatible with the actual world being in $U$’. This idea is captured in the definition of plausibility via mass function: $\pl(A) = \sum\limits_{A\cap B \neq \varnothing} \mass(B)$. Alternatively, we can understand plausibility as a~measure of the information, which does not support the negation of the hypothesis: $\pl(A) =1-\bel( A^c) = \sum\limits_{B \not\subseteq A{^c}} \mass(B)$. We can also see belief and plausibility as approximations, as a~lower and an upper bound for the ‘true’ probability: $\bel(A)\leq \pr(A)\leq \pl(A)$. Although in the classical case all these readings coincide, in the case of $\BD$ logic they do not, which gives us several possibilities of defining belief/plausibility pairs.
\paragraph{Dempster's combination rule on powerset algebras}
Dempster--Shafer theory \cite{Dempster1968,Shafer1976} 
is a~formal framework for decision-making under uncertainty in situations in which some propositions cannot be assigned probabilities. The core proposal of Dempster--Shafer theory is that, in such cases, the missing value can be replaced by a~range of values, the lower and upper bounds of which are assigned by belief and plausibility functions. 
In fact, the correspondence between belief functions and mass functions is used to formalise probabilistic reasoning based on pieces of evidence. A~mass function is assigned to each piece of evidence to encode the information contained in the evidence. For instance, if an expert states that they are 70\% certain that $p \vee q$ is true and that they do not give more information. One would assign the following mass function to this piece of evidence: $\mass(p\vee q)=0.7$ and $\mass(\top)=0.3$. Here, the remaining mass is assigned to $\top$, because it represents the non-informative statement. In the classical case, mass functions, belief functions and plausibility functions are connected via the following interpretation.
While mass function represents the amount of evidence committed exactly to a~particular statement, we can see belief as collecting information which directly supports the statement in question, while  plausibility represents the amount of information which does not contradict the statement. 
Belief (resp., plausibility) is given by the sum of masses of the propositions implying (resp., not contradicting) it.
One can already observe that belief and  plausibility are connected via the negation and the notion of contradiction. Therefore, shifting from  classical logic to $\BD$ logic will impact the definition of plausibility and the connection between belief and plausibility. In fact, this will open the door to many alternative definitions of plausibility.

Since, a~priori, a~mass function is assigned to each piece of evidence, the natural next step is to define a~way to combine the information obtained from each piece of evidence. In what follows, we discuss Dempster's combination rule and its interpretation on powerset algebras. Then, we motivate interpreting it on De Morgan algebras to model more accurately and in a~more informative manner situations in which one handles contradictory evidence.
\begin{definition}[Dempster's combination rule over a~powerset algebra] Let $\mass_1$ and $\mass_2$ be two mass functions on a~powerset algebra $\mathcal{P}(S)$. Dempster's combination rule computes their aggregation $\mass_{1\oplus 2}$ as follows.
\begin{align}
\label{eq:combination:rule:cl}
\mass_{1 \oplus 2} : \P(S) & \rightarrow [0,1]
\\
X & \mapsto 
\left\{
 \begin{aligned}
 \; &0 & \mbox{if } X=\varnothing \\
 \; &\frac{\sum \{ \mass_1(X_1) \cdot \mass_2(X_2) \mid X_1 \cap X_2 = X \} }{\sum \{ \mass_1(X_1) \cdot \mass_2(X_2) \mid X_1 \cap X_2 \neq \varnothing \} } 
 & \mbox{otherwise.}
 \end{aligned}
\right.
\notag
\end{align}
\end{definition}

\begin{example}[Two disagreeing experts. Classical reasoning]\label{example:doctorsclassical} A~patient is sick, and two doctors are asked their opinions about which disease the patient has. Three diseases are being considered: $S=\{a,b,c\}$. It is assumed that the experts are infallible, and that the patient can have one and only one of the considered diseases. Therefore, the events $a$, $b$ and $c$ are incompatible and exhaustive.
Expert 1 thinks that the patient has disease $a$ with certainty $0.9$, disease $b$ with certainty $0.1$, and that it is impossible that they have disease $c$, therefore assigning probability $0$ to that option.
Expert 2 thinks that the patient has disease $c$ with certainty $0.9$, disease $b$ with certainty $0.1$, and that it is impossible that they have disease $a$, therefore assigning probability $0$ to that option.

The opinion of expert 1 is described by the mass function $\mass_1  :  \P(\{a,b,c\}) \rightarrow [0,1]$ such that $\mass_1(\{a\})=0.9$, $\mass_1(\{b\})=0.1$, and $\mass_1(x)=0$ otherwise.
The opinion of expert 2 is described by the mass function $\mass_2  :  \P(\{a,b,c\}) \rightarrow [0,1]$ such that $\mass_2(\{b\})=0.1$, $\mass_2(\{c\})=0.9$, and $\mass_2(x)=0$ otherwise.

The aggregated mass function $\mass_{1\oplus 2}$, using Dempster's combination rule, is as follows: 
$$
\mass_{1\oplus 2}(x) = \left\{
 \begin{array}{ll}
 1 & \mbox{if } x=\{b\} \\
 0 & \mbox{otherwise.}
 \end{array}
\right.
$$
We get $\mass_{1\oplus 2}(\{a\})=0$ because for any two elements $x,y \in \P(\{a,b,c\})$ such that $x \cap y = \{a\}$ we have $\mass_1(x) \cdot \mass_2(y)=0$. This result comes from the fact that Dempster's combination rule simply gets rid of contradiction.

This conclusion  makes sense given the hypothesis. Indeed, one can, for instance, consider a situation where Doctor 1 (resp., Doctor 2) is an expert in disease $c$ (resp., $a$), then when they say $c$ (resp., $a$) is impossible, it must be the case. In that situation, the only reasonable conclusion is that the patient has disease $b$. This is discussed in more detail in~\cite{DuboisPrade1985}. However, in many situations, things are not that clear, and experts are not 100\% reliable, that is, the hypothesis of Dempster's rule is not true and this rule does not give the expected conclusions. Indeed, it could be more reasonable to conclude that there is 50\%-50\% that the patient has disease $a$ and/or disease $c$ and that it is very unlikely that it is disease $b$, because both experts agree on that fact.

Notice that if one decides to assign a~very small mass (e.g.\ $10^{-4}$) instead of $0$ for $\mass_1(\{c\})$ and $\mass_2(\{a\})$, one gets 
the following mass functions 
\begin{align*}
 \mass_1 : \P(S) &\rightarrow [0,1]
 &
 \mass_2 : \P(S) &\rightarrow [0,1]
 \\
 x & \mapsto \left\{
 \begin{array}{ll}
 0.89995 & \mbox{if } x=\{a\} \\
 0.09995 & \mbox{if } x=\{b\} \\
 0.0001 & \mbox{if } x=\{c\} \\
 0 & \mbox{otherwise.}
 \end{array}
 \right.
 &
 x & \mapsto \left\{
 \begin{array}{ll}
 0.0001 & \mbox{if } x=\{a\} \\
 0.09995 & \mbox{if } x=\{b\} \\
 0.89995 & \mbox{if } x=\{c\} \\
 0 & \mbox{otherwise.}
 \end{array}
 \right.
\end{align*}
and the following aggregated mass function:\footnote{Note that those are rounded numbers.}
$$
\mass_{1\oplus 2}(x) = \left\{
 \begin{array}{ll}
 0.00885 & \mbox{if } x=\{a\} \\
 0.9823 & \mbox{if } x=\{b\} \\
 0.00885 & \mbox{if } x=\{c\} \\
 0 & \mbox{otherwise.}
 \end{array}
\right.
$$
That is one still concludes that disease $b$ is way more likely than disease $a$ or $c$. Indeed, one gets $\bel_{1\oplus 2} (b) = 0.9823$ and $\bel_{1\oplus 2}(\{a,c\})=\mass_{1\oplus 2}(\{a\}) + \mass_{1\oplus 2}(\{c\})=0.0177$.
\end{example}

The example above shows how Dempster's combination rule can give counterintuitive results when applied to situations that do not comply with Dempster's hypothesis. Several modifications of this rule have been proposed and studied in the literature to aggregate evidence both from not fully reliable sources and from sources strongly contradicting each other. {\em Discounting} or {\em tradeoff} method to deal with conflict is described in~\cite{Shafer1976}. When the sources have a conflict between them, the analyst discounts sources based on their reliability before using Dempster's combination rule.

Another combination rule is proposed in~\cite{Yager1987}. It is similar to the  Dempster's but the mass attached to conflicting evidence is assigned to the whole frame of discernment. That is, having conflicting evidence is considered equivalent to having no information.

The non-normalized version of Dempster's rule allows for the mass of the empty set to be non-zero. The mass of the empty set can be interpreted as the amount of contradiction between the two sources. In the open-world context, belief functions are not necessarily normalized~\cite{Smets1988,Smets1992} and the mass of the empty set can be interpreted as evidence indicating an unexpected outcome.

Another option~\cite{DuboisPrade1988} is that if two sources attach mass to disjoint sets $A$ and $B$, then in the combination the mass $\mass(A)\cdot\mass(B)$ is attached to the set $A \cup B$. Intuitively this corresponds to the idea that if sources are contradictory, then the analyst concludes that at least one of them is correct. In what follows we argue that reasoning within the framework of $\BD$, that is, on De Morgan algebras, allows for a more detailed description of the available evidence --- especially regarding the contradictory pieces of evidence --- and to treat more diverse situations with the close-world assumption and without needing to evaluate the expertise and reliability of the sources.
\paragraph{Dempster's combination rule on De Morgan algebras} Recall that $\BD$ logic is the logic of De Morgan algebras, therefore, in that framework, saying that a~formula $\phi$ is true means that ‘we have information supporting that fact that the statement $\phi$ is true’. Therefore, as we have already mentioned in the introduction, there are no two formulas $\phi,\chi\in\LBD$ s.t.\ $\phi\wedge\chi\vdash\psi$ is $\BD$-valid \emph{for every $\psi\in\LBD$}. Indeed, one can have pieces of information supporting contradictory statements. Therefore, if we consider De Morgan algebras, we get the following adaptation of Dempster's combination rule.
\begin{definition}[Dempster's combination rule over a~De Morgan algebra]
\label{def:combination:rule:BD}
Let $\mathcal{L}$ be a~De Morgan algebra (without the constants $\bot$ and $\top$ in the language).
Let $\mass_1$ and $\mass_2$ be two general mass functions on $\mathcal{L}$. Dempster's combination rule computes their aggregation $\mass_{1\oplus 2}$ as follows.
\begin{align}
\label{eq:combinatin:rule:BD}
\mass_{1 \oplus 2} : \mathcal{L} & \rightarrow [0,1]
\\
x & \mapsto \sum \{ \mass_1(x_1) \cdot \mass_2(x_2) \mid x_1 \wedge x_2 = x \}.
\notag
\end{align}

Let $\mathcal{L}$ be a~bounded De Morgan algebra (that is, with the constants $\bot$ and $\top$ in the language).
Let $\mass_1$ and $\mass_2$ be two mass functions on $\mathcal{L}$. Dempster's combination rule computes their aggregation $\mass_{1\oplus 2}$ as follows.
\begin{align}
\label{eq:combinatin:rule:BD:bot}
\mass_{1 \oplus 2} : \mathcal{L} & \rightarrow [0,1]
\\
x & \mapsto 
\left\{
 \begin{aligned}
 \; &0 & \mbox{if } x=\bot \\
 \; &\frac{\sum \{ \mass_1(x_1) \cdot \mass_2(x_2) \mid x_1 \wedge x_2 = x \} }{\sum \{ \mass_1(x_1) \cdot \mass_2(x_2) \mid x_1 \wedge x_2 \neq \bot \} } 
 & \mbox{otherwise.}
 \end{aligned}
\right.
\notag
\end{align}
\end{definition}

Notice that in equation \eqref{eq:combinatin:rule:BD}, there is no normalisation term. Indeed, here $$\sum_{x\in \mathcal{L}}\mass_{1 \oplus 2}(x) = \sum \{ \mass_1(x_1)\cdot\mass_{2}(x_2) \mid x_1,x_2 \in \mathcal{L}\}$$ In addition, we have
\begin{align*}
 \sum_{x\in \mathcal{L}}\mass_{1 \oplus 2}(x) 
 & = \sum_{x\in \mathcal{L}}\left(\sum\limits_{\scriptsize{\begin{matrix}x_1,x_2\in \mathcal{L}\\x_1\wedge x_2=x\end{matrix}}}\mass_1(x_1)\cdot \mass_2(x_2)\right)\\
 & = \sum\limits_{\scriptsize{\begin{matrix}x_1,x_2\in \mathcal{L}\\x_1\wedge x_2=x\end{matrix}}}\mass_1(x_1)\cdot \mass_2(x_2)
 \\
 & = \sum_{x_1\in \mathcal{L}} \sum_{x_2\in\mathcal{L}} \mass_1(x_1)\cdot\mass_{2}(x_2)
 \\
 & = \sum_{x_1\in \mathcal{L}} \mass_1(x_1) \cdot \sum_{x_2\in\mathcal{L}} \mass_{2}(x_2).
\end{align*}
Therefore, $\mass_{1 \oplus 2}$ is a~general mass function, because $0 \leq \sum\limits_{x_1\in\mathcal{L}} \mass_{1}(x_1)\leq 1$ and $0 \leq \sum\limits_{x_2\in\mathcal{L}} \mass_{2}(x_2)\leq 1$.

\begin{lemma}
In the case of Dempster's combination rule for 
bounded De Morgan algebras, if we consider the free De Morgan algebra generated by a~finite set of variables $\Prop$ and constants $\{\bot,\top\}$, then we have that for every $x\in \mathcal{L}$,
 \begin{align}
 \mass_{1 \oplus 2} (x)&= \sum\limits_{\scriptsize{\begin{matrix}x_1,x_2\in \mathcal{L}\\x_1\wedge x_2=x\end{matrix}}}\mass_1(x_1)\cdot \mass_2(x_2)
 \end{align}
\end{lemma}
\begin{proof}
In equation \eqref{eq:combinatin:rule:BD:bot}, notice that $x_1 \wedge x_2 = \bot$ iff either $x_1=\bot$ or $x_2=\bot$. This implies that
\begin{align*}
 \sum\limits_{x_1 \wedge x_2 \neq \bot}\mass_1(x_1) \cdot \mass_2(x_2)& = \sum_{x_1\in \mathcal{L}\smallsetminus{\bot}}
 \sum_{x_2\in \mathcal{L}\smallsetminus{\bot}} \mass_1(x_1)\cdot \mass_2(x_2)
 \\
 & = \sum_{x_1\in \mathcal{L}}
 \sum_{x_2\in \mathcal{L}} \mass_1(x_1)\cdot \mass_2(x_2)
 \tag{because $\mass_1(\bot)=\mass_2(\bot)=0$}
 \\
 & = \sum_{x_1\in \mathcal{L}} \mass_1(x_1)\cdot 
 \sum_{x_2\in \mathcal{L}}  \mass_2(x_2)
 \\
 & = \sum_{x_1\in \mathcal{L}} \mass_1(x_1)
 \tag{because $\sum_{x_2\in \mathcal{L}}  \mass_2(x_2)=1$}
 \\
 & = 1.
 \tag{because $\sum_{x_1\in \mathcal{L}} \mass_1(x_1)=1$}
\end{align*}
Therefore, if $x \neq \bot$, we have $\mass_{1 \oplus 2} (x) = \sum\limits_{x_1 \wedge x_2 = x}\mass_1(x_1) \cdot \mass_2(x_2)$. If $x=\bot$, then $\sum\limits_{x_1 \wedge x_2 = x}\mass_1(x_1) \cdot \mass_2(x_2)=0$, because either $x_1=\bot$ or $x_2=\bot$.
\end{proof}
\begin{remark}
In Section~\ref{ssec:two:dimension:belief}, we introduce $\DS$ models (Definition~\ref{def:DS:models}) on which we define belief functions over $\BD$ logic. Notice that we define $\bel^+$ and $\bel^-$, that generalise the non-standard probabilities $\pr^+$ and $\pr^-$, on the Lindenbaum algebra associated to the model. Therefore, we consider belief functions over free De Morgan algebras.

Note furthermore, that the scenario in Example~\ref{example:doctorsclassical} (as well as its $\BD$ modification in the following Example~\ref{example:doctorsBD}) cannot be treated via a possibility or necessity measure since the mass functions are not concordant. But a possibility (or necessity measure) can be generated only by a concordant\footnote{Mass function $\mass$ is \emph{concordant} on a lattice $\mathcal{L}$ iff $\mass(x),\mass(x')>0$ entails that $x$ and $x'$ are comparable w.r.t.~$\leq_\mathcal{L}$.} mass function~\cite[Theorem~2.7.4]{Halpern2017}.
\end{remark}
\begin{example}[Two disagreeing experts. Reasoning with $\BD$ logic]\label{example:doctorsBD}
We consider the previous case study, but we consider the general mass functions over the free De Morgan algebra $\mathcal{A}_{a,b,c}$ generated by $\{a,b,c\}$:
\begin{align*}
 \mass_1 : \mathcal{A}_{a,b,c} &\rightarrow [0,1]
 &
 \mass_2 : \mathcal{A}_{a,b,c} &\rightarrow [0,1]
 \\
 x & \mapsto \left\{
 \begin{array}{ll}
 0.9 & \mbox{if } x=a \\
 0.1 & \mbox{if } x=b \\
 0 & \mbox{otherwise.}
 \end{array}
 \right.
 & 
 x & \mapsto \left\{
 \begin{array}{ll}
 0.9 & \mbox{if } x=c \\
 0.1 & \mbox{if } x=b \\
 0 & \mbox{otherwise.}
 \end{array}
 \right.
\end{align*}
We get the following aggregated mass function
$$
\mass_{1\oplus 2}(x) = \left\{
 \begin{array}{ll}
 0.81 & \mbox{if } x=a \wedge c \\
 0.09 & \mbox{if } x=a \wedge b \text{ or } x= b\wedge c\\
 0.01 & \mbox{if } x=b \\
 0 & \mbox{otherwise.}
 \end{array}
\right.
$$
Here, one still has $\mass_{1\oplus 2}(a)=\mass_{1\oplus 2}(c)=0$, but $\mass_{1\oplus 2}(b)=0.01$ is now small. In addition 
$$\bel_{1\oplus 2}(a) = \mass_{1\oplus 2}(a) + \mass_{1\oplus 2}(a\wedge b) + \mass_{1\oplus 2}(a \wedge c) = 0.9$$ 
is $4.7$ times larger than 
$$\bel_{1\oplus 2}(b) = \mass_{1\oplus 2}(b) + \mass_{1\oplus 2}(a\wedge b) + \mass_{1\oplus 2}(b \wedge c) =0.19.$$
Here, the mass function $\mass_{1\oplus 2}$ tells us that the evidence strongly supports the fact that the patient has disease $a$ and $c$, and that the evidence is less conclusive concerning disease $b$.

One could object that in the classical case, it is assumed that it is impossible for the patient to have two diseases, in which case, one might want to formalise the example with the following mass functions:
\begin{align*}
 \mass_1 : \mathcal{A}_{a,b,c} &\rightarrow [0,1]
 &
 \mass_2 : \mathcal{A}_{a,b,c} &\rightarrow [0,1]
\\
x & \mapsto \left\{
\begin{array}{ll}
0.9 & \mbox{if } x=a\wedge\neg b \wedge \neg c \\
0.1 & \mbox{if } x=\neg a~\wedge b \wedge \neg c \\
0 & \mbox{otherwise.}
\end{array}
\right.
& 
x & \mapsto \left\{
\begin{array}{ll}
0.9 & \mbox{if } x=\neg a~\wedge \neg b \wedge c \\
0.1 & \mbox{if } x=\neg a~\wedge b \wedge \neg c \\
0 & \mbox{otherwise.}
\end{array}
\right.
\end{align*}
This gives us the following aggregated mass function:
$$
\mass_{1\oplus 2}(x) = \left\{
\begin{array}{ll}
0.81 & \mbox{if } x=a \wedge \neg a~\wedge \neg b \wedge c \wedge \neg c\\
0.09 & \mbox{if } x=a \wedge \neg a~\wedge b \wedge \neg b \wedge \neg c \text{ or } x= \neg a~\wedge b \wedge \neg b \wedge c \wedge \neg c\\
0.01 & \mbox{if } x=\neg a~\wedge b \wedge \neg c \\
0 & \mbox{otherwise.}
\end{array}
\right.
$$
Here, the mass function highlights that (1) experts agree the patient does not have disease $b$ and (2) the agent has contradictory information which might lead to the conclusion that further investigation is necessary. Indeed, if one asks two equally qualified experts about their opinions and if they contradict each other, it is only natural to consult more experts. In the same time, we still have
$\bel_{1 \oplus 2}(a)=\bel_{1 \oplus 2}(c)=0.9$ and $\bel_{1 \oplus 2}(b)=0.19$.
In addition, we can also describe in detail the contradictory information available: $\bel_{1\oplus 2}(a \wedge \neg a)=\bel_{1\oplus 2}(a \wedge \neg a)=0.9$ and $\bel_{1\oplus 2}(b \wedge \neg b)=0.18$.

This framework also has the advantage to allow us to formalise a~situation in which both experts did not consider the same set of eventualities. Assume that expert 1 simply did not consider disease $c$ as an option (because they forgot, because they are not aware of it, because they could not test for it...) and expert 2 did not consider disease $a$ as an option. Then the initial mass functions become:
\begin{align*}
 \mass_1 : \mathcal{A}_{a,b,c} &\rightarrow [0,1]
 &
 \mass_2 : \mathcal{A}_{a,b,c} &\rightarrow [0,1]
 \\
 x & \mapsto \left\{
 \begin{array}{ll}
 0.9 & \mbox{if } x=a\wedge\neg b \\
 0.1 & \mbox{if } x=\neg a~\wedge b \\
 0 & \mbox{otherwise.}
 \end{array}
 \right.
 & 
 x & \mapsto \left\{
 \begin{array}{ll}
 0.9 & \mbox{if } x= \neg b \wedge c \\
 0.1 & \mbox{if } x= b \wedge \neg c \\
 0 & \mbox{otherwise.}
 \end{array}
 \right.
\end{align*}
and we get the aggregated mass function
$$
\mass_{1\oplus 2}(x) = \left\{
 \begin{array}{ll}
 0.81 & \mbox{if } x=a \wedge \neg b \wedge c \\
 0.09 & \mbox{if } x=a \wedge b \wedge \neg b \wedge \neg c \text{ or } x= \neg a~\wedge b \wedge \neg b \wedge c \\
 0.01 & \mbox{if } x=\neg a~\wedge b \wedge \neg c \\
 0 & \mbox{otherwise.}
 \end{array}
\right.
$$
Here, the conclusion is that the patient is likely to have diseases $a$ and $c$.
\end{example}

In Section~\ref{ssec:two:dimension:belief}, we present several interpretations of general belief and plausibility functions introduced in~Definitions~\ref{def:generalbelieffunction} and~\ref{def:generalplausibilityfunction} in the $\BD$ framework. We choose general belief (plausibility) functions instead of the usual ones because $\BD$ does not have valid formulas. Furthermore, even though, given a finite set $X$ of literals, their conjunction entails every other $\LBD$-formula and thus has \emph{the lowest} belief or plausibility, it is still not incoherent or absurd and thus, one cannot a~priori assume that $\bel\left(\left|\bigwedge\limits_{l\in X}l\right|^+\right)=0$ since it is possible that this is exactly what the sources are telling the agent.

Dubois and Prade's rule~\cite{DuboisPrade1988} is a~‘hybrid’ rule intermediate between the conjunctive and disjunctive sums, in which the product $\mass_1(B)\cdot \mass_2 (C)$ is assigned to $B \cap C$ whenever $B \cap C \neq \varnothing $, and to $B \cup C$ otherwise. This rule is not associative, but it usually provides a~good summary of partially conflicting items of evidence. This rule deals with contradictory evidence by stating that at least one of the two options supported by the two pieces of evidence must be true, while Dempster's combination rule disregards the two pieces of evidence. In our generalisation of Dempster's combination rule (Definition~\ref{def:combination:rule:BD}), we end up keeping track of the contradictions and where they come from. However, notice that if we work on an arbitrary bounded De Morgan algebra, we could still have cases where $x \wedge y = \bot$ even though neither $x=\bot$ nor $y=\bot$. This would happen if, even though the agent is tolerant to contradiction, they consider that it is impossible to get information that $x$ and $y$ are both true.  Therefore, in this situation, it could make sense to use the combination rule proposed by Dubois and Prade, but on a~De Morgan algebra rather than on a~powerset algebra:
\begin{align*}
\mass_{1 \oplus 2} (x)&=  \sum\limits_{x_1 \wedge x_2 = x}\!\!\!\!\!\mass_1(x_1)\cdot \mass_2(x_2)+\sum\limits_{\scriptsize{\begin{matrix}x_1 \wedge x_2 = \bot\\x_1 \vee x_2 = x\end{matrix}}}\!\!\!\!\!\mass_1(x_1) \cdot \mass_2(x_2).
\end{align*}

\subsection[Two-dimensional $\bel$ and $\pl$]{Two-dimensional reading of belief and plausibility}
\label{ssec:two:dimension:belief}
In this section, we introduce $\DS$ models on which we define belief and plausibility of formulas. Then we discuss different interpretations of the notions of belief and plausibility within the two-dimensional treatment of evidence of  $\BD$ logic.
\begin{definition}[$\DS$ models and their associated belief functions] 
\label{def:DS:models}
Let $\LTBD$ be the Lindenbaum algebra for Belnap--Dunn logic over the set of propositional letters $\Prop$.
A \emph{$\DS$ model} is a~tuple $\mathscr{M}=
(S,\P(S), \bel, v^+, v^-)$ such that 
$(S,v^+, v^-)$ is a~$\BD$ model and $\bel$ is a~belief function on $\P (S)$.
We denote $\bel_\mathscr{M}^+ : \LTBD \rightarrow [0,1]$ and $\bel_\mathscr{M}^- : \LTBD^{op} \rightarrow [0,1]$ the maps such that, for every $\varphi \in \LTBD$,
\begin{align}
 \bel_\mathscr{M}^+(\varphi) = \bel(|\varphi|^+) \qquad & \text{and } \qquad
 \bel_\mathscr{M}^-(\varphi) = \bel(|\varphi|^-) = \bel(|\neg\varphi|^+). 
\end{align}

We drop the subscript whenever there is no ambiguity on the model $\mathscr{M}$ we are considering.
\end{definition}
In the classical case (i.e., on a~Boolean algebra $\mathcal{B}$), one can define plausibility and belief function in terms of one another or via the mass function associated to the belief in several ways. 
This is why, there is no need to define \emph{both} plausibility and belief on \emph{classical} Dempster--Shafer structures. 
Indeed, the plausibility of $a\in\mathcal{B}$ can be construed as the lack of belief in its negation:
\begin{align}
\pl(a)=1-\bel({\sim}a)\label{eq:classicalplausibilityneg}
\end{align}
or, equivalently, as the sum of the mass of every statement compatible with $a$:
\begin{align}
\pl(a)=\sum\limits_{a\wedge b\neq\bot}\mass(b).
\label{eq:classicalplausibilitysum}
\end{align}
Equations \eqref{eq:classicalplausibilityneg} and \eqref{eq:classicalplausibilitysum} are, however, not equivalent on unbounded and free De Morgan algebras. Consider a~free De Morgan algebra $\mathcal{A}$ and a~(general) mass function $\mass$ on $\mathcal{A}$. If $\mathcal{A}$ is unbounded, 
\eqref{eq:classicalplausibilitysum} cannot be defined. If $\mathcal{A}$ is bounded, it gives $\pl(a)=1$ for every $a \in \mathcal{A}$. Indeed, if $\mathcal{A}$ is free, then  $a\wedge b=\bot$ iff $a=\bot$ or $b=\bot$. But $\mass(\bot)=0$, whence $\pl(a)=1$ for any $a\neq\bot$. On the other hand, \eqref{eq:classicalplausibilityneg} does not necessarily equal to $1$ on every $a\neq\bot$ should one substitute Boolean negation ${\sim}$ for the De Morgan $\neg$.
\begin{definition}[$\DS_\pl$ models and their associated plausibility functions] 
\label{def:DS:models:pl}
Let $\LTBD$ be the Lindenbaum algebra for Belnap--Dunn logic over the set of propositional letters $\Prop$.
A \emph{$\DS_\pl$ model} is a~tuple $\mathscr{M}=
(S,\P(S), \bel, \pl, v^+, v^-)$ such that 
$(S,\P(S), \bel, v^+, v^-)$ is a~$\DS$ model, $\pl$ is a~plausibility function on $\P(S)$.
We denote $ \pl_\mathscr{M}^+ : \LTBD \rightarrow [0,1]$ and $ \pl_\mathscr{M}^- : \LTBD^{op} \rightarrow [0,1]$ the maps such that, for every $\varphi \in \LTBD$,
\begin{align}
 \pl_\mathscr{M}^+ (\varphi) = \pl(|\varphi|^+) 
 \qquad & \text{and } \qquad \pl_\mathscr{M}^-(\varphi) = \pl(|\varphi|^-)=\pl(|\neg\varphi|^+). 
\end{align}
We drop the subscript whenever there is no ambiguity on the model $\mathscr{M}$ we are considering.
\end{definition}

Notice that like in the case of non-standard probabilities (see Lemma~\ref{lem:correspondece:lattice:language}), it is equivalent to define belief and plausibility on the Lindenbaum algebra or on the set of formulas. 

\begin{lemma}
\label{lem:bel+:pl+}
Let $\mathscr{M}=
(S,\P(S), \bel, \pl, v^+, v^-)$ be a~$\DS_\pl$ model. $\bel^+$ (resp., $\pl^+$) is a~general belief (resp., plausibility) function on the Lindenbaum algebra. $\bel^-$ (resp., $\pl^-$) is a~general belief (resp., plausibility) function on the dual of the Lindenbaum algebra $\LTBD^{op}$.
\end{lemma}

The previous lemma shows that each $\DS$ model generates a~function on the Lindenbaum, and by extension on the set of formulas of $\BD$ logic, that satisfies the axioms of (general) belief functions from Definition~\ref{def:generalbelieffunction}. The following theorem shows that the converse holds as well: for every (general) belief function on $\BD$ formulas, and by extension on the Lindenbaum algebra, we can define a~canonical $\DS$ model equipped with a~belief function such that both functions correspond.
 
\begin{theorem}[Completeness of belief axioms] \label{th:complBelAxioms}
Let $\bel$ be a~function on  $\BD$ formulas satisfying the  axioms of  (general) belief function (see Definition~\ref{def:generalbelieffunction}). Then there is a~canonical model $\mathscr{M}_c$ and a~ belief function $\bel'$ on the powerset of states of $\mathscr{M}_c$ such that $\bel(\f) = \bel' (|\f|^+)$.
\end{theorem}
\begin{proof}
The proof goes along the same lines as the one of Theorem~\ref{th:completeness_probablities}. 
We start with the canonical model  $\mathscr{M}_c = \langle \P(\LIT),  v^+_c, v^-_c\rangle$ from Definition~\ref{DEF:CanModel}. The general  belief function $\bel$ on $\BD$ formulas can be equivalently represented as a~general belief function $\bel$ on the Lindenbaum algebra $\LTBD$. Using this general belief function $\bel$ on the Lindenbaum algebra $\LTBD$, we define the belief function $\bel^*$ on the uppersets  of the poset $(\P(\LIT),\subseteq)$ (we assign $\bel^*(\varnothing)=0$ and $\bel^*(\P(\LIT))=1$). Then, we finish the proof by applying Lemma~\ref{lem:Zhou:3.7}. Notice that unlike in the case of probabilities this lemma does not guarantee the uniqueness of the extension of $\bel^*$. 
\end{proof}

The proof of the completeness of the plausibility axioms (Theorem \ref{th:complPlAxioms}) relies on the following remark that allows us to define a~De Morgan negation on the powerset of the domain of the canonical model that coincides with the $\BD$ negation on the extensions of formulas.
\begin{remark}[De Morgan negation on $\mathcal{P}(\mathcal{P}(\LIT))$]
\label{rk:demorgan:negation:PPLIT}
Let $w\subseteq\LIT$. We build the set $\LIT^\mathbf{4}(w)$ as follows
\begin{align}
\forall p\in\Prop:
\begin{cases}
\mathbf{T}(p)\in\LIT^\mathbf{4}(w)&\text{ iff }p\in w,\neg p\notin w\\
\mathbf{B}(p)\in\LIT^\mathbf{4}(w)&\text{ iff }p,\neg p\in w\\
\mathbf{N}(p)\in\LIT^\mathbf{4}(w)&\text{ iff }p,\neg p\notin w\\
\mathbf{F}(p)\in\LIT^\mathbf{4}(w)&\text{ iff }p\notin w,\neg p\in w
\end{cases}
\label{eq:LIT4w}
\end{align}
E.g., if $\Prop=\{p,q,r\}$ and $w=\{\neg p,q,\neg q\}$, then $\LIT^\mathbf{4}(w)=\{\mathbf{F}(p),\mathbf{B}(q),\mathbf{N}(r)\}$. We  call $\mathbf{X}p$'s ‘$\mathbf{4}$-literals’.

It is clear that, for any $A\in\mathcal{P}(\mathcal{P}(\LIT))$, there is a~unique (up to permutations) disjunctive normal form $\mathsf{Fm}(X)$ whose conjunctive clauses contain $\mathbf{4}$-literals and that every $\BD$ formula $\phi$ is represented by exactly one $A\in\mathcal{P}(\mathcal{P}(\LIT))$ which we denote $\mathsf{S}(\phi)$.

We now need to define a~proper De Morgan negation on $\langle\mathcal{P}(\mathcal{P}(\LIT)),\subseteq\rangle$ that extends the $\BD$ negation on $\LBD$ formulas. We take $A\subseteq\mathcal{P}(\LIT)$ and then $\mathsf{Fm}(A)$. Now we transform $\neg\mathsf{Fm}(A)$ into its disjunctive normal form using the following additional rules
\begin{align*}
\neg\mathbf{X}(p)&\rightsquigarrow[\neg_\mathbf{4}\mathbf{X}](p)&\mathbf{X}(p)\wedge\mathbf{Y}(p)&\rightsquigarrow[\mathbf{X}\wedge_\mathbf{4}\mathbf{Y}](p)&\mathbf{X}(p)\vee\mathbf{Y}(p)&\rightsquigarrow[\mathbf{X}\vee_\mathbf{4}\mathbf{Y}](p)
\end{align*}
with $\neg_\mathbf{4}$, $\vee_\mathbf{4}$, and $\wedge_\mathbf{4}$ following the truth-table definitions of negation, disjunction, and conjunction in $\BD$. Notice that if $A=\mathsf{S}(\phi)$ for some $\phi\in\LBD$, then $\neg A=\mathsf{S}(\neg\phi)$.
\end{remark}

\begin{theorem}[Completeness of plausibility axioms] \label{th:complPlAxioms}
Let $\pl$ be a~function on  $\BD$ formulas satisfying the  axioms of (general) plausibility function (see Definition~\ref{def:generalplausibilityfunction}). Then, there is a~canonical model $\mathscr{M}_c$ and a~ plausibility function $\pl'$ on the powerset of states of $\mathscr{M}_c$ such that $\pl(\f) = \pl' (|\f|^+)$.
\end{theorem}

\begin{proof}
We start with the canonical model  $\mathscr{M}_c = \langle \P(\LIT),  v^+_c, v^-_c\rangle$ from Definition~\ref{DEF:CanModel}. Let $\P^\uparrow(\LIT)$ denote the distributive lattice generated by the upsets of $(\P(\LIT),\subseteq)$.
The general  plausibility function $\pl$ on $\BD$ formulas can be equivalently represented as a~general plausibility function $\pl$ on the Lindenbaum algebra $\LTBD$.
Using this general plausibility function $\pl$ on the Lindenbaum algebra $\LTBD$, we define the plausibility function $\pl^*$ on $\P^\uparrow(\LIT)$ (we assign $\pl^*(\varnothing)=0$ and $\pl^*(\P(\LIT))=1$). 
Since the lattice $\P^\uparrow(\LIT)$ is isomorphic to the lattice reduct of $\LTBD^*$ (the Lindenbaum algebra for $\BD^*$), we can define a~De Morgan negation $\neg_{\P}$ on it that coincides with the negation of $\LTBD^*$. Therefore, $(\P^\uparrow(\LIT),\cup,\cap,\neg_\P,\varnothing,\P(\LIT))$ is a~finite bounded De Morgan algebra. 

Now, consider the function $\bel^*: (\P^\uparrow(\LIT),\subseteq) \rightarrow [0,1]$  such that
$\bel^*(S)=1-\pl^*(\neg_\P S)$ for every $S\in \P^\uparrow(\LIT)$. From Lemma \ref{lem:pl:associated:mass},
we know that $\bel^*$ is a~belief function on $(\P^\uparrow(\LIT),\subseteq)$.
Let $\mass^* : (\P^\uparrow(\LIT),\subseteq) \rightarrow [0,1]$ be the mass function of $\bel^*$. We can extend it to $\P(\P(\LIT))$ as follows:
\begin{align*}
    \mass': \P(\P(\LIT)) &\rightarrow [0,1]\\
    S & \mapsto \begin{cases}
\mass^*(S) & \mbox{if } S\in \P^\uparrow(\LIT),\\
0 & \mbox{otherwise.}
\end{cases}
\end{align*}
The function $\mass'$ is clearly a~mass function, therefore it defines a~belief function $\bel' : \P(\P(\LIT))\rightarrow [0,1]$ on the distributive lattice $(\P(\P(\LIT)),\cup,\cap,\varnothing, \P(\LIT))$. 
Using Remark \ref{rk:demorgan:negation:PPLIT}, we can extend the De Morgan negation $\neg_\P$ to the distributive lattice $(\P(\P(\LIT)),\cup,\cap,\varnothing, \P(\LIT))$. Therefore, $\bel'$ defines a~belief function on the finite De Morgan algebra $(\P(\P(\LIT)),\cup,\cap,\neg_\P,\varnothing, \P(\LIT))$. From Lemma
\ref{lem:bel:pl:1-bel}, we know that $\pl' : \P(\P(\LIT))\rightarrow [0,1]$ such that $\pl'(S)=1-\bel'(\neg_\P S)$ is a~plausibility function on $(\P(\P(\LIT)),\cup,\cap,\neg_\P,\varnothing, \P(\LIT))$. Therefore, $\pl'$ is also a~plausibility function on its lattice reduct $(\P(\P(\LIT)),\cup,\cap,\varnothing, \P(\LIT))$ and on the underlying Boolean algebra $(\P(\P(\LIT)),\cup,\cap,(\cdot)^c,\varnothing, \P(\LIT))$. Notice that, for every  $\phi\in\LBD$, we have $\pl'(|\phi|^+)=1-\bel'(\neg_\P |\phi|^+)=1-\bel^*(\neg_\P |\phi|^+)=\pl^*(|\phi|^+)=\pl(\phi)$.
\end{proof}

We have introduced $\BD$ models equipped with belief and plausibility functions. Here, we propose different ways to combine belief and plausibility with a~two-dimensional interpretation in order to introduce modalities in Section~\ref{sec:logics}.

\subsubsection{Belnapian belief}
\label{sssec:bel-pl-modalities:(bel+,bel-)}
We consider the following two-dimensional reading of belief. We look at belief as a~generalisation of non-standard probabilities, where the import-export axiom (see axiom (iii) of Definition~\ref{DEF:NSprob}) is weakened to the property of being weakly totally monotone (see Definition~\ref{def:generalbelieffunction}).
If we consider a~probabilistic $\BD$ model $\mathfrak{M}=\langle W,\mu, v^+,v^-\rangle$, then 
the  two-dimensional value of the probability of a~formula $\f$ is $(\mu(|\varphi|^+),\mu(|\varphi|^-))$  and it is interpreted as follows. Positive probability $\mu(|\varphi|^+)$ is the degree to which evidence supports truth of $\f$, while negative probability $\mu(|\varphi|^-)$ is the degree to which evidence supports its falsity (which is the same as positive probability of $\neg\f$).

Following these lines, we define a~Belnapian belief $(\bel(|\f|^+), \bel(|\f|^-))$ based on a~$\DS$ model $\mathscr{M}$. Analogously to the case of non-standard probabilities the value $\bel(|\f|^+)$ represents the degree to which the evidence supports $\f$ and $\bel(|\f|^-)$ represents the degree to which the evidence supports its negation. A~natural way to introduce plausibility of a~formula $\f$ is to use the classical definition via the belief of the negation of $\f$. This definition is correct, because we know from Lemma~\ref{lem:bel:pl:1-bel} that  since $\bel^+$ is a~general belief function on $\LTBD$, then the map $\pl^+$ defined as  $\pl^+(\f)=1-\bel^+(\neg\f)$ is a~general plausibility function on $\LTBD$. Similarly, $\pl^-(\f)=1-\bel^-(\neg\f)$ is a~general plausibility function on $\LTBD^{op}$. Observe that in the case of strong contradictory belief in some proposition, it can happen that plausibility is strictly smaller than belief, contrary to the intuition understanding them as an upper and lower bound.

Formally, we can work with belnapian plausibility and take the pair $(\pl^+(\varphi),\pl^-(\varphi))$ as the primary notion while  belief would be a~derived one, but this choice is less appealing from the point of view of an interpretation. 
In Example~\ref{ex:expressivity_luksquare}, we will show that within this framework we can introduce a~logic  that allows us to characterise evidence in terms of how classical, incomplete or contradictory it is regarding a~specific topic $\varphi$.
Indeed, in our framework, one can separate pieces of evidence into three categories:
classical (where the evidence for and evidence against add up exactly to~$1$),
 incomplete (the evidence for and evidence against add up to a~number smaller than $1$), or contradictory (the evidence for and evidence against add up to a~number greater than $1$).
Each of these gives us different signals.
In particular, classical information might be intuitively interpreted as indicative of us being on the right track in the investigation. I.e., if the information on $p$ is classical, then the investigation into $p$ can be deemed \emph{satisfactory}. In this vein, incomplete information can be interpreted as us needing to \emph{investigate $p$ further}, while contradictory information on $p$ shows us that our previous investigation was faulty, whence we need to \emph{re-investigate once again}.

\subsubsection{Combining belief and plausibility}
\label{sssec:bel-pl-modalities:(bel+,pl-)}
The independence of positive and negative support which is the key idea of the Belnap-Dunn approach gives us more freedom in combining belief and plausibility than in the classical approach. While in the previous case, both positive and negative supports were represented by the same uncertainty measure, now we discuss the possibility of combining them. The idea behind the classical belief--plausibility relation is that we can see the plausibility of a~proposition as a~lack of support for its negation. This motivates our second choice for the representation of negative support of a~proposition: it is not a~(straightforward) support of its negation as in the previous case, but rather a~lack of support of the proposition itself.

Formally, we consider a~$\DS_\pl$ model containing both belief and plausibility (in general they might be computed from different mass functions, so they are not mutually definable) and define the positive and negative support pair as $(\bel^+(\varphi), \pl^-(\varphi)) = (\bel^+(\varphi), \pl^+(\neg \varphi))$.
\subsubsection{Belief and plausibility as lower and upper bounds}
\label{sssec:bel-pl-modalities:(bel+,bel-):bel<pl}
In this section, we focus on the interpretation of  belief and plausibility as a~lower and an upper approximation of the probability of a~formula. To do so, we have to consider $\DS_\pl$ models $\mathscr{M}=(S,\P(S), \bel, \pl, v^+, v^-)$ with both belief and plausibility introduced independently, and study the implications of the following property: 
\begin{align}
\label{eq:prop:bel<pl}
 \bel(Y)\leq \pl(Y)\text{ for every }Y\subseteq\P(S)
\end{align}
Within the framework of $\BD$ logic, this immediately implies that one cannot define $\pl$ via $\bel$ without imposing strong constraints on the valuation of the $\BD$ model. Therefore, here we are studying the meaning of having a~general belief function and a~general plausibility function generated by two different mass functions.
First, notice that for every $v^+$, \eqref{eq:prop:bel<pl} is equivalent to
\begin{equation}
 \sum_{X \subseteq |p|^+} \mass_\bel(X) \leq 1 - \sum_{X \subseteq |\neg p|^+}\mass_\pl(X) .
 \label{eq:bel<pl:general:belief:functions}
\end{equation}
Indeed, from Theorem~\ref{theo:totallymonotone:charactrisation}  and Lemma~\ref{lem:pl:associated:mass}, we have
\begin{align*}
 \bel(|p|^+) = \sum_{X \subseteq |p|^+} \mass_\bel(X) & \leq \pl(|p|^+) = 1 - \bel_\pl(|\neg p|^+) 
 = 1 - \sum_{X \subseteq |\neg p|^+}\mass_\pl(X) .
\end{align*}
In addition, since $\bel$ and $\pl$ are respectively belief and plausibility functions, then we get the constraint
\begin{align}
 \bel(|p|^+) =\sum_{X \subseteq |p|^+} \mass_\bel(X) & \leq \sum_{X \not \subseteq |p|^-} \mass_\pl(X)=\pl(|p|^+),
 \label{eq:bel<pl:belief:functions}
\end{align}
because
\begin{align*}
 \bel(|p|^+) = \sum_{X \subseteq |p|^+} \mass_\bel(X) & \leq \pl(|p|^+) = 1 - \sum_{X \subseteq |\neg p|^+}\mass_\pl(X) \\
 & = \sum_{X \in \P(S)} \mass_\pl(X) - \sum_{X \subseteq |\neg p|^+}\mass_\pl(X)
 \\
 & = \sum_{X \not \subseteq |\neg p|^+} \mass_\pl(X) = \sum_{X \not \subseteq |p|^-} \mass_\pl(X).
\end{align*}
We consider the following two-dimensional interpretation for belief and plausibility of $\phi$ respectively: $\belmod\phi=(\bel^+(\phi),\bel^-(\phi))$ and 
$\Pl\phi=(\pl^+(\phi),\pl^-(\phi))$.

\paragraph{Interpretation of the mass functions $\mass_\bel$ and $\mass_\pl$}
Here, $\belmod(\phi)=(x,y)$ is interpreted as follows: $x=\bel^+(\phi)$ is how much the agent is persuaded that $\phi$ is true based on the evidence, and $y=\bel^-(\phi)$ is how much the agent is persuaded that $\phi$ is false based on the evidence. Persuasion relies on a~wide range of evidence: reliable scientific evidence, but also an emotional reaction to an argument. Therefore, one could ignore the contradictoriness of some evidence to strengthen one's opinion and one can ignore scientific evidence due to some bias. 
Hence, when computing $\mass_\bel$, the agent uses a~weak standard for saying that a~piece of evidence supports some statement and is influenced by its own bias: highly contradictory evidence can be considered to support a~statement and reliable evidence can be ignored.

However, the agent uses a~different standard of evidence to decide that a~statement is not plausible. We have $\pl(|p|^+)=\sum_{X \not \subseteq |\neg p|^+} \mass_\pl(X)$, therefore, $p$ is considered plausible if there is very little strong evidence supporting $\neg p$. We interpret this as follows: the agent considers $p$ plausible if they are not convinced that $\neg p$ is the case. Conviction is built on ‘reliable’ evidence. The meaning of ‘reliable’ will depend on the context: in a~court, it could be the kind of evidence accepted by the court, in science it could be detailed proofs that have been reviewed by experts\ldots One can for instance believe that a~mathematical statement is false because of some personal intuition based on experience, even if there is a~verified proof of the statement which makes it not very plausible (but not impossible) that it is false. Indeed, with time, some proofs are found to be false.

We propose the following interpretation of the mass functions.
$\mass_\bel$ is computed by asking the question ‘does the evidence persuade the agent?’, while $\mass_\pl$ is computed by asking the question ‘is the evidence considered convincing by some given authority?’ In order to get a~better intuition on what is going on, let us have a~look at some examples.
\begin{figure}[h!]
 \centering
 \xymatrix{
&& s_0 \ : 
& s_1 \ : \ p 
& s_2 \ : \ \neg p
& s_3 \ : \ p, \neg p
}
\caption{Canonical model over $\Prop=\{p\}$.}
\label{fig:ex:4:point:model}
\end{figure}
\begin{example}[Strong belief in $p \wedge \neg p$] Consider the set of variables $\Prop=\{p\}$. The canonical model is in Figure~\ref{fig:ex:4:point:model}.
Assume that $\bel(| p\wedge\neg p| ^+)=1$. Therefore, 
$$\sum_{X \subseteq |p\wedge\neg p|^+}\mass_\bel(X)=\sum_{X \subseteq \{s_3\}} \mass_\bel(X)=\mass_\bel(\varnothing)+\mass_\bel(\{s_3\}) = 1.$$ 
Since $\bel$ is a~belief function on $\P(S)$, we have  $\mass_\bel(\varnothing)= 0$, therefore $\mass_\bel(\{s_3\}) = 1$. This means that based on the available evidence, the agent is persuaded that $p\wedge \neg p$ is the case. Notice that 
\begin{align*}
 \pl(|p\wedge \neg p|^+) 
 & = \sum_{X \not \subseteq |p \wedge \neg p|^-} \mass_\pl(X)
 = \sum_{X \not \subseteq |\neg (p \wedge \neg p)|^+} \mass_\pl(X)\\
 & = \sum_{X \not \subseteq |\neg p \vee p|^+} \mass_\pl(X)
 = \sum_{X \not \subseteq |\neg p|^+ \cup |p|^+} \mass_\pl(X)\\
 & = \sum_{X \not \subseteq |p|^- \cup |p|^+} \mass_\pl(X)
\end{align*}
Therefore, the condition that $\bel \leq \pl$ implies that
\begin{align}
 1 = \sum_{X \not \subseteq |p|^- \cup |p|^+} \mass_\pl(X) = \sum_{X \not \subseteq \{s_1,s_2,s_3\}} \mass_\pl(X) = \mass_\pl(S).
\end{align}
Therefore, evidence that is strongly persuasive considering $p \wedge \neg p$ is inconclusive regarding the plausibility of either $p$ or $\neg p$.
\end{example}

\begin{example}[Weak belief in either $p$ or $\neg p$]
We still consider the previous model. 
Assume now that $\bel(|p \vee \neg p|^+)=0$. Therefore $\sum_{X \subseteq \{s_1,s_2,s_3\}}\mass_\bel(X)=0$ and $\mass_\bel(\{s_0,s_1,s_2,s_3\})=1$.
The condition that $\bel \leq \pl$ does not constrain the value of $\mass_\pl$ in any way. Therefore, the available evidence might be convincing from the point of view of a~given authority, however, it did not persuade the agent.
\end{example}

\section[Two-layered logics]{Two-layered logics for uncertainty measures}\label{sec:logics}
When reasoning about the uncertainty measures defined on sets of events, we can formalise the measure itself as a modal operator $\mathsf{M}$. The behaviour of $\mathsf{M}$ can be defined in two ways. The first option is to treat it as a usual modality as done in, e.g.~\cite{DautovicDoderOgnjanovic2021}; the second one is to use a ‘two-layered’ approach and apply $\mathsf{M}$ \emph{only to the formulas describing events} and then reason with these formulas in another, so-called, \emph{outer-layer}, logic (cf., e.g.,~\cite{FaginHalpernMegiddo1990,GodoHajekEsteva1995,CintulaNoguera2014,BaldiCintulaNoguera2020}). The main \emph{syntactic} difference between these two approaches is that in the two-layered treatment of measure modalities, \emph{they do not nest}.

At first glance, this seems as too much of a restriction. It is, however, justified. First, the decidability of two-layered logics is often easy to establish. Furthermore, the decision procedures allow for the extraction of the complexity evaluations of the logics and are, in addition, quite straightforward (cf., e.g.,~\cite{HajekTulipani2001} for the two-layered systems based on fuzzy logics and~\cite{FaginHalpernMegiddo1990} for those dealing with linear inequalities) as one can utilise those for the outer-layer logics. On the other hand, the straightforward decision procedures for the modal logics that \emph{do allow} for the nesting of $\mathsf{M}$ seem more difficult to obtain. E.g.,~\cite{DautovicDoderOgnjanovic2021} shows that the logic $\mathrm{CKL}$ is decidable using filtration, however, this result is not easy to apply if one wants to establish the complexity of $\mathrm{CKL}$-satisfiability or validity.

Second, nested modalities are difficult to interpret in the natural language. While $\mathsf{M}p$ can be understood as ‘$p$ is probable’, ‘the agent believes that $p$ is the case’, etc.\ depending on $\mathsf{M}$, and its value can be straightforwardly derived from the measure of the subset of the sample space where $p$ is true, the interpretation of formulas such as $\mathsf{M}(p\wedge\mathsf{M}q)$ is considerably less intuitive. Thus, in this paper, we are formalising paraconsistent reasoning about uncertainty using two-layered logics.

The two-layered formalisms are also divided into two groups depending on how they formalise statements such as ‘probability of $\phi$ is twice as high as the probability of $\chi$’. There are two options. The first one is less formal and more intuitive. We allow numerical reasoning with inequalities directly on the outer layer. This is the way it was originally done in~\cite{FaginHalpernMegiddo1990}. Another approach, more ‘puristic’ from a~logical point of view, would be to use an outer-layer logic that can express addition and subtraction, e.g., \L{}ukasiewicz logic. This approach originated in~\cite{GodoHajekEsteva1995} and has been investigated further in an abstract algebraic manner, e.g., in~\cite{CintulaNoguera2014,BaldiCintulaNoguera2020}. Usually, the classical propositional logic is employed to describe the events, with an exception of a~$\BD$-based logic for belief functions considered in~\cite{Zhou2013}. We will present the two-layered logics that are based on both of the approaches mentioned above.

Before we proceed further, let us give a short note on the interpretation of atomic modal formulas. Throughout the section, we will use the following atomic modal formulas with $\phi\in\LBD$: $\mathtt{w}^\pm(\phi)$, $\mathtt{b}^\pm(\phi)$, $\Prob\phi$, $\belmod\phi$, and $\Pl\phi$. Their values range over $[0,1]$ and are interpreted as the degrees of an agent's certainty in $\phi$ given by the corresponding measure. E.g., $\mathtt{w}^-(\phi)\geqslant\frac{2}{3}$ means that the probability measure of $|\phi|^-$ is at least $\frac{2}{3}$ (cf.~Definition~\ref{def:FDEweightformulas} for further details); $v^\mathscr{M}_1(\Pl\phi)=\frac{1}{4}$ stands for ‘the plausibility of $|\phi|^+$ is equal to $\frac{1}{4}$’ (cf.~Definition~\ref{def:BelNLuk}), etc.
\subsection{Two-layered logics with inequalities}\label{ssec:2layerineq}
In this section, we present two-layered logics axiomatising reasoning with non-standard belief and plausibility functions that follow the approach of~\cite{FaginHalpernMegiddo1990}. Namely, we employ $\LBD$-formulas on the inner layer and linear inequalities on the outer one.
\subsubsection{Logic for non-standard probabilities}
\begin{definition}[BD weight formulas and their semantics]\label{def:FDEweightformulas} a~\emph{primitive weight formula} (PWF) is an expression of the form \[\sum\limits_{i=1}^{n}a_i\cdot\mathtt{w}^\pm(\phi_i)\geqslant c\]
with $a_i,c\in\mathbb{Z}$, and $\phi_i\in\LBD$. The left hand side of a~PWF is called a~\emph{weight term}.
A \emph{weight formula} is a~Boolean combination of primitive weight formulas. 

Let $\mathscr{M}=\langle W,v^+,v^-,\mu\rangle$ be a~probabilistic $\BD$ model, $\mass$, the mass function associated to $\mu$, and $\alpha$, a~weight formula. The satisfaction relation $\mathscr{M}\models\alpha$ is defined as follows\footnote{Observe that weights correspond to sums of masses of states where formulas are verified or falsified, i.e., to their \emph{probabilities}.}.
\begin{align*}
\mathscr{M}\models\sum\limits_{i=1}^{n}a_i\cdot\mathtt{w}^\pm(\phi)\geqslant c&\text{ iff }\sum\limits_{i=1}^{n}a_i\cdot\sum\limits_{w\vDash^\pm\phi}\mathtt{m}(w)\geqslant c\\
\mathscr{M}\models{\sim}\alpha&\text{ iff }\mathscr{M}\not\models\alpha\\
\mathscr{M}\models\alpha\wedge\alpha'&\text{ iff }\mathscr{M}\models\alpha\text{ and }\mathscr{M}\models\alpha'
\end{align*}
Other connectives: $\vee$, $\supset$, etc.\ can be defined in a~usual manner.

A weight formula is called \emph{valid} iff it is satisfied on every model. We say that a~set of weight formulas $\Xi$ \emph{entails} $\alpha$ ($\Xi\vDash_{\mathtt{m}\mathsf{BD}}\alpha$) iff there is no model $\mathscr{M}$ s.t.\ $\mathscr{M}\models\Xi$ but $\mathscr{M}\not\models\alpha$.
\end{definition}

Below, we give a~calculus that proves weight formulas.
\begin{definition}[Axioms for $\BD$ weight formulas --- calculus $\mathtt{w}\BD$]\label{def:FDEweightcalculus}
\begin{enumerate}
\item[]
\item[$\mathsf{CPL}$] $\alpha$ s.t.\ $\mathsf{CPL}\vdash\alpha$
\item[MP] $\dfrac{\alpha\quad\alpha\rightarrow\beta}{\beta}$
\item[$\mathsf{ineq}$] All instances of valid formulas about linear inequalities.
\item[$W1$] $\mathtt{w}^\pm(\phi)\geqslant0$, $\mathtt{w}^\pm(\phi)\leqslant1$
\item[$W2$] $\mathtt{w}^\mp(\phi)=\mathtt{w}^\pm(\neg\phi)$
\item[$W3$] $\mathtt{w}^+\left(\bigvee\limits_{1\leq i\leq k}\phi_i\right)=\sum\limits_{\scriptsize{\begin{matrix}J\subseteq\{1,\ldots,k\}\\J\neq\varnothing\end{matrix}}}(-1)^{|J|+1}\mathtt{w}^+\left(\bigwedge\limits_{j\in J}\phi_j\right)$
\item[$W4$] $\mathtt{w}^+(\phi)\leqslant\mathtt{w}^+(\phi')$ and $\mathtt{w}^-(\phi)\geqslant\mathtt{w}^-(\phi')$ for all $\phi,\phi'\in\LBD$ s.t.\ $\phi\vdash_\mathsf{BD}\phi'$.
\end{enumerate}
\end{definition}
\begin{theorem}[Completeness]\label{theorem:wFDEcompleteness}
$\mathtt{w}\BD\vdash\alpha$ iff $\alpha$ is valid.
\end{theorem}
\begin{proof}
The soundness part can be obtained by verifying that all axioms hold on all probabilistic BD models. Indeed, we reason classically about the mass, and $W1$--$W4$ are just the axioms for non-standard probabilities.

For the completeness part, we reason by contraposition. Assume that $\mathtt{w}\BD\nvdash{\sim}\alpha$. We show that $\alpha$ is satisfiable. Since we reason with the weight formulas using $\mathsf{CPL}$, we can w.l.o.g.\ assume that $\alpha$ is in $\DNF$. Moreover, a~$\DNF$ is satisfiable iff at least one of its clauses is. Thus, we can consider only weight formulas of the form $\alpha=\bigwedge\limits_{i=1}^{n}\pi_i$ with $\pi_i$'s being either PWFs or their negations. Notice however that a~negation of a~PWF is provably equivalent via $\mathsf{ineq}$ to a~PWF
\begin{align}
{\sim}\left(\sum\limits_{i=1}^{n}a_i\cdot\mathtt{w}^+(\phi_i)\geqslant c\right)\dashv\vdash_{\mathtt{w}\mathsf{BD}}\sum\limits_{i=1}^{n}-a_i\cdot\mathtt{w}^+(\phi_i)>-c\label{eq:-toneg}
\end{align}

Thus, we assume w.l.o.g.\ that each $\pi_i$ is a~PWF. Furthermore, using $W2$ we can get rid of $\mathtt{w}^-$ and work with formulas containing $\mathtt{w}^+$ only.

Now, using $W3$, $W4$ and the fact that each $\phi\in\LBD$ has a~unique representation as an $\fDNF$ over a~given set of literals (cf.~Definition~\ref{def:XfDNF:BD}) up to permutation of clauses and variables we conduct the following provably equivalent transformations.

First, take one $\pi_k$ ($1\leqslant k\leqslant n$)
\begin{align}
\pi_k=a_1\cdot\mathtt{w}^+(\phi_1)+\ldots+a_s\cdot\mathtt{w}^+(\phi_s)\geqslant c_k\label{eq:pi_k}
\end{align}
and put $\mathsf{fDNF}_{\LIT(\alpha)}(\phi_j)$'s instead of $\phi_j$'s. We denote the transformed formula with $\pi^\alpha_k$:
\begin{align}
\pi^\alpha_k=a_1\cdot\mathtt{w}^+(\fDNF_{\LIT(\alpha)}(\phi_1))+\ldots+a_s\cdot\mathtt{w}^+(\mathsf{fDNF}_{\LIT(\alpha)}(\phi_s))\geqslant c_k\label{eq:pi_k^alpha}
\end{align}
Now, we use $W3$ and $\mathsf{ineq}$ to remove terms of the form $\mathtt{w}^+(\psi\vee\psi')$. Namely, let
\[\mathsf{fDNF}_{\LIT(\alpha)}(\phi_j)=\bigvee\limits_{t=1}^{u}\mathsf{cl}_t\]
with $\mathsf{cl}_t$ being a~conjunction of literals. Then, by $W3$, $\mathsf{ineq}$, and propositional reasoning, we obtain
\begin{align}
a_j\cdot\mathtt{w}^+(\mathsf{fDNF}_{\LIT(\alpha)}(\phi_j))=\sum\limits_{\scriptsize{\begin{matrix}J\subseteq\{1,\ldots,u\}\\J\neq\varnothing\end{matrix}}}(-1)^{|J|+1}\cdot a_j\cdot\mathtt{w}^+\left(\bigwedge\limits_{t\in J}\mathsf{cl}_t\right)\label{eq:exclinclfDNF}
\end{align}
Using $W4$, we remove repeating literals from $a_j\cdot\mathtt{w}^+\left(\bigwedge\limits_{t\in J}\mathsf{cl}_t\right)$'s. After that, we also add the following formulas as conjuncts (here, $\bigwedge\limits_{t\in J}\mathsf{cl}'_t$'s are $\bigwedge\limits_{t\in J}\mathsf{cl}_t$'s without repeating literals).
\begin{align}
\left[\sum\limits_{\scriptsize{\begin{matrix}J\subseteq\{1,\ldots,u\}\\J\neq\varnothing\end{matrix}}}(-1)^{|J|+1}\cdot\mathtt{w}^+\left(\bigwedge\limits_{t\in J}\mathsf{cl}'_t\right)\leqslant1\right]\wedge\left[\sum\limits_{\scriptsize{\begin{matrix}J\subseteq\{1,\ldots,u\}\\J\neq\varnothing\end{matrix}}}(-1)^{|J|+1}\cdot\mathtt{w}^+\left(\bigwedge\limits_{t\in J}\mathsf{cl}'_t\right)\geqslant0\right]\label{eq:01constraintexclincl}
\end{align}
Observe that \eqref{eq:01constraintexclincl} is provable from~\eqref{eq:exclinclfDNF} by $\mathsf{ineq}$. We denote the resulting formula with $(\pi^\alpha_k)^\wedge$.

Then, we set
\[\alpha^\wedge=\bigwedge\limits^{n}_{k=1}(\pi^\alpha_k)^\wedge\]
Finally, for each $\bigwedge\limits_{t\in J}\mathsf{cl}'_t$ and $\bigwedge\limits_{t'\in J'}\mathsf{cl}'_{t'}$ present in $\alpha^\wedge$ s.t.\ $\bigwedge\limits_{t\in J}\mathsf{cl}'_t\vdash_{\mathsf{FDE}}\bigwedge\limits_{t'\in J'}\mathsf{cl}'_{t'}$, we add
\begin{align}
\left[\mathtt{w}^+\left(\bigwedge\limits_{t\in J}\mathsf{cl}'_t\right)\leqslant\mathtt{w}^+\left(\bigwedge\limits_{t'\in J'}\mathsf{cl}'_{t'}\right)\right]\wedge\left[\mathtt{w}^-\left(\bigwedge\limits_{t\in J}\mathsf{cl}'_t\right)\geqslant\mathtt{w}^-\left(\bigwedge\limits_{t'\in J'}\mathsf{cl}'_{t'}\right)\right]\label{eq:ordercontstraintexclincl}
\end{align}
as new conjuncts (again, observe that they are axioms). We call the resulting formula $(\alpha^\wedge)^+$. It is clear that $\mathtt{w}\mathsf{FDE}\vdash\alpha\leftrightarrow(\alpha^\wedge)^+$ since all our transformations preserved provable equivalence.

Now, observe that $(\alpha^\wedge)^+$ is satisfiable iff the system of inequalities obtained from its conjuncts by replacing $w^\pm(\psi)$'s with $x^\pm_\psi$'s has a~solution (but if it does not, then ${\sim}(\alpha^\wedge)^+\in\mathsf{ineq}$, whence ${\sim}\alpha\in\mathsf{ineq}$, contrary to our initial assumption). Indeed, $\bigwedge\limits_{t\in J}\mathsf{cl}'_t$'s correspond to states in a~model, and the constraints on their masses are given in~\eqref{eq:01constraintexclincl} and~\eqref{eq:ordercontstraintexclincl}.

Now, we construct a~model from the solution to the system of inequalities setting $W=\mathcal{P}(\LIT(\alpha))$. For this, we need to translate weights into the mass assignments. First, we order $\bigwedge\limits_{t\in J}\mathsf{cl}'_t$'s w.r.t.\ $\vdash_{\mathsf{FDE}}$. Note that by the construction described in~\eqref{eq:pi_k^alpha} and~\eqref{eq:ordercontstraintexclincl}, there is the least element w.r.t.\ this order, namely, the conjunction of all literals generated by the $\mathsf{Var}(\alpha)$ (in fact, every $\mathsf{fDNF}$ contains such clause). Observe, further, that due to~\eqref{eq:ordercontstraintexclincl} it has the smallest positive weight.

Following the proof of~\cite[Theorem 4]{KleinMajerRad2021}, we define masses of states inductively, starting from $\LIT(\alpha)$. We set $\mathtt{m}(\LIT(\alpha))=\mathtt{w}^+(\LIT(\alpha))$. As for the other states obtained from $\bigwedge\limits_{t\in J}\mathsf{cl}'_t$'s, we define their masses as follows. 
\begin{align}
\mathtt{m}\left(\bigwedge\limits_{t\in J}\mathsf{cl}'_t
\right)=\mathtt{w}^+\left(\bigwedge\limits_{t\in J}\mathsf{cl}'_t
\right)-\sum\limits_{\bigwedge\limits_{t'\in J'}\mathsf{cl}'_{t'}\not\dashv~\vdash_{\mathsf{FDE}}\bigwedge\limits_{t\in J}\mathsf{cl}'_t}\mathtt{m}\left(\bigwedge\limits_{t'\in J'}\mathsf{cl}'_{t'}\right)\label{eq:massesfromweights}
\end{align}
If the sum of masses so acquired is still smaller than $1$, the remainder goes to the state that corresponds to the empty set of literals. All other states get mass zero. The measure can be then reconstructed from $\mass$.
\end{proof}
\subsubsection{Logic for general belief functions}
\begin{definition}[BD belief formulas and their semantics]\label{def:FDEbeliefformulas}
A \emph{primitive belief formula} (PBF) is an expression of the following form
\[\sum\limits_{i=1}^{n}a_i\cdot\mathtt{b}^\pm(\phi_i)\geqslant c\]
with $a_i,c\in\mathbb{Z}$, and $\phi_i\in\LBD$. The left hand side of a~PBF is called a~\emph{belief term}. A~\emph{belief formula} is a~Boolean combination of PBFs.

Let now $\mathscr{M}=\langle\mathcal{A},\bel_1,\bel_2\rangle$ be a~De Morgan algebra with a~pair belief functions such that $\bel_i(x)=\bel_j(\neg x)$ ($i,j\in\{1,2\}$, $i\neq j$) for all $x\in\mathcal{A}$ on it, and $\alpha$ be a~belief formula. We define the satisfaction relation $\mathscr{M}\models\alpha$ as follows.
\begin{align*}
\mathscr{M}\models\sum\limits_{i=1}^{n}a_i\cdot\mathtt{b}^+(\phi_i)\geqslant c&\text{ iff }\sum\limits_{i=1}^{n}a_i\cdot\bel_1(v(\phi_i))\geqslant c\\
\mathscr{M}\models\sum\limits_{i=1}^{n}a_i\cdot\mathtt{b}^-(\phi_i)\geqslant c&\text{ iff }\sum\limits_{i=1}^{n}a_i\cdot\bel_2(v(\phi_i))\geqslant c\\
\mathscr{M}\models{\sim}\alpha&\text{ iff }\mathscr{M}\not\models\alpha\\
\mathscr{M}\models\alpha\wedge\alpha'&\text{ iff }\mathscr{M}\models\alpha\text{ and }\mathscr{M}\models\alpha'
\end{align*}
$\alpha$ is called valid iff it is satisfied on every $\mathscr{M}$.
\end{definition}
\begin{definition}[Axioms for $\BD$ belief formulas --- calculus $\mathtt{b}\BD$]\label{def:FDEbeliefcalculus}
\begin{enumerate}
\item[]
\item[$\mathsf{CPL}$] $\alpha$ s.t.\ $\mathsf{CPL}\vdash\alpha$
\item[MP] $\dfrac{\alpha\quad\alpha\rightarrow\beta}{\beta}$
\item[$\mathsf{ineq}$] All instances of valid formulas about linear inequalities.
\item[$B1$] $\mathtt{b}^\pm(\phi)\geqslant0$, $\mathtt{b}^\pm(\phi)\leqslant1$
\item[$B2$] $\mathtt{b}^\mp(\phi)=\mathtt{b}^\pm(\neg\phi)$
\item[$B3$] $\mathtt{b}^+\left(\bigvee\limits_{1\leq i\leq k}\phi_i\right)\geqslant\sum\limits_{\scriptsize{\begin{matrix}J\subseteq\{1,\ldots,k\}\\J\neq\varnothing\end{matrix}}}(-1)^{|J|+1}\mathtt{b}^+\left(\bigwedge\limits_{j\in J}\phi_j\right)$
\item[$B4$] $\mathtt{b}^+(\phi)\leqslant\mathtt{b}^+(\phi')$ and $\mathtt{b}^-(\phi)\geqslant\mathtt{b}^-(\phi')$ for all $\phi,\phi'\in\LBD$ s.t.\ $\phi\vdash_\mathsf{BD}\phi'$.
\end{enumerate}
\end{definition}
\begin{definition}\label{def:canonicalPBF}
Let $\pi=\left[\sum\limits_{i=1}^{n}a_i\cdot\mathtt{b}^\pm(\phi_i)\geqslant c\right]$ be a~PBF and let $X\supseteq\LIT(\pi)$. We call $\pi$ \emph{$X$-canonical} iff $\phi_i$ is not equivalent to $\phi_j$ for any $i,j\in\{1,\ldots,n\}$ s.t.\ $i\neq j$, all occurring $\phi_i$'s are in fDNF w.r.t.\ $X$, and all occurrences of $\mathtt{b}^\pm$ are positive.
\end{definition}
\begin{proposition}\label{prop:bFDEcanonisation}
For any PBF $\pi$ and for any $X\supseteq\LIT(\pi)$ there exists a~canonical PBF $\pi^\mathsf{c}$ s.t.\ $\mathtt{b}\BD\vdash\pi\leftrightarrow\pi^\mathsf{c}$.
\end{proposition}
\begin{proof}
Observe that we can use $\mathsf{ineq}$, $B2$, and $B4$ to make a~PBF canonical.
\end{proof}
\begin{theorem}\label{theorem:bFDEcompleteness}
$\mathtt{b}\BD\vdash\alpha$ iff $\alpha$ is valid.
\end{theorem}
\begin{proof}
To verify soundness, we just check that all axioms are valid. For the completeness part, we reason by contraposition. Assume that $\mathtt{b}\BD\nvdash{\sim}\alpha$. We show that $\alpha$ is satisfiable.

We assume w.l.o.g.\ that $\alpha$ is in DNF. Whence, we can reduce $\alpha$ to a~conjunction of PBFs (recall that because of~\eqref{eq:-toneg} a~negation of PBF is a~PBF as well). Furthermore, using proposition~\ref{prop:bFDEcanonisation}, we can further assume that all PBFs are $\LIT(\alpha)$-canonical. Thus we set $\alpha\coloneqq\bigwedge\limits_{i=1}^{n}\pi^\mathsf{c}_i$ with
\begin{align}
\pi^\mathsf{c}_i&\coloneqq\sum\limits_{j=1}^{m}a_j\cdot\mathtt{b}^\pm(\mathsf{fDNF}_{\LIT(\alpha)}(\phi_j))\geqslant c_i\label{eq:picanonical}
\end{align}
We add the following conjuncts to $\alpha$.
\begin{enumerate}
\item\label{item:inclusion+} $\mathtt{b}^+(\mathsf{fDNF}_{\LIT(\alpha)}(\phi))\!\geqslant\!\sum\limits_{\scriptsize{\begin{matrix}J\subseteq\{1,\ldots,r\}\\J\neq\varnothing\end{matrix}}}(-1)^{|J|+1}\mathtt{b}^+\left(\bigwedge\limits_{k\in K}\mathsf{cl}_k\right)$ for every $\phi$ over $\mathsf{Var}(\alpha)$ up to~$\dashv\vdash_\BD$\footnote{Recall that $\BD$ is tabular (locally finite), whence there are only finitely many pairwise non-equi-provable formulas over a~given finite set of variables.}.
\item\label{item:order} $\mathtt{b}^+(\phi)\leqslant\mathtt{b}^+(\phi')$ for each $\phi$ and $\phi'$ added to $\alpha$ on step~\ref{item:inclusion+} s.t.\ $\phi\vdash_\BD\phi'$.
\item\label{item:0-1} $0\leqslant\mathtt{b}^+(\phi)\leqslant1$ for each $\phi$ added to $\alpha$ on the previous steps.
\end{enumerate}
Observe that these conjuncts are axioms. Hence, the new formula (we dub it with $\alpha^+$) is provably equivalent to $\alpha$. Furthermore, it is clear that $\alpha^+$ is just a~system of linear equations which is satisfiable (otherwise, ${\sim}\alpha^+$ is an instance of $\mathsf{ineq}$).

It remains to construct a~model. For this, we take the Lindenbaum algebra generated by $\mathsf{Var}(\alpha)$, and define $\mathtt{b}^+([\phi])$ using a~solution for $\alpha^+$. Note that because of~\ref{item:order}, $\mathtt{b}^+$ is defined for every element. Then define $\mathtt{b}^-([\phi])$ as $\mathtt{b}^+([\neg\phi])$. Clearly, $\mathtt{b}^+$ and $\mathtt{b}^-$ are belief functions on the Lindenbaum algebra because items~\ref{item:inclusion+},~\ref{item:order}, and~\ref{item:0-1} are exactly the belief functions axioms.
\end{proof}

\subsection{\L{}ukasiewicz logic and its two-dimensional expansions}\label{ssec:Luk}
In this part of Section~\ref{sec:logics} we will develop two-layered modal logics to reason about probabilities, and belief and plausibility functions. Before we do so, we concentrate on two \emph{outer logics} of the two-layered framework, introduce their syntax, semantics, and prove their finite strong standard completeness.

As a~starting point, we have chosen to use \L{}ukasiewicz infinitely-valued logic (cf.~\cite{Hajek1998,CintulaHajekNoguera2015HandbookofMFL1}), which we introduce briefly together with its standard semantics on the real interval $[0,1]$. The choice is motivated by the fact that in the language of \L{}ukasiewicz logic one can express the order and arithmetical operations, and consequently formulate the probability or belief function axioms. For the sake of translations between the linear inequalities-based formalism and the \L{}ukasiewicz logic-based formalism, we consider \L{}ukasiewicz logic extended with globalisation (Baaz delta) operator $\triangle$.
\begin{definition}[Language and standard semantics of $\Luk$ and $\Luk\triangle$]\label{def:Luk_standard_algebra}
The logic $\Luk\triangle$ has the following language $\Langluka$ ($\Luk$ refers to the $\triangle$-free fragment):
$$
\alpha\coloneqq \ p\mid \alpha\rightarrow\beta\mid {\sim}\alpha\mid \triangle\alpha,
$$
with $p\in\Prop_{\Luk}$, and with additional connectives definable as 
\begin{align*}
\alpha\vee\beta &\coloneqq ((\alpha\rightarrow\beta)\rightarrow\beta)  & \alpha\wedge\beta &\coloneqq (\alpha\rightarrow\beta) \& \alpha \\
\alpha\oplus\beta &\coloneqq {\sim}\alpha\rightarrow\beta & \alpha\ominus\beta &\coloneqq {\sim}(\alpha\rightarrow\beta)\\
\alpha\&\beta &\coloneqq {\sim}(\alpha\rightarrow{\sim}\beta) & \bot &\coloneqq {\sim} (p\rightarrow p)
\end{align*}
We interpret the language on the standard MV algebra $[0,1]_\Luk = ([0,1],\rightarrow_\Luk,{\sim_\Luk},\triangle_\Luk)$, with the following operations (we list some of the defined connectives as well for convenience)
\begin{align*}
  a~\&_{\Luk} b &\coloneqq \max \{0, a+b-1 \}  &  a~\rightarrow_{\Luk} b &\coloneqq \min \{ 1, 1 - a+ b \} \\
  a~\oplus_{\Luk} b &\coloneqq \min \{1, a~+ b\}  &  a~\ominus_{\Luk} b &\coloneqq \max \{0, a~- b\} \\
  a~\wedge b &\coloneqq \min \{ a,b\} & a~\vee b &\coloneqq \max \{ a,b\}  \\
\triangle_{\Luk} a~&\coloneqq \begin{cases}
1,\text{ if }a~=1\\
0 \ \ \text{else}
\end{cases} & {\sim_{\Luk}}a &\coloneqq 1 - a
\end{align*}
The logic $\Luk$ ($\Luk\triangle$) is semantically given as the $1$-preserving  consequence relation over $[0,1]_\Luk$ in the corresponding language\footnote{This consequence is \emph{infinitary} \cite{Hajek1998}. Therefore one in general only obtains strong completeness for its finitary restriction.}. 
\end{definition}
The logic $\Luk$ can be axiomatised by the \L{}ukasiewicz axioms in the language $\{\rightarrow,{\sim}\}$, or equivalently as an extension of the basic fuzzy logic BL with the involution axiom \cite{Hajek1998}. $\Luk\triangle$ is axiomatised by extending $\Luk$ with additional $\triangle$-axioms and $\triangle$-necessitation rule. Both the logics $\Luk$ and $\Luk\triangle$ are known to be \emph{finitely} strongly standard complete w.r.t their standard semantics.
The axiomatisations, together with a~proof of their finite strong standard completeness (FSSC), can be found e.g. in \cite{Hajek1998} or \cite{CintulaHajekNoguera2015HandbookofMFL1}. 

We will expand the language of $\Luk$, and extend the logics $\Luk$, $\Luk\triangle$ in the following sections, to encompass a~two-dimensional nature of positive / negative information reflected in the uncertainty measures developed in Subsections~\ref{ssec:probabilities} and~\ref{ssec:two:dimension:belief}. The axiomatisations of $\Luk,\Luk\triangle$ will therefore appear as part of our later definitions, and we will
comment on and refer to the particular places in the literature throughout the following section.

To encompass a~two-dimensional nature of the reasoning semantically, we will use two different expansions of the \emph{twist product} algebra $[0,1]^{\Join}$, i.e. twist product\footnote{In the context of Nelson's paraconsistent logics such product construction has been called twist product of algebras \cite{Vakarelov1977}, or twist structures~\cite[Chapter~8]{Odintsov2008}.} of the lattice $([0,1],\min,\max)$, as the algebra of values of the outer layer.
\begin{definition}[{Twist product $[0,1]^{\Join}$}]\label{def:twistproduct}
The twist product algebra $[0,1]^{\Join}$ with the truth (vertical) order $\leq$ and the information (horizontal) order $\leq_i$ is defined as $[0,1]^{\Join} \coloneqq ([0,1]\times [0,1]^{op},\wedge,\vee,\neg)$, where

\begin{minipage}{.5\textwidth}
        \begin{align*}
      (a_1,a_2)\wedge (b_1,b_2) &= (a_1\wedge b_1,a_2\vee b_2) \\
      (a_1,a_2)\vee (b_1,b_2) &= (a_1\vee b_1,a_2\wedge b_2)\\
      \neg(a_1,a_2) &= (a_2,a_1)\\
      (a_1,a_2)\leq (b_1,b_2) &\mbox{ iff } a_1\leq_{[0,1]} b_1 \mbox{ and } b_2\leq_{[0,1]} a_2\\
      (a_1,a_2)\leq_i (b_1,b_2) &\mbox{ iff } a_1\leq_{[0,1]} b_1 \mbox{ and } a_2\leq_{[0,1]} b_2
      \end{align*}
    \end{minipage}%
    \begin{minipage}{0.5\textwidth}
        \centering
        \begin{tikzpicture}[-,>=stealth,shorten >=0.5pt,auto,node distance=1.2cm,thin,
	main node/.style={circle,draw,font=\sffamily\normalsize},
	]
	\node[main node][label=left:{$(0,0)$}] (1) {};
	\node[main node][label={$(1,0)$}] (2) [above right of=1] {};
	\node[main node][label=below:{$(0,1)$}] (3) [below right of=1] {};
	\node[main node][label=right:{$(1,1)$}] (4) [above right of=3] {};
	\path[every node/.style={font=\sffamily\small}]
	(1) edge (2)
	edge (3)
	(2) edge (4)
	(3) edge (4);
	\path[dotted]
	(2) edge (3)
	(1) edge (4);
\end{tikzpicture}
\end{minipage}
\end{definition}
We will use different expansions of the twist product $[0,1]^{\Join}$ with \L{}ukasiewicz-derived connectives as the algebras (or to be precise, matrices) for the outer logic. There are some design choices to make, most importantly the following: (i) the choice of the set of designated values, and (ii) the choice how to interpret the additional connectives (mainly, the implication) in their second coordinate, i.e. choosing how they are to be negated with $\neg$, in other words, what type of information constitutes grounds for them to be refuted. 

As for (i), we will use the following two natural choices: first, we can use the singleton $\{(1,0)\}$, i.e. the top of the underlying truth order as the only designated value. This entails that only formulas whose truth is absolutely supported \emph{and} their falsity absolutely refuted are going to be valid. This in general yields a~logic whose $(\wedge,\vee,\neg)$ fragment coincides with the Exactly-true logic $\ETL$~\cite{PietzRiveccio2013}. Another choice is to put the filter $(1,1)^{\uparrow} = \{a\mid a_1 = 1\}$ as the set of designated values. This entails that all formulas whose truth is absolutely supported are valid, so, the concept of validity only employs the first coordinate of their semantical value. This in general yields a~logic whose $(\wedge,\vee,\neg)$ fragment coincides with the logic $\BD$.

As for (ii), there are (at least) two natural ways of negating implication:
(a) a~'de Morgan' way, dualising implication with a~\emph{co-implication} (which in case of \L{}ukasiewicz logic is $\ominus_\Luk$) as
        $$ 
        \neg(a\rightarrow b) \coloneqq (\neg b \ominus\neg a),
        $$
and (b) a~'Nelson's way, combining positive and negative semantical values with a~suitable conjunction as
        $$ 
        \neg(a\rightarrowtriangle b) \coloneqq (a\ \&\ \neg b).
        $$
The former mimics one of the options on how to negate implications listed in \cite{Wansing2008},
the latter is inspired directly by how implication is negated in Nelson's paraconsistent logic $\N 4$ \cite{Nelson1949}, and it yields an implication whose associated equivalence connective is not congruential, and it is therefore referred to as a~weak implication (and we will use a~different symbol to denote it). 

Logics of both kinds have been considered in \cite{BilkovaFrittellaKozhemiachenko2021TABLEAUX}, namely the two logics 
$\Luk^2_{(1,0)}(\rightarrow)$
and $\Luk^2_{(1,1)^\uparrow}(\rightarrowtriangle)$ (the lower index pointing to the set of designated values). There, constraint tableaux calculi were provided to decide their sets of tautologies and both logics were shown to be coNP-complete. The logic $\Luk^2_{(1,0)}(\rightarrow)$ has previously been introduced and proven complete in \cite{BilkovaFrittellaMajerNazari2020} under the name $\Luk_\neg$.

In this paper, we employ the logic $\Luk^2_{(1,0)}(\rightarrow)$, extended with a~$\triangle$ operator, to reason about probabilities or belief functions over the logic $\BD$. We will refer to this logic as a~Belnapian extension of $\Luk$, and describe it in detail in the following Subsection~\ref{ssec:belnapluk}. Next, we will employ the logic $\Luk^2_{(1,1)^\uparrow}(\rightarrowtriangle)$ to reason about belief and plausibility as described in the Subsection~\ref{sssec:bel-pl-modalities:(bel+,pl-)}. We will refer to this logic as a~Nelson-style extension of $\Luk$, and describe it in detail in the Subsection~\ref{ssec:Nelsonluk}.

\subsubsection{A Belnapian extension of \L{}ukasiewicz logic}\label{ssec:belnapluk}

In this paper, we first opt to use the logic $\Luk^2_{(1,0)}(\rightarrow)$ of $\cite{BilkovaFrittellaKozhemiachenko2021TABLEAUX}$, expanded with a~globalisation $\triangle$ operator. We will skip mentioning the language in the index throughout this subsection for simplicity of notation, and use $\Luk^2$ to refer to the logic.

\begin{definition}[Language and semantics of $\Luk^2$]\label{def:language_semantics_luksquare}

The logic $\Luk^2$ has the following language $\Langluk$:
$$
\alpha\coloneqq \ p\mid \alpha\rightarrow\beta\mid {\sim}\alpha\mid \triangle\alpha\mid \neg\alpha,
$$
with $p\in\Prop_{\Luk}$, and with additional connectives definable as 
\begin{align*}
\alpha\vee\beta &\coloneqq ((\alpha\rightarrow\beta)\rightarrow\beta)  & \alpha\wedge\beta &\coloneqq (\alpha\rightarrow\beta) \& \alpha \\
\alpha\oplus\beta &\coloneqq {\sim}\alpha\rightarrow\beta & \alpha\ominus\beta &\coloneqq {\sim}(\alpha\rightarrow\beta)\\
\alpha\&\beta &\coloneqq {\sim}(\alpha\rightarrow{\sim}\beta) & \bot &\coloneqq {\sim} (p\rightarrow p)
\end{align*}
To interpret the language, we expand the twist product 
$[0,1]^{\Join}$ 
with the following operations derived from the standard MV algebra (we list the semantics for some of the derived connectives as well for convenience):
\begin{align*}
(a_1,a_2)\rightarrow (b_1,b_2) &= (a_1\rightarrow_{\Luk}b_1,b_2\ominus_{\Luk}b_1) & 
(a_1,a_2)\ominus (b_1,b_2) &= (a_1\ominus_{\Luk}b_1,b_2\rightarrow_{\Luk}b_1) \\
(a_1,a_2)\&(b_1,b_2) &= (a_1\&_{\Luk} b_1, a_2\oplus_{\Luk} b_2) & (a_1,a_2)\oplus(b_1,b_2) &= (a_1\oplus_{\Luk} b_1, a_2\&_{\Luk} b_2) \\
{\sim}(a_1,a_2) &= (\sim_{\Luk}a_1,\sim_{\Luk}a_2) & \triangle(a_1,a_2) &= (\triangle_{\Luk} a_1, {\sim_{\Luk}}\triangle_{\Luk}{\sim_{\Luk}} a_2)\\
\neg(a_1,a_2) &= (a_2,a_1) 
\end{align*}
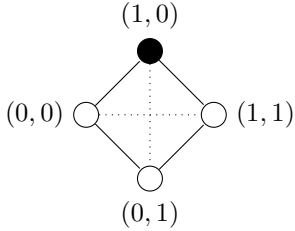
\begin{minipage}{0.3\textwidth}
\begin{tikzpicture}[-,>=stealth,shorten >=0.5pt,auto,node distance=1.2cm,thin,
	main node/.style={circle,draw,font=\sffamily\normalsize},main node2/.style={circle,fill=black,draw,font=\sffamily\normalsize}
	]
	\node[main node][label=left:{$(0,0)$}] (1) {};
	\node[main node2][label={$(1,0)$}] (2) [above right of=1] {};
	\node[main node][label=below:{$(0,1)$}] (3) [below right of=1] {};
	\node[main node][label=right:{$(1,1)$}] (4) [above right of=3] {};
	\path[every node/.style={font=\sffamily\small}]
	(1) edge (2)
	edge (3)
	(2) edge (4)
	(3) edge (4);
	\path[dotted]
	(2) edge (3)
	(1) edge (4);
	
	\end{tikzpicture}
\end{minipage}
\begin{minipage}{0.7\textwidth}
We put the singleton $\{(1,0)\}$ to be the set of designated values and denote the resulting matrix $([0,1]^{\Join}_\Luk,\{(1,0)\})$. Observe, that we can understand the $\neg$-negation geometrically as the symmetry along the horizontal line, while the ${\sim}$-negation ${\sim}(x,y) = (1-y,1-x)$ geometrically denotes the symmetry along the middle point $(0.5,0.5)$.
\end{minipage}
The semantical consequence of $\Luk^2$ is defined as $\{(1,0)\}$ preservation over all valuations $v:\Prop_\Luk \to [0,1]^{\Join}_\Luk$ (extended to the whole language $\Langluk$ as expected):
$$
\Gamma \vDash_{[0,1]^{\Join}_\Luk} \alpha \ \ \mbox{ iff } \ \ \forall v (v[\Gamma]\subseteq \{(1,0)\} \Rightarrow v(\alpha) = (1,0)).
$$
\end{definition}
\begin{convention}
Given a valuation $v$ s.t.\ $v(\phi)=(x,y)$, we will further write $v_1(\phi)=x$ and $v_2(\phi)=y$ when dealing with the coordinates of $v$.
\end{convention}
\begin{remark}[]\label{}
The matrix $([0,1]^{\Join}_\Luk,\{(1,0)\})$ can be equivalently seen as the product MV algebra $[0,1]_\Luk \times [0,1]_\Luk^{op}$, where $[0,1]_\Luk$ is the standard MV algebra and $[0,1]_\Luk^{op}$ is its dual: the operations and order of the standard MV algebra $([0,1],\wedge,\vee,\rightarrow_\Luk,\&_\Luk,\ominus_\Luk,\oplus_\Luk, {\sim_\Luk})$, are turned upside down and dualized in the dual algebra as $([0,1]^{op},\vee,\wedge,\ominus_\Luk,\oplus_\Luk,\rightarrow_\Luk,\&_\Luk,{\sim_\Luk})$. The natural dual of the globalisation Baaz delta operator $\triangle_\Luk$ ("definitely true") is then the operator ${\sim_\Luk}\triangle_\Luk{\sim_\Luk}$, which results in $0$ if its argument is $0$, and in $1$ otherwise.
\end{remark}

\begin{remark}[The two-dimensional delta operators]\label{rem:delta(1,0)detecting}
Observe that $\triangle(a_1,a_2)$ can only result in the four values $(1,0), (1,1), (0,0), (0,1)$ in $[0,1]^{\Join}_\Luk$: $\triangle (a_1,a_2) = (1,0)$ if and only if $((a_1,a_2) = (1,0))$, $\triangle (a_1,a_2) = (1,1)$ if and only if $a_1= 1,a_2\neq 0$, $\triangle (a_1,a_2) = (0,0)$ if and only if $a_1\neq 1,a_2 = 0$, and $\triangle (a_1,a_2) = (0,1)$ else. Defining the following operator $$
\triangletop\alpha \coloneqq \triangle\alpha\wedge {\sim}\neg\triangle\alpha,
$$ 
we obtain the $(1,0)$-detecting operator in $[0,1]^{\Join}_\Luk$: it has value $(1,0)$ iff the argument has the value $(1,0)$, and is $(0,1)$ everywhere else.

It is well known that the \L{}ukasiewicz language expanded with $\triangle$ can define the G\"{o}del implication (cf.~\cite[CH.I,~2.2.1]{BehounekCintulaHajek2011MFL1}) as $\alpha\rightarrow_G \beta \coloneqq \triangle(\alpha\rightarrow\beta)\vee\beta$. Similarly, co-implication can be defined as $\beta\Yleft_{G}\alpha \coloneqq \beta \wedge {\sim}\triangle(\beta\rightarrow\alpha)$. Thus in $[0,1]^{\Join}_\Luk$, 
\begin{align*}
    \triangle(a\rightarrow b)\vee b &= (\triangle(a_1\rightarrow_\Luk b_1)\vee b_1, {\sim}\triangle{\sim}(b_2\ominus_\Luk a_2)\wedge b_2)\\
    &= (\triangle(a_1\rightarrow_\Luk b_1)\vee b_1, {\sim}\triangle(b_2\rightarrow a_2)\wedge b_2)\\
    &= (a_1\rightarrow_G b_1, b_2\Yleft_{G} a_2)
\end{align*}
defines the implication (and consequently also other connectives) of the two-dimensional bi-G\"{o}del logic $\mathsf{G}_{(1,0)}^2(\rightarrow)$ introduced in \cite{BilkovaFrittellaKozhemiachenko2021TABLEAUX}
to reason about comparative uncertainty.

It is worth observing, that due to $\triangletop$, we can express both the order and the strict order of the $[0,1]^{\Join}_\Luk$: First, $a\rightarrow b$ is designated if and only if $a\leq b$, thus $\triangletop(a\rightarrow b)$ is $(1,0)$ if $a\leq b$, and it is $(0,1)$ otherwise. Next, the formula $\triangletop(a\rightarrow b)\wedge{\sim}\triangletop(b\rightarrow a)$ is $(1,0)$ if $a< b$, and it is $(0,1)$ otherwise.
\end{remark}
\begin{definition}[Axiomatisation of $\Luk^2$] \label{def:ax_luk_rightarrow}
The logic $\Luk^2$ is axiomatised by the following axioms and rules:
\begin{center}
\begin{tabular}{llll}
(w) & $\alpha \rightarrow (\beta\rightarrow\alpha)$ & ($\triangle$1) & $\triangle\alpha\vee {\sim}\triangle\alpha$ \\
 (sf) & $(\alpha\rightarrow\beta) \rightarrow ((\beta\rightarrow\gamma)\rightarrow(\alpha\rightarrow\gamma))$ & ($\triangle$2) & $\triangle\alpha \rightarrow \alpha$ \\
(waj) & $((\alpha\rightarrow\beta)\rightarrow\beta) \rightarrow ((\beta\rightarrow\alpha)\rightarrow\alpha)$ & ($\triangle$3) & $\triangle\alpha \rightarrow \triangle\triangle\alpha$ \\
 (co) & $({\sim}\beta\rightarrow{\sim}\alpha) \rightarrow (\alpha\rightarrow\beta)$ & ($\triangle$4) & $\triangle(\alpha\vee\beta) \rightarrow \triangle\alpha\vee\triangle\beta$ \\
 (dn$\neg$) & $\neg\neg\alpha \leftrightarrow \alpha$ & ($\triangle$5) & $\triangle(\alpha\rightarrow\beta) \rightarrow \triangle\alpha\rightarrow \triangle\beta$ \\
 ($\neg{\sim}$) & $\neg{\sim}\alpha \leftrightarrow {\sim}\neg\alpha$ & ($\neg\triangle$) & $\neg\triangle\alpha \leftrightarrow {\sim}\triangle{\sim}\neg\alpha$\\
(K${\sim}\neg$) & $({\sim}\neg\alpha\rightarrow{\sim}\neg\beta) \leftrightarrow {\sim}\neg(\alpha\rightarrow\beta)$ & (Nec) & $\alpha \vdash \triangle\alpha$ \\
 ($\vee$) & $\alpha\vee\beta \leftrightarrow ((\alpha\rightarrow\beta)\rightarrow\beta)$
& (Conf) & $\alpha \vdash {\sim}\neg\alpha$\\
 (MP) & $\alpha, \alpha\rightarrow\beta \vdash \beta $ 
\end{tabular}
\end{center}
\end{definition}
The axioms of weakening, suffixing, Wajsberg's axiom, and (converse) contraposition (w, sf, waj, co) together with the rule (MP) completely axiomatise \L{}ukasiewicz logic in the language $\{\rightarrow,{\sim}\}$ (cf. \cite{Hajek1998}). Axiom ($\vee$) is the definition of disjunction, which we need to formulate two of the the $\triangle$ axioms. 

The delta axioms ($\triangle$1--5) and the (Nec) rule above, except the additional ($\neg\triangle$) axiom, present the usual complete axiomatisation of Baaz Delta operator over \L{}ukasiewicz logic (cf.~\cite[Def.~2.4.5]{Hajek1998} or \cite[CH I,2.2.1]{BehounekCintulaHajek2011MFL1}). 

First, observe that the $\neg$ negations can provably be pushed to the atomic formulas, and we can thus consider formulas, up to provable equivalence, in a~\emph{negation normal form (NNF)}, i.e.\ formulas built using $\{\rightarrow, \sim,\triangle\}$ from \emph{literals} of the form $p, \neg p$:
\begin{lemma}[NNF]\label{lem:nnf_luk_rightarrow}
For each formula $\alpha$ in $\Langluk$ there is a~formula in $\neg$-negation normal form $\alpha^\star$ such that
$$
\vdash_{\Luk^2} \alpha \leftrightarrow \alpha^\star.
$$
\end{lemma}
\begin{proof}
 We have the axiom ($\neg\triangle$), and we have $\neg{\sim}\alpha\leftrightarrow {\sim}\neg\alpha$ and $\neg (\alpha\rightarrow\beta)\leftrightarrow {\sim}{\sim}\neg (\alpha\rightarrow\beta)\leftrightarrow {\sim}({\sim}\neg\alpha\rightarrow {\sim}\neg\beta)$ provable. A~procedure can therefore be defined which turns each $\alpha$ into $\alpha^\star$ in nnf, so that we can prove, by induction, that $({\sim}\alpha)^\star\leftrightarrow {\sim}\alpha^\star$, $(\triangle\alpha) ^\star\leftrightarrow \triangle\alpha^\star$, and $ (\alpha\rightarrow\beta)^\star\leftrightarrow{\sim}({\sim}\alpha^\star\rightarrow {\sim}\beta^\star)$.
\end{proof}
\begin{remark}[The conflation]\label{rem:luksquare_conflation}
The composition of the two negations ${\sim}\neg$ works semantically as follows: ${\sim}\neg(a_1,a_2) = (1-a_2,1-a_1)$, and we can geometrically understand it on the $[0,1]^{\Join}_\Luk$ algebra as the symmetry along the vertical line. Such operation is often referred to as \emph{conflation} (cf., e.g.~\cite{Fitting1994,OmoriSano2015}), and it is a~form of an involutive information negation (indeed, we have provably ${\sim}\neg{\sim}\neg\alpha \leftrightarrow \alpha$). 

One can easily prove in the calculus above that conflation distributes with all the \L{}ukasiewicz connectives $\{\rightarrow, {\sim}, \triangle, \&, \wedge, \vee, \oplus, \ominus\}$ (it suffices to use $(K{\sim}\neg,\neg{\sim})$ and $(\neg\triangle)$ axioms, together with the definitions of the derived connectives).
\end{remark}
\begin{remark}[The $\triangle$ and $\triangletop$ operators]\label{rem:luksquare_deltas}
Observe that the $\triangle$ operator has some useful additional properties (we list only those we will need later: proofs of most of them in $\Luk$ can be found e.g., in~\cite[Lemma 2.4.11]{Hajek1998}, and they can also similarly be proven in the calculus above):
\begin{align*}
    &\vdash_{\Luk^2}\triangle\alpha \leftrightarrow \triangle(\alpha\& \alpha) \leftrightarrow \triangle\alpha \& \triangle\alpha\\
    &\vdash_{\Luk^2}\triangle(\alpha\& \beta) \leftrightarrow \triangle\alpha \& \triangle\beta \leftrightarrow \triangle(\alpha\wedge \beta)\\
    &\vdash_{\Luk^2}(\triangle\alpha\rightarrow(\triangle\alpha \rightarrow\beta)) \rightarrow (\triangle\alpha\rightarrow\beta)
\end{align*}
The last formula is a~form of the contraction principle for $\triangle$-formulas, and its provability follows immediately from the $(\triangle 1)$ and $(\vee)$ axioms.

Observing that $\triangletop\alpha\coloneqq\triangle\alpha\wedge{\sim}\neg\triangle\alpha$ and thus, by the above principles and conflation distribution over $\triangle$, provably $\triangletop\alpha\leftrightarrow\triangle(\alpha\wedge{\sim}\neg\alpha)$, one can consequently prove all the above principles for $\triangletop$-formulas as well. Namely, the contraction principle $(\triangletop\alpha\rightarrow(\triangletop\alpha\rightarrow\beta)) \rightarrow (\triangletop\alpha\rightarrow\beta)$ is provable. 
Moreover, the  $(\triangle 1-2, \triangle 4-5)$ axioms and the (Nec) rule with $\triangle$ replaced by $\triangletop$ remain both sound and derivable. It is worth noting that it is \emph{not} the case for the remaining $(\triangle 3)$ axiom. As an additional principle we can prove that $\triangletop$-formulas are fixpoints of conflation: ${\sim}\neg\triangletop\alpha \leftrightarrow \triangletop\alpha$.
\end{remark}

Applying the insights remarked on above, we can prove that the logic $\Luk^2$ has the following $\triangletop$-deduction theorem (cf.~\cite[Th.~2.4.14]{Hajek1998} or \cite[Ch.I, Th.~2.2.1]{BehounekCintulaHajek2011MFL1}): 
\begin{lemma}[DT$\triangletop$]\label{lemma:DT__luk_rightarrow}
$$\Gamma,\alpha\vdash_{\Luk^2} \beta\ \  \mbox{iff}\ \  \Gamma\vdash_{\Luk^2} \triangletop\alpha\rightarrow\beta.$$
\end{lemma}

\begin{proof}
The proof of the theorem is rather standard, by induction on the length $n$ of the proof of $\Gamma,\alpha\vdash_{\Luk^2} \beta$.
\begin{description}
\item[$n=1$:] In this case $\beta$ is either an instance of an axiom, or an assumption form $\Gamma$, or $\alpha = \beta$. In the first two cases, the result follows by the (w) axiom. The remaining case uses that $\triangletop\alpha\rightarrow\alpha$ is provable.
\item[$n>1$:] We assume that for all shorter proofs the results holds, and let $\gamma_1$, \ldots, $\gamma_n, \beta$ be a~proof of $\beta$ from $\Gamma\cup\{\alpha\}$. We distinguish three cases: 
\begin{description}
\item[(Conf)] Let $\beta\coloneqq{\sim}\neg\gamma_n$ be obtained by the conflation rule. By the induction hypothesis, $\triangletop\alpha\rightarrow\gamma_n$ is provable. By an application of (Conf), (K${\sim}\neg$) and (MP) we obtain ${\sim}\neg\triangletop\alpha\rightarrow{\sim}\neg\gamma_n$. By ${\sim}\neg\triangletop\alpha \leftrightarrow \triangletop\alpha$ (see Remark~\ref{rem:luksquare_deltas}) it follows that 
$\triangletop\alpha \rightarrow {\sim}\neg\gamma_n$ as desired.
\item[(Nec)] Let $\beta\coloneqq\triangle\gamma_n$ be obtained by the necessitation rule. By the induction hypothesis, $\triangletop\alpha\rightarrow\gamma_n$ is provable. By an application of (Nec), using provability of $\triangletop(\triangletop\alpha\rightarrow\gamma_n)\rightarrow(\triangletop\triangletop\alpha\rightarrow\triangletop\gamma_n)$, and (MP) we obtain $\triangletop\triangletop\alpha\rightarrow\triangletop\gamma_n$. Using provability of $\triangletop\triangletop\alpha\leftrightarrow\triangletop\alpha$ (see Remark~\ref{rem:luksquare_deltas}) we conclude $\triangletop\alpha\rightarrow\triangletop\gamma_n$ as desired.
\item[(MP)] Let $m<n$ and let $\gamma_m,\gamma_n = \gamma_m\rightarrow\beta / \beta$ be an instance of modus ponens. By the induction hypothesis we have $\triangletop\alpha\rightarrow\gamma_m$ and $\triangletop\alpha\rightarrow(\gamma_m\rightarrow\beta)$ provable. From the latter we have also $\gamma_m\rightarrow(\triangletop\alpha\rightarrow\beta)$ provable. Thus, by (suf) and (MP) we obtain $\triangletop\alpha\rightarrow(\triangletop\alpha\rightarrow\beta)$. By the contraction principle of Remark~\ref{rem:luksquare_deltas} we finally obtain $\triangletop\alpha\rightarrow\beta$ as desired.
\end{description}
\end{description}
\end{proof}
We proceed to prove finite strong standard completeness (FSSC for short) of $\Luk^2$, i.e. strong completeness of the finite consequence of $\Luk^2$ w.r.t.\ the matrix $([0,1]^{\Join}_\Luk,\{(1,0)\})$.
We do so by reducing the provability in $\Luk^2$ to that in $\Luk\triangle$, which itself is known to be finite strong standard complete, i.e. finite strong complete w.r.t.\ the standard MV algebra $[0,1]_\Luk$ (see e.g. \cite[Th. 3.2.14]{Hajek1998}): 
\begin{lemma}[Reducing $\Luk^2$ to $\Luk\triangle$]\label{lem:luk_rightarrow_luk_delta}
For any \emph{finite} set of formulas $\Gamma,\alpha$ in a~NNF,
\begin{center}
\begin{tabular}{c c}
 $\Gamma\vdash_{\Luk^2}\alpha \ $ iff \ & for some finite $\Sigma:\ \boxdot\Gamma,\Sigma\vdash_{\Luk\triangle}\alpha,$
\end{tabular}
\end{center}
where ${\sim}\neg\Gamma\coloneqq\{{\sim}\neg\gamma\mid\gamma\in\Gamma\}$ and $\boxdot\Gamma \coloneqq {\sim}\neg\Gamma\cup\Gamma$, and $\Sigma$ contains instances of $\neg$-axioms. 
\end{lemma}
\begin{proof}
The right-left direction is almost trivial: $\Luk\triangle$ is a~subsystem of $\Luk^2$, and all the axioms in $\Sigma$ are provable in $\Luk^2$, and, thanks to the ${\sim}\neg$-rule, $\Sigma\vdash_{\Luk^2}\boxdot \gamma$ for each $\gamma\in\Sigma$.

For the other direction, we proceed in a~few steps. First, we denote by $\vdash_{\Luk^{2-}}$ provability in $\Luk^2$ \emph{without} the ${\sim}\neg$-rule. By routine induction on proofs (and using that ${\sim}\neg$ distributes from/to implications, negations, and $\triangle$), we can see that
\begin{center}
\begin{tabular}{c c}
 $\Gamma\vdash_{\Luk^2}\alpha \ $ iff & \ $\boxdot\Gamma\vdash_{\Luk^{2-}}\alpha.$
\end{tabular}
\end{center}
Then we can list all the instances of $\neg$-axioms in the proof above in $\Sigma$, and obtain:
\begin{center}
\begin{tabular}{c c}
 $\boxdot\Gamma\vdash_{\Luk^{2-}}\alpha \ $ iff & \ $ \boxdot\Gamma,\Sigma\vdash_{\Luk\triangle}\alpha.$
\end{tabular}
\end{center}
First, note that we can include in $\Sigma$ all instances of $\neg$-axioms for all subformulas of $\Gamma,\alpha$ as well and still keep the Lemma valid. This will come in handy in the following completeness proof. Second, we stress that in the final proof $\boxdot\Sigma,\Sigma\vdash_{\Luk\triangle}\alpha$ in $\Luk\triangle$, we still use the language of $\Luk^2$, where formulas starting with $\neg$ are seen from the point of view of $\Luk\triangle$ as atomic.
\end{proof}
Lemma~\ref{lem:luk_rightarrow_luk_delta} provides a~translation of provability in $\Luk^2$ to pro\-va\-bility in $\Luk\triangle$, and allows us in particular to observe that the extension of $\Luk\triangle$ by $\neg$ is conservative. Now, using finite strong standard completeness of $\Luk\triangle$, we can see that $\Luk^2$ is finitely strongly standard complete:
\begin{lemma}[FSSC]\label{lemma:luk_rightarrow_completeness} For a~\emph{finite} set of formulas $\Gamma$,
\begin{center}
\begin{tabular}{c c}
 $\Gamma\vdash_{\Luk^2}\alpha$\ \ iff &\ \ $\Gamma\vDash_{[0,1]^{\Join}_\Luk}\alpha.$
\end{tabular}
\end{center}
\end{lemma}
\begin{proof}
The left-right direction expresses soundness and consists of checking that all the axioms are valid and all the rules sound. We only do some cases:

First the ${\sim}\neg$-rule: assume that $v$ is given and $v(\alpha)=(1,0)$. Then 
$$
v({\sim}\neg\alpha) = {\sim}\neg (1,0) = {\sim}(0,1) = (1,0).
$$
Next, for any $v$, 
\begin{align*}
   v({\sim}\neg(\alpha\rightarrow\beta)) &= {\sim}\neg(v(\alpha)\rightarrow v(\beta))\\
   &= {\sim}\neg(v_1(\alpha)\rightarrow_{\Luk}v_1(\beta), {\sim_{\Luk}}(v_2(\beta)\rightarrow_{\Luk} v_2(\alpha)))\\
   &= ((v_2(\beta)\rightarrow_{\Luk} v_2(\alpha)), {\sim_{\Luk}}(v_1(\alpha)\rightarrow_{\Luk}v_1(\beta))),
\end{align*}
and
\begin{align*}
    v({\sim}\neg\alpha\rightarrow{\sim}\neg\beta) &= {\sim}\neg v(\alpha)\rightarrow {\sim}\neg v(\beta) \\
    &= ({\sim_{\Luk}}v_2(\alpha),{\sim_{\Luk}}v_1(\alpha))\rightarrow({\sim_{\Luk}}v_2(\beta),{\sim_{\Luk}}v_1(\beta))\\
    &= ({\sim_{\Luk}}v_2(\alpha)\rightarrow_{\Luk}{\sim_{\Luk}}v_2(\beta),v_1(\alpha) \&_{\Luk} {\sim_{\Luk}}v_1(\beta))\\
    &= ((v_2(\beta)\rightarrow_{\Luk} v_2(\alpha)), {\sim_{\Luk}}(v_1(\alpha)\rightarrow_{\Luk}v_1(\beta))).
\end{align*}
Next, for any $v$, 
$$
v(\neg{\sim}\alpha) = \neg{\sim}v(\alpha) = ({\sim_{\Luk}}v_2(\alpha),{\sim_{\Luk}}v_1(\alpha)) = {\sim}\neg v(\alpha) = v({\sim}\neg\alpha).
$$
Last, for any $v$, 
\begin{align*}
  v(\neg\triangle\alpha) &= \neg (v(\triangle\alpha)) = \neg(\triangle_{\Luk} v_1(\alpha), {\sim_{\Luk}}\triangle_{\Luk}{\sim_{\Luk}} v_2(\alpha)) \\
  &= ({\sim_{\Luk}}\triangle_{\Luk}{\sim_{\Luk}} v_2(\alpha),\triangle_{\Luk} v_1(\alpha)) \\
  &= {\sim}(\triangle_{\Luk}{\sim_{\Luk}} v_2(\alpha),{\sim_{\Luk}}\triangle_{\Luk} v_1(\alpha))\\
  &= {\sim}\triangle({\sim_{\Luk}}v_2(\alpha),{\sim_{\Luk}}v_1(\alpha))\\
  & = {\sim}\triangle{\sim}\neg\alpha.
\end{align*}
We leave the rest of the proof of soundness for the reader.

For the right-left direction, let us assume that $\Gamma\nvdash_{\Luk^2}\alpha$. Then for some finite $\Sigma$ containing instances of $\neg$-axioms (in particular those for subformulas of $\Sigma,\alpha$), we have $\boxdot\Sigma,\Sigma\nvdash_{\Luk\triangle}\alpha.$ Because $\Luk\triangle$ is finitely standard complete, there is an evaluation $v:\mathrm{At} \to [0,1]_{\Luk}$ sending all formulas in $\boxdot\Sigma,\Sigma$ to $1$, while $v(\alpha)< 1$. Here, $\mathrm{At}$ contains literals from $\Sigma,\alpha$ of the form $p,\neg p$, and atoms and formulas of the form $\neg\sigma$ from $\boxdot\Sigma,\Sigma$. We define $v': \Prop_\Luk \to [0,1]^{\Join}_\Luk$ by 
$$
v'(p) = (v(p), v(\neg p)).
$$
Let us denote $\beta^\neg$ the NNF of $\neg\beta$ (i.e. $(\neg\beta)^\star$).
We can then prove, by routine induction, that for each formula $\beta$ a~subformula of $\Sigma,\alpha$, we have 
$$
v'(\beta) = (v(\beta),v(\beta^\neg)).
$$ 
To do so, we use the fact that $v[\Sigma]\subseteq\{1\}$, and $\Sigma$ contains all instances of $\neg$-axioms for all subformulas of $\Sigma,\alpha$, and, subsequently, $\Sigma$ ensures that $v(\neg\beta) = v(\beta^\neg)$.

We now immediately see that $v'(\alpha) < (1,0)$, because $v(\alpha) < 1$. 
To prove that indeed $v'[\Gamma]\subseteq\{(1,0)\}$, we use the fact that $v[\boxdot\Gamma]\subseteq\{1\}$: as for all $\gamma\in\Gamma$, $v({\sim}\neg\gamma) = 1$, $v(\neg\gamma) = v(\gamma^\neg) = 0$. 
For the latter, we again need to use the fact that $v[\Sigma]\subseteq\{1\}$, and $\Sigma$ contains all instances of $\neg$-axioms for all subformulas of $\Gamma$, as they prove, by means of $\Luk\triangle$, that $\neg\gamma\leftrightarrow\gamma^\neg$, and $v$ has to respect that. 
Now we conclude, that for all $\gamma\in\Gamma$, $v'(\gamma) = (v(\gamma),v(\gamma^\neg)) = (1,0)$. 
\end{proof}
\begin{remark}[Changing the designated values]\label{rem:luksquarewith(1,1)}
Changing the set of designated values in $\Luk^2$ to the filter $(1,1)^\uparrow$ yields the logic $\Luk^2_{(1,1)^\uparrow}(\rightarrow)$ of \cite{BilkovaFrittellaKozhemiachenko2021TABLEAUX}, expanded with $\triangle$. It can be axiomatised by removing the conflation rule (Conf) from the axiomatisation of $\Luk^2$ provided in Definition~\ref{def:ax_luk_rightarrow}. Its $(\wedge,\vee,\neg)$ fragment coincides with $\BD$, while the  $(\wedge,\vee,\neg)$ fragment of $\Luk^2$ coincides with $\ETL$. We will not use this logic in this paper, however, it can be proved finitely strongly complete by similar manner as $\Luk^2$.
\end{remark}
\begin{example}[What $\Luk^2$ can express]\label{ex:expressivity_luksquare}
Recall first that, semantically, $\neg$ is the symmetry along the horizontal, $\sim$ is the symmetry along the middle point, and the conflation ${\sim}\neg$ is the symmetry along the vertical line.
Let now $v_1(p)$ quantify the evidence for $p$ and $v_2(p)$ quantify the evidence against $p$. As discussed in Section~\ref{sssec:bel-pl-modalities:(bel+,bel-)}, given any statement $p$, we can determine whether the evidence we have about it is classical, incomplete, or contradictory:

The formula $\neg p\leftrightarrow{\sim} p$, or equivalently $p \leftrightarrow {\sim}\neg p$, defines the vertical line: for each valuation $v$, $v(\neg p\leftrightarrow{\sim} p) = (1,0)$ if and only if $v_1(p) = 1 - v_2(p)$. In a~sense, $p\leftrightarrow{\sim}\neg p$ measures how close our information about $p$ is to the non-controversial classical case. Moreover, from the semantics of $\triangletop$, it is clear that we can test whether the information is classical by
\begin{align}
v(\triangletop(p\leftrightarrow{\sim}\neg p))&=
\begin{cases}
(1,0)&\text{ if }v_1(p)=1-v_2(p)\\
(0,1)&\text{ otherwise}
\end{cases}\label{equ:triangleequivalenceclassicality}
\end{align}
\eqref{equ:triangleequivalenceclassicality} is a~useful tool. Assume, we have some finite set of statements $\{\alpha_1,\ldots,\alpha_n\}$. Then, using finite strong completeness of $\Luk^2$, we can check whether those entail classicality of $p$ by establishing the following inference
\[\alpha_1,\ldots,\alpha_n\vdash_{\Luk^2}\triangletop(p\leftrightarrow{\sim}\neg p).\]
Unfortunately, $p\rightarrow{\sim}\neg p$ proves ${\sim}\neg p\rightarrow p$, and thus we cannot really distinguish the left-hand triangle (incompleteness) from the right-hand triangle (contradictoriness) using these implications. However, 
if $p$ is \emph{not} classical, we can establish whether it is contradictory or incomplete for example as follows:
\begin{align}
v_1(\triangle(p\rightarrow{\sim}\neg p))&=
\begin{cases}
1&\text{ iff }$p$\text{ is incomplete}\\
0&\text{ iff }$p$\text{ is contradictory}
\end{cases}
\label{equ:triangleimplicationtest}
\end{align}
\end{example}

\subsubsection{A Nelson-style extension of \L{}ukasiewicz logic}\label{ssec:Nelsonluk}
For part of reasoning about belief functions and plausibilities, we will employ the logic  $\Luk^2_{(1,1)^\uparrow}(\rightarrowtriangle)$ of \cite{BilkovaFrittellaKozhemiachenko2021TABLEAUX}. The interpretation of implication in this logic takes a~direct inspiration from Nelson's paraconsistent logic.
We will again skip mentioning the language in the index throughout this subsection and use simply $\N\Luk$ to refer to the logic $\Luk^2_{(1,1)^\uparrow}(\rightarrowtriangle)$. 
\begin{definition}[Language and semantics of $\N\Luk$]\label{def:language_semantics_luksnelsonquare}
The logic $\N\Luk$ has the following language~$\LanglukN$:
$$
\alpha\coloneqq \ p\mid \alpha\wedge\beta\mid \alpha\rightarrowtriangle\beta\mid {\sim}\alpha\mid \neg\alpha,
$$
with $p\in\Prop_{\Luk}$. We call the $\rightarrowtriangle$ connective the weak implication. Additional connectives are definable as follows (where $\Rightarrow, \Leftrightarrow$ are called strong implication and strong equivalence):
\begin{align*}
\alpha\&\beta &\coloneqq {\sim}(\alpha\rightarrowtriangle{\sim}\beta) & \alpha\Rightarrow\beta &\coloneqq (\alpha\rightarrowtriangle\beta)\wedge (\neg\beta\rightarrowtriangle\neg\alpha)\\
\alpha\leftrightarrowtriangle\beta &\coloneqq (\alpha\rightarrowtriangle\beta)\wedge(\beta\rightarrowtriangle\alpha) & 
\alpha\Leftrightarrow\beta &\coloneqq (\alpha\leftrightarrowtriangle\beta)\wedge (\neg\alpha\leftrightarrowtriangle\neg\beta)\\
\alpha\oplus\beta &\coloneqq {\sim}\alpha\rightarrowtriangle\beta & \alpha\ominus\beta &\coloneqq {\sim}(\alpha\rightarrowtriangle\beta)\\
\alpha\vee\beta &\coloneqq \neg (\neg\alpha\wedge\neg\beta) & \bot &\coloneqq {\sim}(p\rightarrow p)
\end{align*}
To interpret the language, we expand the bilattice $[0,1]^{\Join}$ with the following operations, derived again from the standard MV algebra (we spell out the semantics of some of the defined connectives for convenience):
\begin{align*}
(a_1,a_2)\rightarrowtriangle (b_1,b_2) &= (a_1\rightarrow_{\Luk}b_1,a_1 \&_\Luk b_2 )
 & 
(a_1,a_2)\&(b_1,b_2) &= (a_1\&_{\Luk} b_1, a_1 \rightarrow_{\Luk}{\sim}b_1)\\
{\sim}(a_1,a_2) &= ({\sim}_{\Luk}a_1,a_1) & 
\neg(a_1,a_2) &= (a_2,a_1) \\
(a_1,a_2)\oplus (b_1,b_2) &= (a_1\oplus_{\Luk}a_2,{\sim_{\Luk}}a_1 \&_{\Luk} b_2 ) & (a_1,a_2)\ominus (b_1,b_2) &= (a_1\ominus_{\Luk}a_2, a_1\rightarrow_{\Luk}b_1)
\end{align*}
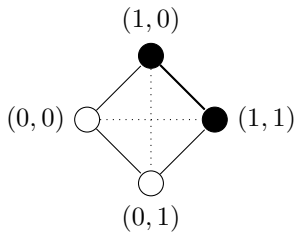
\begin{minipage}{0.3\textwidth}
\begin{tikzpicture}[-,>=stealth,shorten >=0.5pt,auto,node distance=1.2cm,thin,
	main node/.style={circle,draw,font=\sffamily\normalsize},main node2/.style={circle,fill=black,draw,font=\sffamily\normalsize}
	]
	
	\node[main node][label=left:{$(0,0)$}] (1) {};
	\node[main node2][label={$(1,0)$}] (2) [above right of=1] {};
	\node[main node][label=below:{$(0,1)$}] (3) [below right of=1] {};
	\node[main node2][label=right:{$(1,1)$}] (4) [above right of=3] {};
	\path[every node/.style={font=\sffamily\small}]
	(1) edge (2)
	edge (3)
	(2) edge (4)
	(3) edge (4);
	\path[thick]
	(2) edge (4);
	\path[dotted]
	(2) edge (3)
	(1) edge (4);
	\end{tikzpicture}
\end{minipage}
\begin{minipage}{0.7\textwidth}
We put the filter $(1,1)^\uparrow \coloneqq \{(x,y)\mid x=1\}$ to be the set of designated values, and denote the resulting matrix $([0,1]^{\Join}_{\N\Luk},(1,1)^\uparrow)$. Observe, that ${\sim}$ now behaves differently than in the logic $\Luk^2$: namely, its value is always on the classical vertical line. The semantical consequence of $\N\Luk$ is defined as $(1,1)^\uparrow$ preservation over all valuation in this algebra (where $v(\alpha) = (v_1(\alpha),v_2(\alpha))$):
$$
\Gamma \vDash_{[0,1]^{\Join}_{\N\Luk}} \alpha \ \ \mbox{ iff } \ \ \forall v (v_1[\Gamma]\subseteq \{1\} \Rightarrow v_1(\alpha) = 1).
$$
\end{minipage}
\end{definition}
\begin{remark}[Strong and weak implication]\label{rem:lukN_weak_implication}
The connective $\rightarrowtriangle$ is called the weak implication. If $v(\alpha\rightarrowtriangle\beta)$ is designated, then we only know that $v_1(\alpha)\leq v_1(\beta)$.
The derived weak equivalence $\alpha\leftrightarrowtriangle\beta \coloneqq  (\alpha\rightarrowtriangle\beta)\wedge(\beta\rightarrowtriangle\alpha)$ is not congruential, because its provability does not entail equality of the semantical values, only of their first coordinates. It is therefore useful to define the strong counterparts of the two connectives. In particular,     
the strong equivalence $\alpha\Leftrightarrow\beta \coloneqq(\alpha\leftrightarrowtriangle\beta)\wedge (\neg\alpha\leftrightarrowtriangle\neg\beta)$ is congruential. 

Observe that, thanks to the peculiar nature of the logic, some of the usual definitions of \L{}ukasiewicz connectives only work as weak equivalences when lifted to $\N\Luk$: namely, $\alpha\wedge\beta \leftrightarrowtriangle (\alpha\rightarrowtriangle\beta)\&\alpha$ and $\alpha\vee\beta \leftrightarrowtriangle ((\alpha\rightarrowtriangle\beta)\rightarrowtriangle\beta)$ are valid, but it is no longer the case when both sides of the equivalences are $\neg$-negated. 
\end{remark}
\begin{definition}\label{def:ax_luk_rightarrowtriangle}
The logic $\N\Luk$ is axiomatised by the following axioms and rules:
\begin{center}
\begin{tabular}{llll}
(w) & $\alpha \rightarrowtriangle (\beta\rightarrowtriangle\alpha)$ & ($\vee$) & $((\alpha\rightarrowtriangle\beta)\rightarrowtriangle\beta) \leftrightarrowtriangle \alpha\vee\beta$ \\
 (sf) & $(\alpha\rightarrowtriangle\beta) \rightarrowtriangle ((\beta\rightarrowtriangle\gamma)\rightarrowtriangle(\alpha\rightarrowtriangle\gamma))$ & ($\&$) & $\alpha\&\beta \Leftrightarrow {\sim}(\alpha\rightarrowtriangle{\sim}\beta)$ \\
(waj) & $((\alpha\rightarrowtriangle\beta)\rightarrowtriangle\beta) \rightarrowtriangle ((\beta\rightarrowtriangle\alpha)\rightarrowtriangle\alpha)$ & ($\wedge$) & $\alpha\wedge\beta \leftrightarrowtriangle (\alpha\rightarrowtriangle\beta)\&\alpha$ \\
 (co) & $({\sim}\beta\rightarrowtriangle{\sim}\alpha) \rightarrowtriangle (\alpha\rightarrowtriangle\beta)$ & ($\neg\wedge$) & $\neg(\alpha\wedge\beta) \Leftrightarrow \neg\alpha\vee\neg\beta$ \\
 (dn$\neg$) & $\neg\neg\alpha \Leftrightarrow \alpha$ & ($\neg\vee$) & $\neg(\alpha\vee\beta) \Leftrightarrow \neg\alpha\wedge\neg\beta$ \\
 ($\neg{\rightarrowtriangle}$) & $\neg(\alpha\rightarrowtriangle\beta) \leftrightarrowtriangle (\alpha \& \neg\beta)$ & ($\neg{\sim}$) & $\neg{\sim}\alpha \leftrightarrowtriangle \alpha$ \\
 (MP) & $\alpha, \alpha\rightarrowtriangle\beta \vdash \beta $ & &
\end{tabular}
\end{center}
\end{definition}
The axioms (w, sf, waj, co, $\vee$, $\wedge$) with the rule (MP) axiomatise \L{}ukasiewicz logic in terms of weak implication, ($\&$) is the (strong) definition of the fusion. The remaining are $\neg$-axioms, analogous to those of the Nelson's paraconsistent logic $\N 4$.

We can, up to weak equivalence, consider formulas in the $\neg$-normal form:
\begin{lemma}[NNF]
For each formula $\alpha$ in $\LanglukN$ there is a~formula in $\neg$-negation normal form $\alpha^\star$ such that
$$
\vdash_{\N\Luk} \alpha \leftrightarrowtriangle \alpha^\star.
$$
\end{lemma}
As (MP) is the only rule, the logic possesses a~local deduction theorem (LDT) in terms of the weak implication, as in $\Luk$ (cf. \cite[Ch.I,Th.~1.2.10]{BehounekCintulaHajek2011MFL1}).

We proceed to prove the finite strong standard completeness (FSSC) of the logic $\N\Luk$. We do so by means of a~reduction to \L{}ukasiewicz logic, which itself is well-known to be finite strong standard complete, i.e. finite strong complete w.r.t.\ the standard MV algebra $[0,1]_\Luk$ \cite[Th. 3.2.14]{Hajek1998}.
\begin{lemma}[Reduction of $\N\Luk$ to $\Luk$]\label{lem:Nluk_to_luk}
W.l.o.g.\ we assume that $\Gamma,\alpha$ consist of formulas in $\neg$-$\NNF$.

$\Gamma \vdash_{\N\Luk} \alpha$ iff $\Gamma,\Sigma \vdash_{\Luk} \alpha$ for $\Sigma$ a~finite set consisting of instances of $\neg$-axioms.
\end{lemma}
\begin{proof}
This is almost immediate: $\N\Luk$ is an extension of $\Luk$, so any proof in  $\N\Luk$ can be understood as a~proof in $\Luk$, with additional assumptions consisting of all the (finitely many) instances of $\neg$-axioms occurring in the proof.
Observe, that we can equivalently include (finitely many) additional instances of $\neg$-axioms for subformulas of $\Gamma,\alpha$ in $\Sigma$, and the claim remains true. 
\end{proof}
\begin{lemma}[FSSC] \label{lem:FSSC_Luk_weak}
for $\Gamma,\alpha$ a~finite set of formulas in $\LanglukN$,
$$\Gamma \vdash_{\N\Luk} \alpha\ \ \mbox{iff}\ \ \Gamma\vDash_{[0,1]^{\Join}_{\N\Luk}}\alpha.$$
\end{lemma}
\begin{proof}
W.l.o.g.\ we assume that $\Gamma,\alpha$ consist of formulas in $\neg$-negation normal form. Assume that $\Gamma \vdash_{\N\Luk} \alpha$. By the Lemma~\ref{lem:Nluk_to_luk}, it is the case if and only if $\Gamma,\Sigma \vdash_{\Luk} \alpha$ for $\Sigma $ a~finite set consisting of instances of $\neg$-axioms. We will add the following (finitely many) formulas to $\Sigma$: (i) instances of $\neg$-axioms for subformulas of $\Gamma,\alpha$, (ii) for each subformula $\beta$ of $\Gamma,\alpha$ of the form listed on the left, we include an equivalence listed on the right ($\beta^\neg$ denotes the NNF of $\neg\beta$, i.e. $(\neg\beta)^\star$):
\begin{align*}
\beta_1 &\wedge \beta_2 & ((\beta_1)^\neg \vee (\beta_2)^\neg) &\leftrightarrowtriangle (\beta_1 \wedge \beta_2)^\neg \\
\beta_1 &\vee \beta_2 & ((\beta_1)^\neg \wedge (\beta_2)^\neg) &\leftrightarrowtriangle (\beta_1 \vee \beta_2)^\neg \\
\beta_1 &\rightarrowtriangle \beta_2 & (\beta_1 \& (\beta_2)^\neg) &\leftrightarrowtriangle (\beta_1 \rightarrowtriangle \beta_2)^\neg \\
\beta_1 &\& \beta_2 & (\beta_1 \rightarrowtriangle {\sim}\beta_2) &\leftrightarrowtriangle (\beta_1 \& \beta_2)^\neg \\
&{\sim}\beta  & \beta &\leftrightarrowtriangle ({\sim}\beta)^\neg
\end{align*}
We call the resulting finite set $\Sigma'$. First we observe that all the equivalences listed are provable from $\Sigma$ by means of the logic $\Luk$ and $\neg$-axioms instances for subformulas of $\Gamma,\alpha$ (we opt for enlisting them explicitly for convenience). In particular, we have:
$$
\Gamma,\Sigma \vdash_{\Luk} \alpha\ \ \mbox{iff}\ \ \Gamma,\Sigma' \vdash_{\Luk} \alpha.
$$
Assume now that $\Gamma,\Sigma' \nvdash_{\Luk} \alpha$. By the FSSC of $\Luk$, we know there is an evaluation $v: \Phi \to [0,1]$, where $\Phi$ consists of (i) literals from $\Gamma,\Sigma',\alpha$, and (ii) $\neg$-subformulas of $\Sigma'$, so that:
$v(\Gamma\cup\Sigma') \subseteq \{1\}$ while $v(\alpha)<1$. 
We define an evaluation $v': \Prop_\Luk \to [0,1]^{\Join}_{\N\Luk}$ by 
$$
v'(p)\coloneqq (v(p),v(\neg p)).
$$
Next, we can easily prove by induction, for each subformula $\beta$ of $\Gamma,\alpha$, that the extended valuation satisfies
$$
v'(\beta)\coloneqq (v(\beta),v(\beta^\neg)).
$$
The proof essentially uses the fact that formulas contained in $\Sigma '$ are $v$-evaluated by $1$. 

To conclude, we observe that because $v(\alpha)<1$, by the above we have $v'_1(\alpha)<1$. Similarly, because for each $\gamma\in \Gamma$ $v(\gamma)=1$, we have that $v'_1(\gamma)=1$, making all formulas from $\Gamma$ designated and $\alpha$ not designated, as desired.
\end{proof}
\begin{example}[What we can express in $\N\Luk$]\label{ex:expressivity_luknelson}
Recall first that for each $v$, $v({\sim}p)$ is always on the vertical, classical line. Observe that, in contrast with $\Luk^2$, we can now directly express incompleteness and contradictoriness: indeed,  $v({\sim}p\rightarrowtriangle\neg p)$ is designated if and only if $v(p)$ lays within the right triangle, $v(\neg p\rightarrowtriangle{\sim}p)$ is designated if and only if $v(p)$ lays within the left triangle, and $v({\sim}p\leftrightarrowtriangle\neg p)$ is designated if and only if $v(p)$ lays on the vertical line.
In $\N\Luk$, we can thus define ‘$p$ is not incomplete’ as ${\sim}p\rightarrowtriangle\neg p$, and ‘$p$~is not contradictory’ as $\neg p\rightarrowtriangle{\sim}p$, and '$p$ is classical' as ${\sim}p\leftrightarrowtriangle\neg p$.

Moreover, $v((p\rightarrowtriangle q)\wedge(\neg p\rightarrowtriangle \neg q))$ is designated if and only if both $v_1(p)\leq v_1(q)$ and $v_2(p)\leq v_2(q)$, and therefore the formula captures the \emph{information order} (the left-right horizontal order) on the $[0,1]^{\Join}$.
\end{example}

\subsection{Two-layered logic for probabilities}\label{ssec:logicforprobabilities}

In this subsection, we introduce a~two-layered logic $\Prob^{\Luk^2}_\triangle$ designed to capture reasoning about non-standard probabilities of Subsection~\ref{ssec:probabilities}. The \emph{inner layer} consists of the language, semantics and axiomatisation of the logic $\BD$ as described in Subsection~\ref{ssec:FDE}. The \emph{outer layer} consists of the language, semantics and axiomatisation of the logic $\Luk^2$ as described in Subsection~\ref{ssec:belnapluk}. The two layers are connected with a~single modality $\Prob$ (which can be read as ‘probably’). An application of the modality to an inner formula yields an atomic modal outer formula, which is to be interpreted using a~probability function over the underlying probabilistic $\BD$ model, outputting a~value in the outer algebra $[0,1]^{\Join}_\Luk$. Other applications of the modality are forbidden in the syntax, in particular, modalities do not nest. 
\begin{definition}[Language and semantics of $\Prob^{\Luk^2}_\triangle$]\label{def:BDBLuk2Deltalanguage}
The two-layered language of the logic $\Prob^{\Luk^2}_\triangle$ is of the form $ (\LBD,\{\Prob\},\Langluk)$, with $\Prob$ being a~unary modality, and it is defined as follows: 
\begin{description}
\item[inner formulas:] formulas of $\LBD$
\item[atomic modal formulas:] Let $\phi\in\LBD$, then $\Prob\phi$ is a~modal atomic formula.
\item[outer formulas:] Let $\Prob\phi$ range over modal atomic formulas, the set of outer formulas is defined via the following grammar, with the connectives of $\Langluk$:
\[\alpha\coloneqq \Prob\phi\mid\neg\alpha\mid\alpha\rightarrow\alpha\mid{\sim}\alpha\mid\triangle\alpha.\]
Other connectives --- $\bot, \wedge, \vee, \&$, $\ominus$ and $\oplus$ --- can be defined as in $\Langluk$ of Definition~\ref{def:language_semantics_luksquare}.
\end{description}
The semantics of $\Langluk$ is given over a~$\BD$\ probabilistic model $\mathscr{M} = (W,v^+,v^-,\pr)$ (cf.~Section~\ref{ssec:probabilities}), and the outer algebra $[0,1]^{\Join}_{\Luk}$, as follows:
\begin{description}
\item[inner formulas:] for each $\phi\in\LBD$ the model provides sets $|\phi|^+ = \{w\mid w \vDash^+\phi\}$ and $|\phi|^- = \{w\mid w\vDash^-\phi \}$, as described in Definition~\ref{def:BDframesemantics}.
\item[atomic modal formulas:] for each atomic modal formula, $v^{\mathscr{M}}(\Prob\phi)$ in $[0,1]^{\Join}_{\Luk}$ is defined as
$$
v^{\mathscr{M}}(\Prob\phi) = (v^{\mathscr{M}}_1(\Prob\phi),v^{\mathscr{M}}_2(\Prob\phi)) \coloneqq(\pr(|\phi|^+),\pr(|\phi|^-)) = \left(\sum\limits_{w\vDash^+\phi}\mathtt{m}_{\pr}(w),\sum\limits_{w\vDash^-\phi}\mathtt{m}_{\pr}(w)\right).
$$
\item[outer formulas:] The map $v^{\mathscr{M}}$ is extended to the outer formulas in an expected manner, following the Definition~\ref{def:language_semantics_luksquare} of the interpretation of $\Langluk$ connectives in the algebra $[0,1]^{\Join}_{\Luk}$.
\end{description}
The outer formula $\alpha$ is said to be valid in $\mathscr{M}$, if $v^{\mathscr{M}}(\alpha)$ is designated (i.e. $v^{\mathscr{M}}(\alpha) = (1,0)$). It is said to be valid, if it is valid in all $\BD$\ probabilistic models. 
The consequence relation for $\Gamma,\alpha$ a~set of outer formulas is defined as
$$
\Gamma \vDash_{\Prob^{\Luk^2}_\triangle} \alpha \ \ \mbox{iff}\ \ \forall\mathscr{M} (v^{\mathscr{M}}[\Gamma] \subseteq (1,0) \Rightarrow v^{\mathscr{M}}(\alpha) = (1,0)).
$$
\end{definition}

\begin{definition}[Axiomatisation of $\Prob^{\Luk^2}_\triangle$]\label{def:logicBD_P_Luk}
The two-layered axiomatisation of logic $\Prob^{\Luk^2}_\triangle = (\BD,M_p,\Luk^2)$ consists of the following parts:
\begin{description}
\item[inner logic:] the axiomatisation of $\BD$ in the language $\LBD$ of Definition~\ref{def:BDBD*}
\item[modal axioms:] $M_p$, where $\f,\p\in\LBD$: 
\begin{align*}
 \{ \Prob\f &\rightarrow \Prob\psi\mid \f\vdash_{\BD}\psi \}\\
 \Prob\neg\f &\leftrightarrow \neg \Prob\f\\
 \Prob(\f\vee\p) &\leftrightarrow(\Prob\f \ominus\Prob(\f\wedge\psi))\oplus \Prob\psi
\end{align*}
\item[outer logic:] the axiomatisation of $\Luk^2$ of Definition~\ref{def:ax_luk_rightarrow}, instantiated in outer formulas. 
\end{description}
A \emph{proof} of a~formula $\alpha$ from a~set of assumptions $\Gamma$ in the logic $\Prob^{\Luk^2}_\triangle$ (in short $\Gamma\vdash_{\Prob^{\Luk^2}_\triangle}\alpha$), where $\Gamma\cup\{\alpha\}$ is a~finite set of \emph{outer formulas}, is a~finite tree which can be parsed as follows: the smallest proofs consist of either (i) a~finite tree labelled by $\BD$ formulas, proving $\f\vdash_{\BD}\psi$, followed by the instance $\Prob\f\rightarrow \Prob\psi$ of the first modal axiom, or (ii) a~node labelled by an instance of either of the remaining two modal axioms or an instance of an axiom of the outer logic in the outer language, or (iii) an assumption $\beta\in\Gamma$ in the outer language. More complex proofs are then formed inductively in the standard manner, by means of syntactically correct applications of the rules of the outer logic to smaller proofs. The root of the tree is to be labelled by $\alpha$.
\end{definition}

\begin{theorem}[FSSC of $\Prob^{\Luk^2}_\triangle$]\label{th:FSSC_probLuk2} 
For each $\Gamma,\alpha$ a~finite set of outer formulas, 
 $$
 \Gamma\vdash_{\Prob^{\Luk^2}_\triangle}\alpha \ \ \mbox{iff} \ \ \Gamma\vDash_{\Prob^{\Luk^2}_\triangle}\alpha.
 $$
\end{theorem}

\begin{proof}
The soundness and completeness proof essentially relies on the two-layered nature of the logic. We will use the following facts: (i) the inner logic $\BD$ is strongly sound and complete w.r.t.\ its double-valuation semantics, and it is locally finite, (ii) the outer logic $\Luk^2$ is finitely strongly sound and complete w.r.t.\ the outer algebra $[0,1]^{\Join}_{\Luk}$, and (iii) the modal axioms are a~sound and complete axiomatisation of probablities over $\BD$ probabilistic models.

\medskip\par\noindent
\textit{Soundness}: Given a~proof $\Gamma\vdash_{\Prob^{\Luk^2}_\triangle}\alpha$ and any $\BD$ probabilistic model $\mathscr{M} = (S, v^+, v^-, \pr)$, we reason inductively: (i) we know that any proof of $\f\vdash_{\BD}\psi$ is sound, therefore it entails that $|\f|^+\subseteq |\psi|^+$ and $|\psi|^-\subseteq |\f|^-$ in any $\BD$ model, and thus in $\mathscr{M}$. This entails that $\pr(|\f|^+)\leq\pr(|\psi|^+)$ and $\pr(|\psi|^-)\leq\pr(|\f|^-)$, and the following instance $\Prob\f\rightarrow\Prob\psi$ of the first modal axiom is justified, as it results in the value $(1,0)$ in the outer algebra. The rest of the proof follows from the soundness of $\Luk^2$, provided the rest of the modal axioms are sound. Assume an instance $\Prob\neg\f\leftrightarrow\neg\Prob\f$ of the second modal axiom: its left-hand side $\Prob\neg\f$ amounts to $(\pr(|\neg\f|^+),\pr(|\neg\f|^-))=(\pr(|\f|^-),\pr(|\f|^+))$ which is the semantics of $\neg \Prob\f$ as desired. 

Assume now an instance
$\Prob(\f\vee\p) \leftrightarrow(\Prob\f \ominus \Prob(\f\wedge\psi))\oplus \Prob\psi$ of the third modal axiom: Its left-hand side $\Prob(\f\vee\p)$ amounts to $$
(\pr(|\f|^+\cup|\psi|^+),\pr(|\f|^-\cap|\psi|^-)).
$$ 
The first coordinate we know to satisfy 
$$
\pr(|\f|^+\cup|\psi|^+) = \pr(|\f|^+) + \pr(|\psi|^+) - \pr(|\f|^+\cap |\psi|^+)
$$ 
directly by the axiom (iii) of Definition~\ref{DEF:NSprob}, while the second coordinate satisfies 
$$ 
\pr(|\f|^-\cap|\psi|^-) = \pr(|\f|^-) + \pr(|\psi|^-) - \pr(|\f|^-\cup |\psi|^-).
$$  
The first coordinate immediately matches 
$$
v_1((\Prob\f \ominus \Prob(\f\wedge\psi))\oplus \Prob\psi) = (v_1(\Prob\f) - v_1(\Prob(\f\wedge\psi))) + v_1(\Prob\psi),
$$
because $(|\f|^+\cap|\psi|^+)\leq\pr(|\f|^+)$ by monotonicity and therefore the truncated $\ominus_{\Luk}$ semantics outputs $v_1(\Prob\f)-v_1(\Prob(\f\wedge\psi))$ and does not trivialise to $0$. The second coordinate
$$
v_2((\Prob\f \ominus \Prob(\f\wedge\psi))\oplus \Prob\psi) = (v_2(\Prob(\f\wedge\psi)) \rightarrow_{\Luk} v_2(\Prob\f)) \&_{\Luk} v_2(\Prob\psi)
$$
and because now by monotonicity $\pr(|\f\wedge\psi|^-)=\pr(|\f|^-\cup |\psi|^-) \geq \pr(|\f|^-)$, the implication does not trivialise to $1$ and this computes exactly as required.

\medskip\par\noindent
\textit{Completeness}: Assume that $\Gamma\nvdash_{\Prob^{\Luk^2}_\triangle}\alpha$. We will show that there is $\BD$ probabilistic model validating $\Gamma$ while refuting $\alpha$. We will use the canonical model $\mathscr{M}_c = (\P(\LIT),v_c^+,v_c^-)$ of Definition~\ref{DEF:CanModel} to do so. 
Consider the finite set of $\BD$ subformulas of modal atoms of $\Gamma,\Phi$ to be given by
$$
\Sf_{\Gamma,\alpha} \coloneqq \{\f\mid\exists\f':\Prob\f'\mbox{ a modal atom of }\Gamma\cup\{\alpha\}\mbox{ and }\f\mbox{ a subformula of }\f'\}.
$$
Because $\BD$ is locally finite, also the set 
$$
\Phi_{\Gamma,\alpha} \coloneqq \{\psi\mid \exists \f\in\Sf_{\Gamma,\alpha} \ \mbox{ and }\ \f\dashv\vdash_{\BD}\psi\}
$$
is finite. Now define the finite set $\Sigma$ of outer formulas to contain all instances of the modal axioms for formulas from $\Phi_{\Gamma,\alpha}$.

Given that $\Gamma\nvdash_{\Prob^{\Luk^2}_\triangle}\alpha$, it is obviously also the case that $\Gamma,\Sigma\nvdash_{\Prob^{\Luk^2}_\triangle}\alpha$.
Then we know, by FSSC of $\Luk^2$, that there is a~refuting valuation 
$$
v: \{\Prob\f\mid \f\in \Phi_{\Gamma,\alpha}\} \to [0,1]^{\Join}_\Luk,
$$
i.e. $v[\Gamma\cup\Sigma]\subseteq \{(1,0)\}$ while $v(\alpha) < (1,0)$. Because $v[\Sigma]\subseteq \{(1,0)\}$, and $\Sigma$ in fact fully describes behaviour of a~probability on the $\BD$ formula algebra generated by propositional atoms from the set $\Phi_{\Gamma,\alpha}$, we define the probability of a~$\BD$ formula $\phi\in\Phi_{\Gamma,\alpha} $ to be
$$
(\pr^+(\f), \pr^-(\f)) = v(\Prob\f).
$$
It is then a~routine to check that this assignment indeed is a~probability on the respective formula algebra. Therefore, by Theorem~\ref{th:completeness_probablities}, we know that this map can be extended to act as a~probability on the powerset of the canonical $\BD$ model 
$(\P(\LIT),v^+,v^-,\pr' )$ with $\LIT$ being the finite set of literals from the set $\Phi_{\Gamma,\alpha}$, i.e. $\pr^+(\f) = \pr'(|\f|^+)$ and $\pr^-(\f) = \pr'(|\f|^-)$ for each formula from the set $\Phi_{\Gamma,\alpha}$.
This concludes the proof.
\end{proof}
\begin{example}\label{ex:probLuk2}
Assume for example that ${\sim}\Prob(\f\wedge\neg\f)$, i.e. that the positive probability of the contradiction is $0$. Observe that then positive probability of $(\f\vee\neg\f)$ is $1$:
$$
{\sim}\Prob(\f\wedge\neg\f)\vdash_{\Prob^{\Luk^2}_\triangle}{\sim}\neg{\sim}\Prob(\f\wedge\neg\f)
\vdash_{\Prob^{\Luk^2}_\triangle}{\sim}{\sim}\neg\Prob(\f\wedge\neg\f)
\vdash_{\Prob^{\Luk^2}_\triangle}\neg\Prob(\f\wedge\neg\f)\vdash_{\Prob^{\Luk^2}_\triangle}\Prob(\f\vee\neg\f).
$$
From the third axiom we then obtain
$$
\vdash_{\Prob^{\Luk^2}_\triangle} (\Prob\f \ominus \Prob(\f\wedge\neg\f))\oplus \Prob\neg\f
$$
which is equivalent (and therefore provably so) to
$$
\vdash_{\Prob^{\Luk^2}_\triangle} (\Prob\f \rightarrow \Prob(\f\wedge\neg\f))\rightarrow \Prob\neg\f.
$$
It is easy to prove that
${\sim}\Prob(\f\wedge\neg\f)\vdash_{\Prob^{\Luk^2}_\triangle}(\Prob\f \rightarrow\Prob(\f\wedge\neg\f))\leftrightarrow{\sim}\Prob\f$, and from that we obtain
$$
{\sim}\Prob(\f\wedge\neg\f)\vdash_{\Prob^{\Luk^2}_\triangle} {\sim}\Prob\f\rightarrow\neg \Prob\f \vdash_{\Prob^{\Luk^2}_\triangle} {\sim}{\neg}\Prob\f\rightarrow \Prob\f.
$$
As ${\sim}{\neg}\Prob\f\rightarrow \Prob\f$ and $\Prob\f\rightarrow {\sim}{\neg}\Prob\f$ are provable from each other (cf. Example~\ref{ex:expressivity_luksquare}), we can see that assuming that ${\sim}\Prob(\f\wedge\neg\f)$ entails that $\Prob\f$ is classical. On the other hand, assuming $\Prob\f$ is classical, we can prove that 
${\sim}\Prob(\f\wedge\neg\f)\leftrightarrow \Prob(\f\vee\neg\f)$.
\end{example}
\subsection[Logics for belief and plausibility]{Two-layered logics for belief and plausibility functions}\label{ssec:logicforbelieffunctions}
In this subsection, we introduce two logics formalising two different two-dimensional interpretations of belief and plausibility. The first option is to treat the belief modality $\belmod$ as a ‘Belnapian belief’ described in Section~\ref{sssec:bel-pl-modalities:(bel+,bel-)}. This way, \emph{the negative support} of $\belmod\phi$ is equal to $\bel(|\phi|^-)$ (i.e., the belief in $\phi$'s negative extension). The \emph{plausibility of $\phi$} can be defined in terms of its belief (cf.~Remark~\ref{rem:nonnelsonianplausibility} for the details).

The second option is to interpret belief and plausibility as two independent measures. In this approach, $\belmod\phi$'s negative support to be equal to the \emph{plausibility of $\neg\phi$} as presented in Section~\ref{sssec:bel-pl-modalities:(bel+,pl-)}. Dually, the negative support of $\Pl\phi$ \emph{is the belief in $\neg\phi$}.

These two approaches require two different logics. $\Luk^2$ is more suitable to formalise the Belnapian belief since the validity takes into account both positive and negative supports and since the belief modality should act in a symmetric manner. On the other hand, $\N\Luk$ is better suited to formalise the reasoning with independent $\bel$ and $\pl$ since it can separate positive and negative supports of a given statement (recall Example~\ref{ex:expressivity_luknelson}).
\subsubsection{Two-layered logic for belief functions}\label{sssec:logicforbelieffunctions}
\begin{definition}[Language and semantics of $\Bell^{\Luk^2}_\triangle$]\label{def:BelLuksquare}
The two-layered language of the logic $\Bell^{\Luk^2}_\triangle$ is of the form $(\LBD,\{\belmod\},\Langluk)$, with a~unary modality $\belmod$, and is defined as follows:
\begin{description}
\item[inner formulas:] formulas of $\LBD$
\item[atomic modal formulas:] Let $\phi\in\LBD$, then $\belmod\phi$ is a~modal atomic formula.
\item[outer formulas:] Let $\belmod\phi$ range over modal atomic formulas, the set of outer formulas is defined via the following grammar, with the connectives of $\Langluk$:
\[\alpha\coloneqq \belmod\phi\mid\neg\alpha\mid\alpha\rightarrow\alpha\mid{\sim}\alpha\mid\triangle\alpha.\]
Other connectives: $\bot$, $\wedge$, $\vee$, $\&$, $\ominus$, and $\oplus$, can be defined as in $\Langluk$ of Definition~\ref{def:language_semantics_luksquare}.
\end{description}
The semantics of the two-layered language is given over a~$\DS$ model $\mathscr{M} = (W,v^+,v^-,\bel)$, and the outer algebra $[0,1]^{\Join}_{\Luk}$, as follows:
\begin{description}
\item[inner formulas:] for each $\phi\in\LBD$ the model provides sets $|\phi|^+ = \{w\mid w \vDash^+\phi\}$ and $|\phi|^- = \{w\mid w\vDash^-\phi \}$, as described in Definition~\ref{def:BDframesemantics}.
\item[atomic modal formulas:] for each atomic modal formula $\belmod\phi$, $v^{\mathscr{M}}(\belmod\phi)$ in $[0,1]^{\Join}_{\Luk}$ is defined as
$$
v^{\mathscr{M}}(\belmod\phi) = (v^{\mathscr{M}}_1(\belmod\phi),v^{\mathscr{M}}_2(\belmod\phi)) \coloneqq (\bel(|\phi|^+),\bel(|\phi|^-)) = \left(\sum\limits_{X\subseteq |\phi|^+ }\!\!\!\!\!\mathtt{m}_{\bel}(X),\sum\limits_{X\subseteq |\phi|^-}\!\!\!\!\!\mathtt{m}_{\bel}(X)\right). 
$$
\item[outer formulas:] The map $v^{\mathscr{M}}$ is extended to the outer formulas in an expected manner, following the Definition~\ref{def:language_semantics_luksquare} of the interpretation of $\Langluk$ connectives in the algebra $[0,1]^{\Join}_{\Luk}$.
\end{description}
The outer formula $\alpha$ is said to be valid in $\mathscr{M}$, if $v^{\mathscr{M}}(\alpha)$ is designated (i.e. $v^{\mathscr{M}}(\alpha) = (1,0)$). It is said to be valid if it is valid in all $\DS$ models. 
The consequence relation for $\Gamma,\alpha$ a~set of outer formulas is defined as
$$
\Gamma \vDash_{\Bell^{\Luk^2}_\triangle} \alpha \ \ \mbox{iff}\ \ \forall\mathscr{M} (v^{\mathscr{M}}[\Gamma] \subseteq (1,0) \Rightarrow v^{\mathscr{M}}(\alpha) = (1,0)).
$$
\end{definition}
\begin{remark}\label{rem:nonnelsonianplausibility}
Note that we can also introduce a definable modality $\Pl$ into $\Bell^{\Luk^2}_\triangle$. Indeed, since we can interpret $1-\bel(|\neg\phi|^+)$ as the value of the \emph{plausibility} of $\phi$ (recall Lemma~\ref{lem:bel:pl:1-bel}), we have that the following equivalence is valid:
\[\Pl\phi\leftrightarrow{\sim}\belmod\neg\phi\]
\end{remark}
\begin{definition}[Belief function axioms]\label{def:bfces_Luk_rightarrow}
We define a~sequence of outer formulas $\gamma_n$
 in propositional letters of the inner language $p_1,\ldots,p_n$ inductively as follows:
 \begin{align*}
 \gamma_1 & \coloneqq\belmod p_1 \\
 \gamma_{n+1} & \coloneqq \gamma_n \oplus (\belmod p_{n+1} \ominus \gamma_n[\belmod\psi:\belmod(\psi\wedge p_{n+1})\mid\belmod\psi\mbox{ atoms of }\gamma_n])
 \end{align*}
where $\gamma_n[\belmod\psi:\belmod(\psi\wedge p_{n+1})\mid\belmod\psi\mbox{ modal atoms of }\gamma_n])$ is the result of replacing each modal atom $\belmod\psi$ in $\gamma_n$ with the modal atom $\belmod(\psi\wedge p_{n+1})$ (semantically, it is a~relativisation of the corresponding belief function to the sets $|p_{n+1}|^{+-}$).

The $n$-th belief function axiom (i.e. the $n$-monotonicity) is expressed by substitution instances (substituting inner formulas for the atomic letters $p_1,\ldots,p_n$) of 
$$
\alpha_n \coloneqq \gamma_n \rightarrow\belmod\left(\bigvee_{i=1}^n p_n\right).
$$
\end{definition}
\begin{lemma}[Soundness of the axioms]\label{lem:bfces_Luk_rightarrow} The modal axioms $\alpha_n$ are sound w.r.t.\ $\DS$ models.
\end{lemma}
\begin{proof}
We need to prove that the axioms introduced in Definition~\ref{def:bfces_Luk_rightarrow} are sound, i.e. that the formulas $\gamma_n$ indeed express semantically the appropriate sums from Definition~\ref{def:generalbelieffunction}. Because of the two-dimensional nature of the semantics of the modality 
$$
v^{\mathscr{M}}(\belmod\phi) = (v^{\mathscr{M}}_1(\belmod\phi),v^{\mathscr{M}}_2(\belmod\phi)) \coloneqq (\bel(|\phi|^+),\bel(|\phi|^-)),
$$
we need to consider both $\bel^+$ as a~general belief function on the Lindenbaum algebra $\LBD$, and $\bel^-$ as a~general belief function on its dual $\LBD^{op}$. 

Recall that $\bel^+$ on $\LBD$ is \emph{$n$-monotone} if, for every $a_1,\ldots,a_n$, it holds that
\begin{equation*}
\bel^+\left(\bigvee\limits^n_{i=1}a_i\right)\geq\sum\limits_{\scriptsize{\begin{matrix}J\subseteq\{1,\ldots,n\}\\J\neq\varnothing\end{matrix}}}(-1)^{|J|+1}\cdot \bel^+\left(\bigwedge\limits_{j\in J}a_j\right).
\end{equation*}
What the sum expresses is that we always add ($+$) beliefs of conjunctions of an odd size, while we subtract ($-$) beliefs of conjunctions of an even size. To write down exactly the same sum as a~$\Luk^2$ formula is however more problematic than in the probability case because the sums and subtractions of $\Luk^2$ are truncated to stay within the margins of $[0,1]$ and we have to be careful to reflect this in the way we arrange the elements of the sum so that we never lose information. We first need to see that the axioms can indeed be re-defined recursively as follows ($\mathtt{t}_n$ now denote terms in the language of linear inequalities):
\begin{align*}
 \mathtt{t}_1 & \coloneqq \bel^+ (a_1) \\
 \mathtt{t}_{n+1} & \coloneqq \mathtt{t}_n + (\bel^+ (a_{n+1}) - \mathtt{t}_n[\bel^+ (\psi): \bel^+(\psi\wedge a_{n+1})\mid \bel^+(\psi)\ \mbox{atoms of}\ \mathtt{t}_n])
\end{align*}
The $n$-th axiom is now written as (the trivial case of $n=1$ is included for convenience):
$$
\mathtt{t}_n \leq \bel^+\left(\bigvee_{i=1}^n a_n\right).
$$
The intuition behind this definition comes from the powerset algebras and it is geometrical. In each inductive step, we look at what happens if we add one more set (element $a_{n+1}$) to the union: we keep the sum we had so far ($\mathtt{t}_n$), add $\bel(a_{n+1})$, and subtract the sum we had \emph{relativised} to $\bel(a_{n+1})$ (replacing beliefs of the sets in the original sum by beliefs of their intersection with the new set $a_{n+1}$). That the two definitions are equivalent can be proven by a~routine (but tedious) induction. The whole point of this definition is that in the above,
$$
\bel^+ (a_{n+1}) \geq \mathtt{t}_n[\bel^+ (\psi): \bel^+(\psi\wedge a_{n+1})\mid \bel^+(\psi)\ \mbox{atoms of}\ \mathtt{t}_n],
$$ 
and therefore 
$$ 
0 \leq \bel^+ (a_{n+1}) - \mathtt{t}_n[\bel^+ (\psi): \bel^+(\psi\wedge a_{n+1})\mid \bel^+(\psi)\ \mbox{atoms of}\ \mathtt{t}_n] \leq 1
.
$$
Dually, recall that as $\bel^-$ is a~general belief function on $\LBD^{op}$, then for every $a_1,\ldots,a_n$ it holds that
\begin{equation*}
\bel^-\left(\bigwedge\limits^n_{i=1}a_i\right)\leq\sum\limits_{\scriptsize{\begin{matrix}J\subseteq\{1,\ldots,n\}\\J\neq\varnothing\end{matrix}}}(-1)^{|J|+1}\cdot \bel^-\left(\bigvee\limits_{j\in J}a_j\right).
\end{equation*}
These axioms can be re-defined recursively as follows (this is tailored to be expressed in the \L{}ukasiewicz language):
\begin{align*}
 \mathtt{s}_1 & \coloneqq \bel^- (a_1) \\
 \mathtt{s}_{n+1} & \coloneqq (\mathtt{s}_n + 1 -\mathtt{s}_n[\bel^- (\psi): \bel^-(\psi\vee a_{n+1})\mid \bel^-(\psi)\ \mbox{atoms of}\ \mathtt{s}_n])
 -
 (1 - \bel^- (a_{n+1})).
\end{align*}
The $n$-th axiom is now written as:
$$
\mathtt{s}_n \geq \bel^-\left(\bigwedge_{i=1}^n a_n\right).
$$
The main reason we did it this particular way is that now the following sum stays within margins of $[0,1]$ (and therefore directly corresponds to a~$\Luk$ formula):
$$
0 \leq \mathtt{s}_n + 1 -\mathtt{s}_n[\bel^- (\psi): \bel^-(\psi\vee a_{n+1})\mid \bel^-(\psi)\ \mbox{atoms of}\ \mathtt{s}_n] \leq 1.
$$
It remains to show that for each $n$ and each $\DS$ model $(S,v^+,v^-,\bel)$ we will have
$$
v_1(\gamma_n) \leq \bel\left|\bigvee_{i=1}^n p_n\right|^+\mbox{ and}\quad v_2(\gamma_n) \geq \bel \left|\bigvee_{i=1}^n p_n\right|^-.
$$
The case $n=1$ is trivial, we, therefore, go directly for the recursion step. For the first coordinate,
\begin{align*}
  v_1(\gamma_{n+1}) &= v_1(\gamma_n) \oplus_{\Luk} (\bel(|p_{n+1}|^+) \ominus_{\Luk} v_1(\gamma_n[\belmod\psi: \belmod(\psi\wedge p_{n+1})]))\\
  &= \mathtt{t}_{n}(|p_1|^+,\ldots,|p_n|^+) + (\bel(|p_{n+1}|^+) - \mathtt{t}_{n}[\bel (|\psi|^+): \bel(|\psi\wedge p_{n+1}|^+)])\\
  &= \mathtt{t}_{n} + (\bel(|p_{n+1}|^+) - \mathtt{t}_{n}[\bel (|\psi|^+): \bel(|\psi|^+\cap |p_{n+1}|^+)])\\
  &\leq \bel \bigcup_{i=1}^n |p_n|^+ = \bel\left|\bigvee_{i=1}^n p_n\right|^+.
\end{align*}
For the second coordinate,
\begin{align*}
  v_2(\gamma_{n+1}) &=
              v_2(\gamma_n)~\&_{\Luk}~(v_2(\gamma_n[\belmod\psi: \belmod(\psi\wedge p_{n+1})]) \rightarrow_{\Luk} \bel(|p_{n+1}|^-)) \\
              &= (v_2(\gamma_n)\oplus_{\Luk} {\sim_{\Luk}}v_2(\gamma_n[\belmod\psi: \belmod(\psi\wedge p_{n+1})])) \ominus_{\Luk} {\sim_{\Luk}}\bel(|p_{n+1}|^-) \\
   &= (\mathtt{s}_n(|p_1|^-,\ldots,|p_n|^-) + 1 -\mathtt{s}_n[\bel(|\psi|^-): \bel(|\psi\vee p_{n+1}|^-)]) - (1 - \bel(|p_{n+1}|^-))\\
  &=  (\mathtt{s}_n + 1 -\mathtt{s}_n[\bel(|\psi|^-): \bel(|\psi|^-\cap |p_{n+1}|^-)]) - (1 - \bel(|p_{n+1}|^-))\\
  &\geq \bel\bigcap_{i=1}^n\left|p_n\right|^- = \bel\left|\bigvee_{i=1}^n p_n\right|^-
\end{align*}
\end{proof}
\begin{definition}[Axiomatisation of $\Bell^{\Luk^2}_\triangle$]\label{def:logicBD_B_Luk}
The two-layered axiomatisation of logic $\Bell^{\Luk^2}_\triangle = (\BD,M_b, \Luk^2)$ consists of 
\begin{description}
\item[inner logic:] the axiomatisation of $\BD$ of Definition~\ref{def:BDBD*}
\item[modal axioms:] $M_b$: 
\begin{align*}
 \{ \belmod\f &\rightarrow \belmod\psi\mid \f\vdash_{\BD}\psi \}\\
 \belmod\neg\f &\leftrightarrow \neg \belmod\f\\
 \{ \alpha_n \ &|\ n\in\mathbb{N}\}\mbox{ (substitution instances of)}
\end{align*}
\item[outer logic:] the axiomatisation of $\Luk^2$ of Definition~\ref{def:ax_luk_rightarrow}
\end{description}
A proof of a~formula $\alpha$ from a~set of assumptions $\Gamma$ in the logic $\Bell^{\Luk^2}_\triangle$ (in short $\Gamma\vdash_{\Bell^{\Luk^2}_\triangle}\alpha$), where $\Gamma,\alpha$ is a~finite set of outer formulas, is a~finite tree which can be parsed as follows: the smallest proofs consist of either (i) a~finite tree labeled by $\BD$ formulas, proving $\f\vdash_{\BD}\psi$, followed by the instance $\belmod\f\rightarrow \belmod\psi$ of the first modal axiom, or (ii) a~node labelled by an instance of the remaining two modal axioms or an instance of an axiom of the outer logic in the outer language, or (iii) an assumption $\beta\in\Gamma$ in the outer language. More complex proofs are then formed inductively in the standard manner, by means of syntactically correct applications of the rules of the outer logic to smaller proofs. The root of the tree is to be labelled by $\alpha$. 
\end{definition}
\begin{theorem}[FSSC of $\Bell^{\Luk^2}_\triangle$]\label{th:FSSC_belLuk2} For each $\Gamma,\alpha$ a~finite set of outer formulas, 
 $$
 \Gamma\vdash_{\Bell^{\Luk^2}_\triangle}\alpha \ \ \mbox{iff} \ \ \Gamma\vDash_{\Bell^{\Luk^2}_\triangle}\alpha.
 $$
\end{theorem}
\begin{proof}
The soundness and completeness proof essentially relies on the two-layered nature of the logic. We will use the following facts: (i) the inner logic $\BD$ is strongly sound and complete w.r.t.\ its double-valuation semantics, and it is locally finite, (ii) the outer logic $\Luk^2$ is finitely strongly sound and complete w.r.t.\ the outer algebra $[0,1]^{\Join}_{\Luk}$, and (iii) the modal axioms are a~sound and complete axiomatisation of belief functions over $\DS$  models.

\medskip\par\noindent
\textit{Soundness}: Given a~proof $\Gamma\vdash_{\Bell^{\Luk^2}_\triangle}\alpha$ and any $\DS$ model $\mathscr{M} = (S, v^+, v^-, \bel)$, we reason inductively: (i) we know that any proof of $\f\vdash_{\BD}\psi$ is sound, therefore it entails that $|\f|^+\subseteq |\psi|^+$ and $|\psi|^-\subseteq |\f|^-$ in any $\BD$ model, and thus in $\mathscr{M}$. This entails that $\bel^+\f \leq \bel^+\psi$ and $\bel^-\psi \leq \bel^-\f$, and  the following instance $\belmod\f\rightarrow \belmod\psi$ of the first modal axiom is justified, as it results in the value $(1,0)$ in the outer algebra. The rest of the proof follows from the soundness of $\Luk^2$, providing the rest of the modal axioms are sound. Assume an instance $\belmod\neg\f \leftrightarrow \neg \belmod\f$ of the second modal axiom: Its left-hand side $\belmod\neg\f$ amounts to $(\bel|\neg\f|^+,\bel|\neg\f|^-) = (\bel|\f|^-,\bel|\f|^+)$ which is the semantics of $\neg\belmod\f$ as desired. The soundness of the axioms $\alpha_n$ has been demonstrated in Lemma~\ref{lem:bfces_Luk_rightarrow}.

\medskip\par\noindent
\textit{Completeness}: Assume that $\Gamma\nvdash_{\Bell^{\Luk^2}_\triangle}\alpha$. We will show that there is $\DS$ model validating $\Gamma$ while refuting $\alpha$. We will again use the canonical model $\mathscr{M}_c = (\P(\LIT),v_c^+,v_c^-)$ of Definition~\ref{DEF:CanModel} to do so. 
Consider the finite set of $\BD$ subformulas of modal atoms of $\Gamma,\Phi$ to be given by
$$
\Sf_{\Gamma,\alpha} \coloneqq \{\f\mid \exists \f'\ \belmod\f'\mbox{ a modal atom of }\Gamma\cup\{\alpha\}\mbox{ and }\f\mbox{ a subformula of }\f' \}.
$$
Because $\BD$ is locally finite, also the set 
$$
\Phi_{\Gamma,\alpha} \coloneqq \{\psi\mid \exists \f\in\Sf_{\Gamma,\alpha} \ \mbox{ and }\ \f\dashv\vdash_{\BD} \psi\}
$$
is finite. Now define the finite set $\Sigma$ of outer formulas to contain all instances of the modal axioms for formulas from $\Phi_{\Gamma,\alpha}$.

Given that $\Gamma\nvdash_{\Bell^{\Luk^2}_\triangle}\alpha$, it is obviously also the case that $\Gamma,\Sigma\nvdash_{\Bell^{\Luk^2}_\triangle}\alpha$.
Then we know, by FSSC of $\Luk^2$, that there is a~refuting valuation 
$$
v: \{\belmod\f\mid \f\in \Phi_{\Gamma,\alpha}\} \to [0,1]^{\Join}_\Luk,
$$
i.e. $v[\Gamma\cup\Sigma]\subseteq \{(1,0)\}$ while $v(\alpha) < (1,0)$. Because $v[\Sigma]\subseteq \{(1,0)\}$, and $\Sigma$ in fact fully describes behaviour of a~belief function on the $\BD$ formula algebra generated by propositional atoms from the set $\Phi_{\Gamma,\alpha}$, we define the belief function of a~$\BD$ formula $\phi\in\Phi_{\Gamma,\alpha} $ to be
$$
(\bel^+(\f), \bel^-(\f)) = v(\belmod\f).
$$
It is then a~routine to check that this assignment indeed is a~belief function on the respective formula algebra (or its dual). Therefore, by Theorem~\ref{th:complBelAxioms}, we know that this map can be extended to a~belief function $\bel'$ on the powerset of the canonical $\BD$ model 
$(\P(\LIT),v^+,v^-,\bel' )$ with $\LIT$ being the finite set of literals from the set $\Phi_{\Gamma,\alpha}$, i.e. $\bel^+(\f) = \bel'(|\f|^+)$ and $\bel^-(\f) = \bel'(|\f|^-)$ for each formula from the set $\Phi_{\Gamma,\alpha}$.
This concludes the proof.
\end{proof}
\begin{example}\label{ex:bel_luksquare}
Assume that $\belmod\f$ is classical, i.e. that $\belmod\f\leftrightarrow {\sim}\neg \belmod\f$. Then, starting from the $\alpha_2$ axiom, and $\belmod(\f\vee\neg\f)\leftrightarrow\neg \belmod(\f\wedge\neg\f)$, we obtain, consequently using the conflation rule and distribution, the classicality, contraposition, $\vee$-definition, and finally $\belmod(\f\vee\neg\f)\vee \belmod\neg\f \leftrightarrow\belmod(\f\vee\neg\f)$:
\begin{align*}
&\vdash_{\Bell^{\Luk^2}_\triangle} ((\belmod\f\rightarrow \belmod(\f\wedge\neg\f))\rightarrow \belmod\neg\f) \rightarrow \neg \belmod(\f\wedge\neg\f)\\
&\vdash_{\Bell^{\Luk^2}_\triangle} {\sim}\neg((\belmod\f\rightarrow \belmod(\f\wedge\neg\f))\rightarrow \belmod\neg\f) \rightarrow \neg \belmod(\f\wedge\neg\f)\\
&\vdash_{\Bell^{\Luk^2}_\triangle} ((\belmod\f\rightarrow {\sim}\neg \belmod(\f\wedge\neg\f))\rightarrow {\sim}\belmod\f) \rightarrow {\sim} \belmod(\f\wedge\neg\f)\\
&\vdash_{\Bell^{\Luk^2}_\triangle} (( \neg \belmod(\f\wedge\neg\f)\rightarrow {\sim}\belmod\f )\rightarrow {\sim}\belmod\f) \rightarrow {\sim} \belmod(\f\wedge\neg\f)\\
&\vdash_{\Bell^{\Luk^2}_\triangle} (\neg \belmod(\f\wedge\neg\f)\vee {\sim}\belmod\f) \rightarrow {\sim} \belmod(\f\wedge\neg\f)\\
&\vdash_{\Bell^{\Luk^2}_\triangle} (\belmod(\f\vee\neg\f)\vee \neg \belmod\f) \rightarrow {\sim} \belmod(\f\wedge\neg\f)\\
&\vdash_{\Bell^{\Luk^2}_\triangle}\belmod(\f\vee\neg\f)\rightarrow{\sim}\belmod(\f\wedge\neg\f) 
\end{align*}
The converse implication is provable as well. (Also note that the same proof with $\belmod$ replaced by $\Prob$ would work for probabilities, thus extending the Example~\ref{ex:probLuk2}).
\end{example}
\subsubsection{Two-layered logic for belief and plausibility functions}\label{sssec:logicforbelieffunctionsandplausibilities}

\begin{definition}[Language and semantics of $\Bell^{\N\Luk}$]\label{def:BelNLuk}
The two-layered language of the logic $\Bell^{\N\Luk}$ is of the form $(\LBD,\{\belmod,\Pl\},\LanglukN)$, with unary modalities $\belmod$ and $\Pl$, and it is defined as follows: 
\begin{description}
\item[inner formulas:] formulas of $\LBD$
\item[atomic modal formulas:] Let $\phi\in\LBD$, then $\belmod\phi$ and $\Pl\phi$ are modal atomic formulas.
\item[outer formulas:] Let $\belmod\phi$ and $\Pl\phi$ range over modal atomic formulas, the set of outer formulas is defined via the following grammar, with the connectives of $\LanglukN$:
\[\alpha\coloneqq \belmod\phi\mid\Pl\phi\mid\neg\alpha\mid\alpha\wedge\alpha\mid\alpha\rightarrowtriangle\alpha\mid{\sim}\alpha.\]
Other connectives: $\bot$, $\vee$, $\&$, $\ominus$ and $\oplus$, can be defined as in $\LanglukN$ of Definition~\ref{def:language_semantics_luksnelsonquare}.
\end{description}
The semantics of the two-layered language is given over a~$\DS_\pl$ model $\mathscr{M} = (W,v^+,v^-,\bel,\pl)$, and the outer algebra $[0,1]^{\Join}_{\N\Luk}$, as follows:
\begin{description}
\item[inner formulas:] for each $\phi\in\LBD$ the model provides sets $|\phi|^+ = \{w\mid w \vDash^+\phi\}$ and $|\phi|^- = \{w\mid w\vDash^-\phi \}$, as described in Definition~\ref{def:BDframesemantics}.
\item[atomic modal formulas:] for atomic modal formulas $\belmod\phi$ and $\Pl\phi$, $v^{\mathscr{M}}(\belmod\phi)$ and $v^\mathscr{M}(\Pl\phi)$ in $[0,1]^{\Join}_{\N\Luk}$ are defined as follows:
$$
v^{\mathscr{M}}(\belmod\phi) = (v^{\mathscr{M}}_1(\belmod\phi),v^{\mathscr{M}}_2(\belmod\phi)) \coloneqq (\bel(|\phi|^+),\pl(|\phi|^-));
$$
$$
v^{\mathscr{M}}(\Pl\phi) = (v^{\mathscr{M}}_1(\Pl\phi),v^{\mathscr{M}}_2(\Pl\phi)) \coloneqq (\pl(|\phi|^+),\bel(|\phi|^-)). 
$$
\item[outer formulas:] The map $v^{\mathscr{M}}$ is extended to the outer formulas in an expected manner, following the Definition~\ref{def:language_semantics_luksquare} of the interpretation of $\LanglukN$ connectives in the algebra $[0,1]^{\Join}_{\N\Luk}$.
\end{description}
The outer formula $\alpha$ is said to be valid in $\mathscr{M}$, if $v^{\mathscr{M}}(\alpha)$ is designated (i.e. $v_1^{\mathscr{M}}(\alpha) = 1$). It is said to be valid if it is valid in all $\DS_\pl$ models. 
The consequence relation for $\Gamma,\alpha$ a~set of outer formulas is defined as
$$
\Gamma \vDash_{\Bell^{\N\Luk}} \alpha \ \ \mbox{iff}\ \ \forall\mathscr{M} (v_1^{\mathscr{M}}[\Gamma] \subseteq \{1\} \Rightarrow v_1^{\mathscr{M}}(\alpha) = 1).
$$
\end{definition}
\begin{definition}[Belief and plausibility functions axioms]\label{def:bfces_Luk_rightarrowtriangle}
We define sequences of outer formulas $\gamma_n$, and $\sigma_n$ in propositional letters of the inner language $p_1,\ldots,p_n$ inductively as follows:
 \begin{align*}
 \gamma_1 & \coloneqq\belmod p_1 \\
 \gamma_{n+1} & \coloneqq \gamma_n \oplus (\belmod p_{n+1} \ominus \gamma_n[\belmod\psi:\belmod(\psi\wedge p_{n+1})\mid\belmod\psi\mbox{ atoms of }\gamma_n])\\
 \sigma_1 &\coloneqq\Pl p_1 \\
 \sigma_{n+1} & \coloneqq (\sigma_n \oplus {\sim}\sigma_n[\Pl\psi:\Pl(\psi\vee p_{n+1})\mid\Pl\psi\mbox{ atoms of }\sigma_n])\ominus {\sim}\Pl p_{n+1} 
 \end{align*} 
where $\gamma_n[\belmod\psi:\belmod(\psi\wedge p_{n+1})\mid\belmod\psi\mbox{ atoms of }\gamma_n])$ is the result of replacing each atom $\belmod\psi$ in $\gamma_n$ with the atom $\belmod(\psi\wedge p_{n+1})$, and $\sigma_n[\Pl\psi:\Pl(\psi\vee p_{n+1})\mid\Pl\psi\mbox{ atoms of }\sigma_n]$ is the result of replacing each atom $\Pl\psi$ in $\sigma_n$ with the atom $\Pl(\psi\vee p_{n+1})$.

The $n$-th belief function axiom (i.e. the $n$-monotonicity) is expressed by substitution instances (substituting inner formulas for the atomic letters $p_1,\ldots,p_n$) of 
$$
\alpha_n \coloneqq \gamma_n \rightarrowtriangle \belmod\left(\bigvee_{i=1}^n p_n\right).
$$
The $n$-th plausibility function axiom is expressed by substitution instances (substituting inner formulas for the atomic letters $p_1,\ldots,p_n$) of 
$$
\beta_n \coloneqq\Pl\left(\bigwedge_{i=1}^n p_n\right)\rightarrowtriangle \sigma_n .
$$
\end{definition}
\begin{lemma}[Soundness of the axioms]\label{lem:bfces_Luk_rightarrowtriangle}
The axioms $\alpha_n$ and $\beta_n$ are sound w.r.t.\ $\DS_\pl$ models.
\end{lemma}
\begin{proof}
We need to prove that the axioms introduced in Definition~\ref{def:bfces_Luk_rightarrowtriangle} are sound, i.e. that the formulas $\gamma_n,\sigma_n$ indeed express semantically the appropriate sums from Definition~\ref{def:generalbelieffunction} and Definition~\ref{def:generalplausibilityfunction}. First note that the validity of the axioms $\alpha_n,\beta_n$ means that, for each $\DS_\pl$ model $\mathscr{M}$, $v^{\mathscr{M}}_1(\alpha_n) =1$ and $v^{\mathscr{M}}_1(\beta_n) =1$. Thus we only care about the first coordinates. Because the weak implication still expresses the order of the first coordinates ($(a_1,a_2)\rightarrowtriangle (b_1,b_2)$ is designated if and only if $a_1\leq b_1$), we only need to check that the first coordinates of formulas $\gamma_n,\sigma_n$ express semantically the appropriate sums from Definition~\ref{def:generalbelieffunction} and Definition~\ref{def:generalplausibilityfunction}. But this can be done exactly as in proof of Lemma~\ref{lem:bfces_Luk_rightarrow}, using the terms $\mathtt{t}_n$ for $\gamma_n$, and straightforward analogues of terms $\mathtt{s}_n$ for $\sigma_n$. This follows from the fact, that the $n$-th axiom of the general belief function $\bel^-$ on $\LBD^{op}$ and the $n$-th axiom of the general plausibility function $\pl^+$ on $\LBD$ are of the same shape (not to confuse the two notions however, notice that while $\bel^-$ is antitone on $\LBD$, $\pl^+$ is monotone on $\LBD$).
\end{proof}
\begin{definition}[Axiomatisation of $\Bell^{\N\Luk} $]\label{def:logicBD_MNB_Luk}
The two-layered axiomatisation of logic $\Bell^{\N\Luk} = (\BD,M^\mathbb{N}_b,\Luk^2_{(1,1)^\uparrow}(\rightarrowtriangle))$
\begin{description}
\item[inner logic:] the axiomatisation of $\BD$ of Definition~\ref{def:BDBD*}
\item[modal axioms:] $M_b^\mathbb{N}$: 
\begin{align*}
 \{ \belmod\f &\Rightarrow \belmod\psi,\Pl\f \Rightarrow\Pl\psi\mid \f\vdash_{\BD}\psi \}\\
 \belmod\neg\f &\Leftrightarrow\neg\Pl\f\\
 \{ \alpha_n, \beta_n&\mid n\in\mathbb{N}\}\mbox{ (substitution instances of)}
\end{align*}
\item[outer logic:] the axiomatisation of $\N\Luk$ of Definition~\ref{def:ax_luk_rightarrowtriangle}
\end{description}
A proof of a~formula $\alpha$ from a~set of assumptions $\Gamma$ in the logic $\Bell^{\N\Luk}$ (in short $\Gamma\vdash_{\Bell^{\N\Luk}}\alpha$), where $\Gamma,\alpha$ is a~finite set of outer formulas, is a~finite tree which can be parsed as follows: the smallest proofs consist of either (i) a~finite tree labeled by $\BD$ formulas, proving $\f\vdash_{\BD}\psi$, followed by the instance $\belmod\f\Rightarrow \belmod\psi$ or by the instance $\Pl\f\Rightarrow\Pl\psi$ of the first modal axiom, or (ii) a~node labelled by an instance of the remaining two modal axioms or an instance of an axiom of the outer logic in the outer language, or (iii) an assumption $\beta\in\Gamma$ in the outer language. More complex proofs are then formed inductively in the standard manner, by means of syntactically correct applications of the rules of the outer logic to smaller proofs. The root of the tree is to be labelled by $\alpha$. 
\end{definition}
\begin{theorem}[FSSC of $\Bell^{\N\Luk}$] For each $\Gamma,\alpha$ a~finite set of outer formulas, 
 $$
 \Gamma\vdash_{\Bell^{\N\Luk}}\alpha \ \ \mbox{iff} \ \ \Gamma\vDash_{\Bell^{\N\Luk}}\alpha.
 $$
\end{theorem}
\begin{proof}
The soundness and completeness proof is almost the same as before. We will now use the following facts: (i) the inner logic $\BD$ is strongly sound and complete w.r.t.\ its double-valuation semantics, and it is locally finite, (ii) the outer logic $\N\Luk$ is finitely strongly sound and complete w.r.t.\ the outer algebra $[0,1]^{\Join}_{\N\Luk}$, and (iii) the modal axioms are a~sound and complete axiomatisation of belief and plausibility functions over $\DS_\pl$  models.

\medskip\par\noindent
\textit{Soundness}: Given a~proof $\Gamma\vdash_{\Bel^{N\Luk}}\alpha$ and any $\DS_\pl$ model $\mathscr{M} = (S, v^+, v^-, \bel,\pl)$, we reason inductively: (i) we know that any proof of $\f\vdash_{\BD}\psi$ is sound, therefore it entails that $|\f|^+\subseteq |\psi|^+$ and $|\psi|^-\subseteq |\f|^-$ in any $\BD$ model, and thus in $\mathscr{M}$. This entails that $\bel(|\f|^+)\leq\bel(|\psi|^+)$ and $\bel(|\psi|^+)\leq\bel(|\f|^+)$, and $\pl(|\f|^+)\leq\pl(|\psi|^+)$ and $\pl(|\psi|^-)\leq\pl(|\f|^-)$, and either of the following instances $\belmod\f\Rightarrow \belmod\psi$ and $\Pl\f\Rightarrow\Pl\psi$ of the first modal axiom is justified, as it results in a~value $(1,x)$ in the outer algebra. The rest of the proof follows from the soundness of $\N\Luk$, providing the rest of the modal axioms are sound. Assume an instance $\belmod\neg\f \Leftrightarrow \neg\Pl\f$ of the second modal axiom: its left-hand side $\belmod\neg\f$ amounts to $(\bel(|\neg\f|^+),\pl(|\neg\f|^-)) = (\bel(|\f|^-),\pl(|\f|^+))$ which matches exactly the semantics of $\neg\Pl\f$ as desired. The soundness of the axioms $\alpha_n$, and $\beta_n$ has been demonstrated in Lemma~\ref{lem:bfces_Luk_rightarrowtriangle}.

\medskip\par\noindent
\textit{Completeness}: Assume that $\Gamma\nvdash_{\Bel^{\N\Luk}}\alpha$. We will show that there is $\DS_\pl$ model validating $\Gamma$ while refuting $\alpha$. We will again use the canonical model $\mathscr{M}_c = (\P(\LIT),v_c^+,v_c^-)$ of Definition~\ref{DEF:CanModel} to do so. Consider the finite set of $\BD$ subformulas of modal atoms of $\Gamma,\Phi$ to be given by
$$
\Sf_{\Gamma,\alpha} \coloneqq \{\f\mid\exists\mathsf{M}\!\in\!\{\belmod,\Pl\}~\exists \f':\mathsf{M}\f'\mbox{ a modal atom of }\Gamma\cup\{\alpha\}\mbox{ and }\f \mbox{ a subformula of }\f' \}.
$$
Because $\BD$ is locally finite, also the set 
$$
\Phi_{\Gamma,\alpha} \coloneqq \{\psi\mid \exists \f\in\Sf_{\Gamma,\alpha} \ \mbox{ and }\ \f\dashv\vdash_{\BD} \psi)\}
$$
is finite. Now define the finite set $\Sigma$ of outer formulas to contain all instances of the modal axioms for formulas from $\Phi_{\Gamma,\alpha}$.

Given that $\Gamma\nvdash_{\Bel^{\N\Luk}}\alpha$, it is obviously also the case that $\Gamma,\Sigma\nvdash_{\Bel^{\N\Luk}}\alpha$.
Then we know, by FSSC of $\N\Luk$, that there is a~refuting valuation 
$$
v: \{\belmod\f\mid \f\in \Phi_{\Gamma,\alpha}\} \to [0,1]^{\Join}_{\N\Luk},
$$
i.e. $v_1[\Gamma\cup\Sigma]\subseteq \{1\}$ while $v_1(\alpha)<1$. Because $v_1[\Sigma]\subseteq \{1\}$, and $\Sigma$ in fact fully describes behaviour of a~belief or plausibility function on the $\BD$ formula algebra generated by propositional atoms from the set $\Phi_{\Gamma,\alpha}$, we can define the belief functions $\bel^+,\bel^-$ and the plausibility functions $\pl^+,\pl^-$ of a~$\BD$ formula $\f\in\Phi_{\Gamma,\alpha}$ to be
\begin{align*}
\bel^+(\f) = v_1(\belmod\f)&&\bel^-(\f) = v_2(\Pl\f)&&\pl^+(\f)=v_1(\Pl\f)&&\pl^-(\f) = v_2(\belmod\f).
\end{align*}

It is then a~routine to check that this assignment indeed results in belief and plausibility functions on the respective formula algebra (or its dual). 
 
Therefore, by Theorem~\ref{th:complBelAxioms} and Theorem~\ref{th:complPlAxioms}, we know that $\bel$ can be extended to a~belief function $\bel'$, and $\pl$ to $\pl'$, on the powerset of the canonical $\BD$ model 
$(\P(\LIT),v^+,v^-,\bel',\pl')$ with $\LIT$ being the finite set of literals from the set $\Phi_{\Gamma,\alpha}$, i.e. $\bel^+(\f) = \bel'(|\f|^+)$, $\bel^-(\f) = \bel'(|\f|^-)$, $\pl^+(\f) = \pl'(|\f|^+)$ and $\pl^-(\f) = \pl'(|\f|^-)$ for each formula from the set $\Phi_{\Gamma,\alpha}$.
This concludes the proof.
\end{proof}
\begin{remark}[Belief below plausibility]\label{rem:addingaxiomB_Pl}
When dealing with $\DS_\pl$ models in the context of Subsection~\ref{sssec:bel-pl-modalities:(bel+,bel-):bel<pl} where $\bel(X) \leq \pl(X)$ for any $X\subseteq \P(S)$, we simply extend the logic $\Bell^{\N\Luk}$ with additional axioms 
$$
\belmod\f \rightarrowtriangle\Pl \f.
$$
Note that $\belmod\f \rightarrowtriangle\Pl\f$ proves $\neg\Pl\f \rightarrowtriangle \neg\belmod\f$ because of the $\belmod\neg$-axiom. The additional axioms are sound, whenever we ensure that $\bel(|\f|^+)\leq \pl(|\f|^+)$ in the underlying $\DS_\pl$ models.
\end{remark}

\bigskip\par\noindent
We conclude this whole subsection with some general remarks concerning the two-layered \L{}ukasiewicz-based logics presented above.
\begin{remark}[Adding $\bot$ and $\top$ to the inner logic]\label{rem:addingconstantsBD}
Adding the constants $\bot$ and $\top$ to $\BD$ results in the logic $\BD^*$ (see Subsection~\ref{ssec:FDE} for its axiomatisation). Using $\BD^*$ as the inner logic, however, affects how the modal part of the logic is to be completely axiomatised because we now have formulas whose measure is always $0$ (or $1$ respectively). This is simply achieved by extending the modal part with the modal axiom $\Prob\top$ (respectively, $\belmod\top$ $\Pl\top$). One can think of it as a~way to avoid a~necessitation modal rule, which would only make sense to add if we allowed inner formulas as assumptions in the consequence relation of the two-layered logics (see Remark~\ref{rem:mixedconsequence}).
\end{remark}
\begin{remark}\label{rem:normalityetc}
As one can see from Remark~\ref{rem:addingconstantsBD}, $\Prob$, $\belmod$, and $\Pl$ are not \emph{normal} modalities (unless the inner logic includes constants) in the sense that no formula of the form $\mathsf{M}\phi$ with $\phi\in\LBD$ and $\mathsf{M}\in\{\Prob,\belmod,\Pl\}$ is valid. Note, however, that all three modalities are \emph{regular}: if $\phi\vdash\chi$ is $\BD$- or $\BD^*$-valid, then $\mathsf{M}\phi\rightsquigarrow\mathsf{M}\chi$ is valid as well for $\rightsquigarrow\in\{\rightarrow,\rightarrowtriangle,\Rightarrow\}$.
\end{remark}
\begin{remark}[Allowing inner assumptions]\label{rem:mixedconsequence}
If we wish to allow inner formulas as assumptions in the consequence relation of the two-layered logics (in line with the abstract approach of \cite{CintulaNoguera2014,BaldiCintulaNoguera2020}), we would consider a~mixed-consequence relation of the form
$$
\Phi,\Gamma\vdash\alpha,
$$
where $\Phi$ is a~finite set of inner formulas, and $\Gamma,\alpha$ is a~finite set of outer formulas. The semantics is now given as follows: for each model $\mathscr{M} = (S,v^+,v^-,f)$ validating $\Phi$ (in the sense of $\BD$ semantics), if the $f$-derived valuation $v^\mathscr{M}$ validates $\Gamma$ in the outer algebra, it also validates~$\alpha$. For the axiomatisation, we need to update it with the following (obviously sound) \emph{modal rules}, i.e. rules whose assumptions are inner formulas and whose conclusion is an outer formula (replacing now redundant monotonicity modal axioms):
\begin{align*}
   \f  &\lhd \belmod\f \\
   \{ \Phi &\lhd \belmod\f\rightarrow \belmod\psi\mid \Phi,\f\vdash_{\BD} \psi \}
\end{align*}
and similarly for the other modalities. The notion of a~two-layered proof needs to be adapted as well in an expected manner: upper parts of a~proof tree may consist of a~$\BD$ proof followed by an application of a~modal rule. The completeness proof becomes more subtle now as well: basically, to refute $\Phi,\Gamma\vdash\alpha$, we need to construct an appropriate measure on the canonical $\BD$ model for the finite set $\LIT$ of literals from $\Phi,\Gamma\vdash\alpha$, relativized to $\Phi$ (i.e., $(\P(\LIT),v^+_c,v^-_c)$, relativized to the set $|\bigwedge\Phi|^+$). This takes into account that we need to measure $\BD$ formulas up to equi-provability from $\Phi$\footnote{It is worth mentioning that the relation $\f\equiv_\Phi \psi$ given as $\Phi,\f\vdash_{\BD}\psi$ and $\Phi,\psi\vdash_{\BD}\f$ is in general not a~congruence w.r.t.\ $\neg$. The logic $\BD$ is only weakly Fregean.}. 
\end{remark}
\subsection{Translations}
Here, we follow~\cite{BaldiCintulaNoguera2020} and provide a~translation of the two-layered logics with inequalities given in section~\ref{ssec:2layerineq} into $\Prob^{\Luk^2}_\triangle$ and $\Bell^{\Luk^2}_\triangle$ and back. In fact, the original translations work just fine.
\subsubsection{Translations to $\Prob^{\Luk^2}_\triangle$}
\begin{definition}[Functional counterparts]\label{def:f-counterpart}
Let $l_1$, \ldots, $l_n$ be literals, and $\alpha(l_1,\ldots,l_n)$ be an $\Luk^2_\triangletop$ formula in NNF. A~\emph{functional counterpart} of $\alpha$ (denoted $f_\alpha$) is a~function $f_\alpha:\mathbb{R}^n\rightarrow\mathbb{R}$ s.t.\ for any $v$,
\begin{align*}
v_1(\alpha)&=f_\alpha(v_1(l_1),\ldots,v_1(l_n))\\
v_2(\alpha)&=f_\alpha(v_2(l_1),\ldots,v_2(l_n))
\end{align*}
\end{definition}
\begin{definition}[From weight formulas to \L{}ukasiewicz two-layer formulas]\label{def:weighttoLuk2}
Let $$\underbrace{\sum\limits_{i=1}^{n}a_i\cdot\mathtt{w}^+(\phi_i)}_{t}\geqslant c$$ be a~PWF that does not contain $\mathtt{w}^-$. Let, further, for any $f:\mathbb{R}^n\rightarrow\mathbb{R}$, $f^\sharp=\min(1,\max(f,0))$.

Now, for
$$f(x_1,\ldots,x_n)=\sum\limits_{i=1}^{n}a_i\cdot\mathtt{w}^\pm(\phi_i)-c+1$$
let $\beta(p_1,\ldots,p_n)$ be an $\Luk$-formula s.t.\ $f_\beta=f^\sharp$. Thus, we can define the $^\bullet$-translation as follows\footnote{Recall that we manipulate with weight formulas classically.}.
\begin{align*}
(t\geqslant c)^\bullet&=\triangletop\beta(\Prob\phi_1,\ldots,\Prob\phi_n)\\
\bot^\bullet&=\bot\\
\alpha\supset\alpha'&=\alpha^\bullet\rightarrow_\Luk\alpha'^\bullet
\end{align*}
\end{definition}
\begin{remark}
Observe that we can actually avoid using $\neg$ if we use $\sim$ or $\bot$ as a~primitive $\Luk^2$-connective. This will give us an embedding into $\mathsf{Pr_{lin}}$ from~\cite{BaldiCintulaNoguera2020}. Note, furthermore, that because of $\triangletop$, all translated formulas have either value $(1,0)$, or $(0,1)$. Again, we do not actually use two-dimensionality. Furthermore, it is evident that all $\Luk^2_{(x,y)}$ entailments are the same and that $\neg$ and $\sim$ are identical if all atomic propositions are prenexed by $\triangletop$.
\end{remark}
\begin{proposition}\label{prop:weighttoLuk2}
Let $\Xi\cup\{\alpha\}$ be a~set of weight formulas. Then, $\Xi\vDash_{\mathtt{m}\mathsf{BD}}\alpha$ iff $\Xi^\bullet\vDash_{\Luk^2}\alpha^\bullet$.
\end{proposition}
We define our translation on NNFs of $\Luk^2$ formulas. Note also that $\neg\triangletop\phi\leftrightarrow_\Luk{\sim}\triangletop\phi$. Thus, we can ignore the Belnap--Dunn negation.
\begin{definition}[From \L{}ukasiewicz two-layer formulas to weight formulas]\label{def:Luk2toweight}
Let $\alpha$ be an NNF of an $\Luk^2$ probabilistic formula and let $\alpha^-$ be the result of replacing each $\Prob\phi_i$ with a~$p_i$. Since $f_{\alpha^-}=\max\limits_{k\in K}\min\limits_{j\in J_k}t_{k,j}$ for some $K$ and $J$ where for every $k\in K$ and $j_k\in J$ there is a~linear function $f_{k,j}$ with integer coefficients and
\[t_{k,j}=f^\sharp_{k,j}\text{ or }t_{k,j}=1-\triangletop(1-f^\sharp_{f,j})\]
it is possible to define the $^\circ$-translation of $\alpha$ as follows.
\[\alpha^\circ=\bigvee\limits_{k\in K}\bigwedge\limits_{j\in J_k}\alpha_{k,j}\]
Here for $f_{k,j}=\sum\limits_{i=1}^{n}a_i\cdot x_i+c$, we have
\begin{align*}
\alpha_{k,j}&=
\begin{cases}
\sum\limits_{i=1}^{n}a_i\cdot\mathtt{w}^+(\phi_i)\geqslant1-c&\text{ if }t_{k,j}=f^\sharp_{k,j}\\
\sum\limits_{i=1}^{n}a_i\cdot\mathtt{w}^+(\phi_i)<-c&\text{ otherwise}
\end{cases}
\end{align*}
\end{definition}
\begin{proposition}\label{prop:Luk2toweight}
Let $\Xi\cup\{\alpha\}$ be a~set of $\Luk^2$ two-layer formulas. Then, $\Xi\!\vDash_{\Luk^2}\!\alpha$ iff $\Xi^\circ\!\vDash_{\mathtt{m}\mathsf{BD}}\!\alpha^\circ$.
\end{proposition}
\subsubsection{Translations to $\Bell^{\Luk^2}_\triangle$}
\begin{definition}[Language and semantics of $\langle\mathsf{BD},B,\LsquaretopDelta\rangle$ for beliefs]\label{def:bBDBLuk2Deltalanguage}
The language is as follows. Let $\phi\in\LBD$. The set of (upper-layer) formulas is defined via the following grammar.
\[\alpha\coloneqq \Bel\phi\mid\neg\alpha\mid\alpha\rightarrow_\Luk\alpha\mid\triangletop\alpha\]
$\bot$ is defined as $\neg(\beta\rightarrow_\Luk\beta)$. All other connectives --- $\sim$, $\wedge$, $\vee$, $\odot$, and $\oplus$ --- can be defined as usual.

The semantics is given over De Morgan algebras with belief functions: $v_1(\Bel\phi)=\bel(v(\phi))$; $v_2(\Bel\phi)=\bel(v(\neg\phi))$. The connectives are defined in the expected fashion.
\end{definition}
\begin{definition}[From belief formulas to \L{}ukasiewicz modal formulas]\label{def:belieftoLuk2}
Let $$\underbrace{\sum\limits_{i=1}^{n}a_i\cdot\mathtt{b}^+(\phi_i)}_{t}\geqslant c$$ be a~PBF that does not contain $\mathtt{b}^-$. Let, further, for any $f:\mathbb{R}^n\rightarrow\mathbb{R}$, $f^\sharp=\min(1,\max(f,0))$.

Now, for
$$f(x_1,\ldots,x_n)=\sum\limits_{i=1}^{n}a_i\cdot\mathtt{w}^\pm(\phi_i)-c+1$$
let $\beta(p_1,\ldots,p_n)$ be an $\Luk$-formula s.t.\ $f_\beta=f^\sharp$ (recall definition~\ref{def:f-counterpart}). Thus, we can define the $^\bullet$-translation as follows\footnote{Recall that we manipulate with weight formulas classically.}.
\begin{align*}
(t\geqslant c)^\bullet&=\triangletop\beta(\belmod\phi_1,\ldots,\belmod\phi_n)\\
\bot^\bullet&=\bot\\
\alpha\supset\alpha'&=\alpha^\bullet\rightarrow_\Luk\alpha'^\bullet
\end{align*}
\end{definition}
\begin{proposition}\label{prop:belieftoLuk2}
Let $\Xi\cup\{\alpha\}$ be a~set of belief formulas. Then, $\Xi\vDash_{\mathtt{b}\mathsf{BD}}\alpha$ iff $\Xi^\bullet\vDash_{\Luk^2}\alpha^\bullet$.
\end{proposition}
\begin{definition}[From \L{}ukasiewicz two-layer formulas to belief formulas]\label{def:Luk2tobelief}
Let $\alpha$ be an NNF of an $\Luk^2$ two-layer formula and let $\alpha^-$ be the result of replacing each $\Bel\phi_i$ with a~$p_i$. Since $f_{\alpha^-}=\max\limits_{k\in K}\min\limits_{j\in J_k}t_{k,j}$ for some $K$ and $J$ where for every $k\in K$ and $j_k\in J$ there is a~linear function $f_{k,j}$ with integer coefficients and
\[t_{k,j}=f^\sharp_{k,j}\text{ or }t_{k,j}=1-\triangletop(1-f^\sharp_{f,j})\]
it is possible to define the $^\circ$-translation of $\alpha$ as follows
\[\alpha^\circ=\bigvee\limits_{k\in K}\bigwedge\limits_{j\in J_k}\alpha_{k,j}\]
Here for $f_{k,j}=\sum\limits_{i=1}^{n}a_i\cdot x_i+c$, we have
\begin{align*}
\alpha_{k,j}&=
\begin{cases}
\sum\limits_{i=1}^{n}a_i\cdot\mathtt{b}^+(\phi_i)\geqslant1-c&\text{ if }t_{k,j}=f^\sharp_{k,j}\\
\sum\limits_{i=1}^{n}a_i\cdot\mathtt{b}^+(\phi_i)<-c&\text{ otherwise}
\end{cases}
\end{align*}
\end{definition}
\begin{proposition}\label{prop:Luk2tobelief}
Let $\Xi\cup\{\alpha\}$ be a~set of $\Luk^2$ two-layer formulas. Then, $\Xi\!\vDash_{\Luk^2}\!\alpha$ iff $\Xi^\circ\!\vDash_{\mathtt{b}\mathsf{BD}}\!\alpha^\circ$.
\end{proposition}

\section{Conclusion}
Let us recapitulate the main results of our paper. We developed the axiomatics of non-standard belief and plausibility functions (Section~\ref{sec:uncertainty}) and proved their completeness (Theorems~\ref{th:complBelAxioms} and~\ref{th:complPlAxioms}). We then proposed two different kinds of formalising reasoning with non-standard probabilities and belief functions (Section~\ref{sec:logics}). In both cases, we show how to reason with inconsistent and incomplete information in a~non-trivial manner. We established completeness of these logics: Theorems~\ref{th:FSSC_probLuk2} and~\ref{th:FSSC_belLuk2} give completeness of two-layered logics while Theorems~\ref{theorem:wFDEcompleteness} and~\ref{theorem:bFDEcompleteness} provide completeness of calculi of inequalities \emph{\`{a} la} Fagin, Halpern, and Megiddo~\cite{FaginHalpernMegiddo1990}. Finally, we defined faithful translations in the manner of~\cite{BaldiCintulaNoguera2020} between these two frameworks in Propositions~\ref{prop:weighttoLuk2},~\ref{prop:Luk2toweight},~\ref{prop:belieftoLuk2}, and~\ref{prop:Luk2tobelief}.

Still, several important questions remain unanswered. First, in the case of classical probability theory, belief and plausibility functions are not the only uncertainty measures more general than probabilities proper. In particular, probabilities can be generalised by lower and upper probabilities (cf., e.g.~\cite[\S2.3]{Halpern2017} for more details). It would be instructive to define their $\BD$ counterparts and investigate their properties. Second, Bayes' and Jeffrey's conditionings of non-standard probabilities are given in~\cite{KleinMajerRad2021}. It is reasonable to construct conditioning of non-standard belief and plausibility functions. Furthermore, to the best of our knowledge, there is no work done on the conditioning of non-standard probabilities on arbitrary lattices.

On the logic side, we plan to investigate  the proof theory of our two-layered logics. We are particularly interested in providing structural proof theory in the form of display or sequent calculi and in expanding on the tableaux framework presented in~\cite{BilkovaFrittellaKozhemiachenko2021TABLEAUX}. The tableaux will also allow us to establish the decidability and complexity bounds of our two-layered logics.
\bibliographystyle{plain}
\addcontentsline{toc}{section}{References}
\bibliography{references}

\newpage
\appendix
\section[Proofs of Section~2.3]{Proofs of Section~\ref{sssec:belief:plausibility:mass}}
\label{app:belief:plausibility:mass}

The following Lemma is used to prove Theorem~\ref{theo:totallymonotone:charactrisation}
\begin{lemma} \emph{\cite [Lemma 2.6]{Zhou2013}} \label{lemma:zhou:inc_exc}
Let $\mu:{\mathcal{P}}(X)\to \mathbb{R}$ be a~function such that:
$\mu(\varnothing)=0$ and $\mu(A\cup B)=\mu(A)+\mu(B)-\mu(A\cap B)$ for every $A,B\in\P(X)$. Then, we have:
\begin{equation}
    \label{eq:inclusion-exclusion}
    \mu\left( \bigcup\limits_{i=1}^{n}R_i \right)=\sum\limits_{J\subseteq \{1,\ldots,n\}, J\ne \varnothing} (-1)^{|J|+1} \mu\left( \bigcap\limits_{j\in J} R_i \right).
\end{equation}
\end{lemma}

\paragraph{Theorem~\ref{theo:totallymonotone:charactrisation}}
Let $\mathcal{L}$ be a~finite lattice and  $\bel:\mainL\rightarrow [0,1]$ be a~monotone function. Let further, $\mass_\bel : \mainL\rightarrow [0,1]$ be the M\"obius transform of $\bel$. Then, $\bel$ is a general belief function iff $\mass_\bel$ is a general mass function. 

If $\mathcal{L}$ is a~finite bounded lattice, then, $\bel$ is a  belief function iff $\mass_\bel$ is a mass function.

In addition, if $\bel$ is a (general) belief function, we have for every $x\in\mathcal{L}$, 
\begin{equation}
\bel(x)=\sum\limits_{y \leq x}\mass_\bel(y).
\end{equation}
We call $\mass_\bel$ the (general) mass function associated to $\bel$.

\begin{proof}
    The statement ``$\bel$ is a belief function iff $\mass_\bel$ is a  mass function'' 
    follows from \cite[Theorem~2.8]{Zhou2013} that states that $f$ is weakly totally monotone iff $g$ is non-negative. It is immediate to prove that $g$ is indeed a mass function.

    For the statement ``$\bel$ is a general belief function iff $\mass_\bel$ is a general mass function''.
 The proof is similar to the proof of 
Theorem~\ref{theo:totallymonotone:charactrisation}.
 The differences are the following: $\bot$ and $\top$ are not in the signature of the lattice $\mainL$, we do not require the lattice to be distributive, and $f$ is not required to be a~belief function, therefore $g$ is not necessarily a~mass function.
 
 Let $a \in \mainL$. First, we  show that $g(a)\geq 0$.
 Notice that, since $f(a)=\sum_{b\leq a}g(b)$, we have $g(a)=f(a)-\sum_{b<a}g(b)$. Let $A=\{x\in \mainL \mid x<a\}$. Recall that $A= \bigcup_{b<a}\downarrow b$, where $\downarrow b=\{x\in \mainL \mid x\leq b\}$.
 Let $\mu : \P(\mainL) \rightarrow [0,1]$ be such that $\mu(A)\coloneqq \sum_{x\in A}g(x)$. Notice that $\mu$ is additive. 
 The following chain of equalities holds: 
 \begin{align*}
 \sum_{b< a}g(b)& = \sum_{x\in A}g(x)\\
 & =\mu(A) \tag{definition of $\mu$}\\
 &=\sum\limits_{J\subseteq A, J\ne \varnothing} (-1)^{|J|+1} \mu\left( \bigcap_{b\in J}\downarrow b\right)\tag{$A= \bigcup_{b<a}\downarrow b$ and $\mu$ satisfies \eqref{eq:inclusion-exclusion}} \\
 &=\sum\limits_{J\subseteq A, J\ne \varnothing} (-1)^{|J|+1} \left( \sum_{x\in \bigcap_{b\in J}\downarrow b} g(x) \right)\tag{definition of $\mu$ } \\ 
 &=\sum\limits_{J\subseteq A, J\ne \varnothing} (-1)^{|J|+1} \left( \sum_{x\in \downarrow{\bigwedge \{x\mid x\in J\}}}g(x) \right)\tag{$\downarrow\! a~\cap \downarrow\! b=\downarrow\! (a\wedge b)$ } \\ 
 &=\sum\limits_{J\subseteq A, J\ne \varnothing} (-1)^{|J|+1} f\left(\bigwedge\limits_{x\in J}x\right)\tag{definition of $f$ and $\downarrow$}
 \end{align*}
Notice that, if there is no element smaller than $a$, we have $f(a)=g(a)$ and  $g(a)\geq 0$ (because $f$ is positive). Otherwise, the following chain of inequalities follows from the equality above and because $f$ is monotone and weakly totally monotone.
 \begin{equation*}
 f(a)\geq f\left(\bigvee_{b<a}b\right)\geq\sum\limits_{J\subseteq A, J\ne \varnothing} (-1)^{|J|+1} f\left(\bigwedge\limits_{x\in J}x\right)= \sum_{b<a}\mass(b).
\end{equation*}
 Therefore, $g(a)=f(a)-\sum_{b<a}g(b)\geq 0$ as required.
 
 Since $\mathcal{L}$ is a~finite lattice, it has a~unique maximal element $t$. Since $g$ is the M\"obius transform of $f$, we have $\sum_{x\in\mathcal{L}}g(x)=\sum_{x\leq t}g(x)=f(t)$. Therefore, since $g$ is non-negative,  $0\leq \sum_{x\in\mathcal{L}}g(x)\leq 1$ and $g : \mathcal{L}\rightarrow [0,1]$. that is, $g$ is a~general mass function as required.
 \end{proof}

\paragraph{Lemma~\ref{lem:pl:associated:mass}} 
Let $\mathcal{L}$ be a~De Morgan algebra and $\pl : \mathcal{L} \rightarrow [0,1]$ be a~general plausibility function.
Then, the function $\bel_\pl : \mathcal{L} \rightarrow [0,1]$ such that 
$\bel_\pl(x)=1-\pl(\neg x)$ is a~general belief function. We denote $\mass_\pl$ the mass function associated to $\bel_\pl$ and we call $\mass_\pl$ the mass function associated to $\pl$. Then
\begin{equation*}
 \pl(x)=1 - \sum_{y \leq \neg x} \mass_{{\pl}}(y).
\end{equation*} If $\pl$ is a~plausibility function, then $\bel_\pl$ is a~belief function.

\begin{proof}
Consider the function $\bel_\pl(x)=1-\pl(\neg x)$. Since $0\leq \pl(\neg x)\leq 1$, then $0\leq\bel_\pl(x)\leq 1$. Therefore, $\bel_\pl$ is well-defined.
Notice that, since $\neg$ is order-reversing and $\pl$ is order-preserving, $\bel_\pl$ is order-preserving. 
For every $a_1, \dots, a_k\in \mathcal{L}$, we have 
\begin{align*}
\mathtt{pl}\left(\neg \bigvee\limits_{1\leq i\leq k}a_i\right) &=
\mathtt{pl}\left(\bigwedge\limits_{1\leq i\leq k}\neg a_i\right)
\tag{$\neg$ is a~De Morgan negation}
\\
 &\leq
\sum\limits_{\scriptsize{\begin{matrix}J\subseteq\{1,\ldots,k\}\\J\neq\varnothing\end{matrix}}}\!\!\!\!\!(-1)^{|J|+1}\cdot\mathtt{pl}\left(\bigvee\limits_{j\in J}\neg a_j\right)
\tag{$\pl$ is a~general plausibility function}
\\
& =
\sum\limits_{\scriptsize{\begin{matrix}J\subseteq\{1,\ldots,k\}\\J\neq\varnothing\end{matrix}}}(-1)^{|J|+1}\cdot\mathtt{pl}\left(\neg\bigwedge\limits_{j\in J} a_j\right)
\tag{$\neg$ is a~De Morgan negation}
\end{align*}
In addition, we have\footnote{Recall that $\left(\begin{matrix} n \\ k \end{matrix} \right)$ denotes the binomial coefficient, that is, the number of ways to choose an (unordered) subset of $k$ elements from a~fixed set of $n$ elements.}
\begin{align*}
 \sum\limits_{\scriptsize{\begin{matrix}J\subseteq\{1,\ldots,n\}\\J\neq\varnothing\end{matrix}}}(-1)^{|J|+1} 
 & = (-1) \cdot \sum\limits_{\scriptsize{\begin{matrix}J\subseteq\{1,\ldots,n\}\\J\neq\varnothing\end{matrix}}}(-1)^{|J|} 
  = (-1) \cdot \sum_{1\leq k \leq n}(-1)^{k} \left(\begin{matrix} n \\ k \end{matrix} \right)
 \\
 & = (-1) \cdot \left( \sum_{0\leq k \leq n}(-1)^{k} \left(\begin{matrix} n \\ k \end{matrix} \right) -1 \right)
 \tag{since $\sum_{0\leq k \leq n}(-1)^{k} \left(\begin{matrix} n \\ k \end{matrix} \right) =0$}
 \\
 & = 1
\end{align*}
Therefore,
\begin{align*}
\bel_\pl\left( \bigvee\limits_{1\leq i\leq k}a_i\right)
 &=
1 - \mathtt{pl}\left(\neg \bigvee\limits_{1\leq i\leq k}a_i\right)
\\
& \geq
\left(\sum\limits_{\scriptsize{\begin{matrix}J\subseteq\{1,\ldots,n\}\\J\neq\varnothing\end{matrix}}}(-1)^{|J|+1}\right) - \left(\sum\limits_{\scriptsize{\begin{matrix}J\subseteq\{1,\ldots,k\}\\J\neq\varnothing\end{matrix}}}(-1)^{|J|+1}\cdot\mathtt{pl}\left(\neg\bigwedge\limits_{j\in J} a_j\right)\right)
\\
& = 
\sum\limits_{\scriptsize{\begin{matrix}J\subseteq\{1,\ldots,n\}\\J\neq\varnothing\end{matrix}}}
\left((-1)^{|J|+1} - (-1)^{|J|+1}\cdot\mathtt{pl}\left(\neg\bigwedge\limits_{j\in J} a_j\right) \right)
\\
& = 
\sum\limits_{\scriptsize{\begin{matrix}J\subseteq\{1,\ldots,n\}\\J\neq\varnothing\end{matrix}}}
(-1)^{|J|+1} \cdot \left(1 - \mathtt{pl}\left(\neg\bigwedge\limits_{j\in J} a_j\right) \right)
\\
& = 
\sum\limits_{\scriptsize{\begin{matrix}J\subseteq\{1,\ldots,n\}\\J\neq\varnothing\end{matrix}}}
(-1)^{|J|+1} \cdot \bel_\pl\left(\bigwedge\limits_{j\in J} a_j \right).
\end{align*}
Therefore $\bel_\pl$ is order-preserving and $k$-monotone for every $k \geq 1$. Hence, $\bel_\pl$ is a~general belief function. In addition, if $\pl$ is a~plausibility function, then $\bel(\bot)=1 - \pl (\top)=0$ and $\bel(\top)=1 - \pl(\bot)=1$. Therefore, $\bel$ is a~belief function. Let $\mass_\pl$ be the mass function associated to $\bel_\pl$, then we have $\pl(x) = 1 - \bel(\neg x) = 1 - \sum_{y \leq \neg x} \mass_\pl(y)$.
\end{proof}

\paragraph{Lemma
\ref{lem:bel:pl:1-bel}}
Let $\mathcal{L}$ be a~De Morgan algebra and $\bel : \mathcal{L} \rightarrow [0,1]$ be a~general belief function.
Then, the function $\pl : \mathcal{L} \rightarrow [0,1]$ such that 
$\pl(x)=1-\bel(\neg x)$ is a~general plausibility function. 
 If $\bel$ is a~belief function, then $\pl$ is a~plausibility function.

\begin{proof}
The proof is similar to the proof of lemma~\ref{lem:pl:associated:mass}. We only detail the proof that $\pl$ satisfies equation \eqref{eq:pl:k:inequality} for every $k \geq 1$. Let $a_1, \dots, a_k\in \mathcal{L}$.
Recall that 
\begin{align*}
\mathtt{bel}\left(\bigvee\limits_{1\leq i\leq n}a_i\right)\geq\sum\limits_{\scriptsize{\begin{matrix}J\subseteq\{1,\ldots,n\}\\J\neq\varnothing\end{matrix}}}(-1)^{|J|+1}\cdot \mathtt{bel}\left(\bigwedge\limits_{j\in J}a_j\right).
\end{align*}
Therefore, we have the following chain of inequalities.
\begin{align*}
\mathtt{bel}\left(\bigvee\limits_{1\leq i\leq n}\neg a_i\right)&\geq
\sum\limits_{\scriptsize{\begin{matrix}J\subseteq\{1,\ldots,n\}\\J\neq\varnothing\end{matrix}}}(-1)^{|J|+1}\cdot \mathtt{bel}\left(\bigwedge\limits_{j\in J}\neg a_j\right)
\\
-\mathtt{bel}\left(\neg \bigwedge\limits_{1\leq i\leq n} a_i\right)
&\leq - \sum\limits_{\scriptsize{\begin{matrix}J\subseteq\{1,\ldots,n\}\\J\neq\varnothing\end{matrix}}}(-1)^{|J|+1}\cdot \mathtt{bel}\left(\neg \bigvee\limits_{j\in J} a_j\right).
\end{align*}
Since, $\neg$ is a~De Morgan negation and multiplication by $(-1)$ reverses the inequality, we have
\begin{align*}
-\mathtt{bel}\left(\neg \bigwedge\limits_{1\leq i\leq n} a_i\right)
&\leq \sum\limits_{\scriptsize{\begin{matrix}J\subseteq\{1,\ldots,n\}\\J\neq\varnothing\end{matrix}}}(-1)^{|J|+1}\cdot \left(- \mathtt{bel}\left(\neg \bigvee\limits_{j\in J} a_j\right)\right)
\\
\mathtt{pl}\left( \bigwedge\limits_{1\leq i\leq n} a_i\right) -1
&\leq \sum\limits_{\scriptsize{\begin{matrix}J\subseteq\{1,\ldots,n\}\\J\neq\varnothing\end{matrix}}}(-1)^{|J|+1}
\left(\mathtt{pl}\left( \bigvee\limits_{j\in J} a_j\right) -1
\right)
\tag{$-\bel(\neg x)=\pl(x)-1$}
\\
\mathtt{pl}\left( \bigwedge\limits_{1\leq i\leq n} a_i\right) -1
&\leq 
\left(
\sum\limits_{\scriptsize{\begin{matrix}J\subseteq\{1,\ldots,n\}\\J\neq\varnothing\end{matrix}}}(-1)^{|J|+1}\cdot
\mathtt{pl} 
\left(\bigvee\limits_{j\in J} a_j\right)\right) - 
\left(
\sum\limits_{\scriptsize{\begin{matrix}J\subseteq\{1,\ldots,n\}\\J\neq\varnothing\end{matrix}}}(-1)^{|J|+1}\right)
\\
\mathtt{pl}\left( \bigwedge\limits_{1\leq i\leq n} a_i\right) -1
&\leq 
\left(
\sum\limits_{\scriptsize{\begin{matrix}J\subseteq\{1,\ldots,n\}\\J\neq\varnothing\end{matrix}}}(-1)^{|J|+1}\cdot
\mathtt{pl} 
\left(\bigvee\limits_{j\in J} a_j\right)\right) + 
\left(
\sum\limits_{\scriptsize{\begin{matrix}J\subseteq\{1,\ldots,n\}\\J\neq\varnothing\end{matrix}}}(-1)^{|J|}\right)
\\
\mathtt{pl}\left( \bigwedge\limits_{1\leq i\leq n} a_i\right) -1
&\leq 
\left(
\sum\limits_{\scriptsize{\begin{matrix}J\subseteq\{1,\ldots,n\}\\J\neq\varnothing\end{matrix}}}(-1)^{|J|+1}\cdot \mathtt{pl} 
\left(\bigvee\limits_{j\in J} a_j\right)\right) + 
\left(
\sum_{1 \leq k \leq n}(-1)^{k} \left(\begin{matrix} n \\ k \end{matrix} \right)\right).
\end{align*}
Since, because $
\sum_{0\leq k \leq n}(-1)^{k} \left(\begin{matrix} n \\ k \end{matrix} \right) =0 = 1 + \sum_{1\leq k \leq n}(-1)^{k} \left(\begin{matrix} n \\ k \end{matrix} \right)$, we get
\begin{align*}
\mathtt{pl}\left( \bigwedge\limits_{1\leq i\leq n} a_i\right) -1
&\leq 
\left(
\sum\limits_{\scriptsize{\begin{matrix}J\subseteq\{1,\ldots,n\}\\J\neq\varnothing\end{matrix}}}(-1)^{|J|+1}\cdot \mathtt{pl} 
\left(\bigvee\limits_{j\in J} a_j\right)\right) -1
\\
\mathtt{pl}\left(\bigwedge\limits_{1\leq i\leq k}a_i\right)
& \leq
\sum\limits_{\scriptsize{\begin{matrix}J\subseteq\{1,\ldots,k\}\\J\neq\varnothing\end{matrix}}}(-1)^{|J|+1}\cdot\mathtt{pl}\left(\bigvee\limits_{j\in J}a_j\right),
\end{align*}
as required.
\end{proof}

\section[Proofs of Section~3]{Proofs of Section~\ref{sec:uncertainty}}
\label{app:proof:sec3}
\paragraph{Lemma~\ref{lem:correspondece:lattice:language}}  There is a~ one-one correspondence between the functions on $\LBD$ satisfying   the properties (i)--(iii) of Definition~\ref{DEF:NSprob} and the functions on the Lindenbaum algebra $\LatBD$ with the same properties. 
\begin{proof}
First, we need to show that if $g$ is a~function on $\LBD$ satisfying the properties 
\begin{enumerate}
\item $ 0\leq g(\f) \leq 1$ (normalisation)
\item if $\varphi \vdash_{\BD} \psi$, then $g(\varphi) \leq g(\psi)$ (monotonicity),
\item $g(\f\lor \p) = g(\f) + g(\p) - g(\varphi\wedge \p)$ (inclusion/exclusion),
\end{enumerate}
then there is a~corresponding function $g'$ on the Lindenbaum algebra $\LTBD$ with the same properties.  

Let $g: \LBD \rightarrow \mathbb{R}$. We define $g': \LTBD \rightarrow \mathbb{R}$ as follows: $g'([\f]):=g(\f)$ where $[\f]$ represents the equivalence class. Notice that the monotonicity of $g$ implies that $g(\f)=g(\psi)$ for any two equivalent formulas $\f$ and $\psi$. Therefore, $g'$ is well-defined. $g'$
satisfies the three properties above because: (i) $0\leq g'([\f])\leq 1$ since $0\leq g(\f)\leq 1$; (ii) if $[\f]\leq [\psi]$, then $\f\vdash_{\BD} \psi$, so $g(\f)\leq g(\psi)$, hence $g'([\f])\leq g'([\psi])$; (iii) we have 
 \begin{align*}
g'([\f]\vee[\psi])&=g'([\f\vee\psi])\\
&=g(\f\vee\psi)\tag{by definition of $g'$}\\
&=g(\f)+g(\psi)-g(\f\wedge \psi)\tag{by assumption} \\
&= g'([\f])+g'([\psi])-g'([\f\wedge\psi])\tag{by definition of $g'$} \\
&= g'([\f])+g'([\psi])-g'([\f]\wedge[\psi]) 
 \end{align*}
 
The proof of the converse is similar.
\end{proof}

\paragraph{Lemma~\ref{lem:bel+:pl+}}
Let $\mathscr{M}=
(S,\P(S), \bel, \pl, v^+, v^-)$ be a~$\DS_\pl$ model. $\bel^+$ (resp., $\pl^+$) is a~general belief (resp., plausibility) function on the Lindenbaum algebra. $\bel^-$ (resp., $\pl^-$) is a~general belief (resp., plausibility) function on the dual of the Lindenbaum algebra $\LTBD^{op}$.
\begin{proof}
$\bel$ and $| \cdot|^+$ are monotone maps, therefore $\bel^+$ is monotone.
In addition, for every $\varphi_1, \dots, \varphi_n \in \LTBD$, we have
\begin{align*}
 \bel^+\left( \bigvee\limits_{1\leq i\leq k}\varphi_i \right)
 & = \bel\left(\left| \bigvee\limits_{1\leq i\leq k}\varphi_i\right|^+ \right)\\
& = \bel\left(\bigcup\limits_{1\leq i\leq k} | \varphi_i|^+ \right)
 \\
 & \geq\sum\limits_{\scriptsize{\begin{matrix}J\subseteq\{1,\ldots,k\}\\J\neq\varnothing\end{matrix}}}(-1)^{|J|+1}\cdot \bel\left(\bigcap\limits_{j\in J} |\varphi_j|^+\right)
 \\
 & =\sum\limits_{\scriptsize{\begin{matrix}J\subseteq\{1,\ldots,k\}\\J\neq\varnothing\end{matrix}}}(-1)^{|J|+1}\cdot \bel\left(\left|\bigwedge\limits_{j\in J}\varphi_j\right|^+\right)
 \\
 & =\sum\limits_{\scriptsize{\begin{matrix}J\subseteq\{1,\ldots,k\}\\J\neq\varnothing\end{matrix}}}(-1)^{|J|+1}\cdot \bel^+\left(\bigwedge\limits_{j\in J}\varphi_j\right).
\end{align*}
The proof for $\pl^+$ is similar. The proofs for $\bel^-$ and $\pl^-$ are dual.
\end{proof}
\end{document}